\documentclass[reqno,11pt]{amsart}
\usepackage{amsmath,amssymb,amsthm,graphicx,mathrsfs,url}
\usepackage[usenames,dvipsnames]{color}
\usepackage[colorlinks=true,linkcolor=Red,citecolor=Green]{hyperref}
\usepackage[all]{xy}

\usepackage{verbatim}
\usepackage{latexsym}
\usepackage{eucal}

 \usepackage[applemac]{inputenc}
\usepackage[cyr]{aeguill}

\usepackage{a4wide}

\def\?[#1]{\textbf{[#1]}\marginpar{\Large{\textbf{??}}}}

\newtheorem{theo}{Theorem}
\newtheorem{prop}{Proposition}[section]
\newtheorem{lem}[prop]{Lemma}
\newtheorem{cor}[prop]{Corollary}

\theoremstyle{definition}
\newtheorem{defi}[prop]{Definition}
\newtheorem{rem}[prop]{Remark}
\numberwithin{equation}{section}
\newtheorem*{Acknowledgement}{Acknowledgements}

\DeclareMathOperator{\supp}{supp}
\DeclareMathOperator{\Vol}{dvol}

\DeclareMathOperator{\Tr}{Tr}

\newcommand{\mc}{\mathcal}
\newcommand{\rr}{\mathbb{R}}
\newcommand{\nn}{\mathbb{N}}
\newcommand{\cc}{\mathbb{C}}
\newcommand{\hh}{\mathbb{H}}
\newcommand{\zz}{\mathbb{Z}}

\newcommand{\bX}{\bbar{X}}
\newcommand{\la}{\lambda}
\newcommand{\eps}{\varepsilon}

\newcommand{\pl}{\partial}
\newcommand{\x}{\times}

\newcommand{\til}{\widetilde}
\newcommand{\bbar}{\overline}

\newcommand{\cjd}{\rangle}
\newcommand{\cjg}{\langle}

\newcommand{\demi}{\frac{1}{2}}

\newcommand{\VolR}{\textrm{Vol}_R}

\newcommand{\indic}{\operatorname{1\hskip-2.75pt\relax l}}

\newcommand{\CI}{\mathcal{C}^{\infty}}

\newcommand{\pa}{\partial}
\newcommand{\cC}{\mathcal{C}}
\newcommand{\tvarphi}{\til{\varphi}}

\renewcommand\Re{\operatorname{Re}}

\newcommand\res{\operatorname{res}}
\newcommand\tU{\widetilde{U}}
\newcommand\hU{\widehat{U}}

\newcommand\cf{\operatorname{cf}}
\newcommand\Scal{{\rm Scal}}
\newcommand{\II}{\mathrm{I\hspace{-0.04cm}I}}

\begin{document}

\title[Renormalized volume of punctured surfaces]{Renormalized volume on the Teichm\"uller space of punctured surfaces}
\author{Colin Guillarmou}
\email{cguillar@dma.ens.fr}
\address{DMA, U.M.R. 8553 CNRS, \'Ecole Normale Sup\'erieure, 45 rue d'Ulm,
75230 Paris cedex 05, France}
\author{Sergiu Moroianu}
\email{moroianu@alum.mit.edu}
\address{Institutul de Matematic\u{a} al Academiei Rom\^{a}ne\\
P.O. Box 1-764\\RO-014700
Bu\-cha\-rest, Romania}
\author{Fr\'ed\'eric Rochon}
\email{rochon.frederic@uqam.ca}
\address{D\'epartement de mat\'ematiques, UQ\`AM}
\date{\today}
\begin{abstract}
We define and study the renormalized volume for geometrically finite hyperbolic 
$3$-manifolds, including with rank-$1$ cusps. We prove a variation formula, and show that 
for certain families of convex co-compact hyperbolic metrics $g_\eps$ degenerating 
to a geometrically finite hyperbolic metric $g_0$ with rank-$1$ cusps, 
the renormalized volume converges to the renormalized volume of the limiting metric. 
\end{abstract}
\maketitle

\section{Introduction}

The renormalized volume is a geometric quantity for certain infinite volume hyperbolic 
$3$-dimensional manifolds, namely those which are convex co-compact. Such a manifold $X$ 
can be compactified into a smooth compact manifold with boundary $\bbar{X}$ in a way 
that its metric $g$ has the following property: for any smooth function 
$\rho\in \mc{C}^\infty(\bbar{X})$ which is a boundary defining function 
(i.e., $\rho\ge 0$, $\rho^{-1}(0)=\pl\bbar{X}$ and $d\rho|_{\pl\bbar{X}}$ does not vanish), 
$\rho^2g$ extends to a smooth metric on $\bbar{X}$. This induces a natural 
conformal class on the boundary $M:=\pl\bbar{X}$ by picking the conformal class 
$[h]$ of $h=(\rho^2g)|_{TM}$.  We call $(M,[h])$ the \emph{conformal boundary} of $X$. 
We say that a boundary defining function $\rho$ is a 
\emph{geodesic boundary defining function} in $\bbar{X}$ if $|d\log(\rho)|_g=1$ 
near the boundary $M$.  Notice that such a function induces an \emph{equidistant foliation} 
near $M$, given by the level sets of $\rho$. It turns out that there is a one-to-one 
correspondence $\hat{h}=e^{2\varphi}h\in [h]\mapsto \hat{\rho}$ between geodesic 
boundary defining functions (or equivalently, equidistant foliations) near $M$ and 
the elements of the conformal class $[h]$ on $M$, where $\hat{\rho}$ solves the 
Hamilton-Jacobi equation near $M$
\begin{equation}\label{HamJacintro} 
\Big|\frac{d\hat{\rho}}{\hat{\rho}}\Big|_g=1, \quad (\hat{\rho}^2g)|_{TM}=\hat{h}.
\end{equation} 

The renormalized volume of $(X,g)$ is the  function on $[h]$ defined by 
\[ {\rm Vol}_R(X,g;\hat{h}):={\rm FP}_{z=0} \int_{X}\hat{\rho}^z{\rm dvol}_g\]
where $\hat{\rho}$ is any smooth positive extension to $X$ of the function solving 
\eqref{HamJacintro} and ${\rm FP}_{z=0}$ denotes the finite part (or regular value) at $z=0$ 
of a meromorphic function in the variable $z\in \cc$. 
In a way, this definition has similarities with the renormalization used to define the determinant of the Laplacian on a compact manifold. In fact, the functional $\varphi \mapsto {\rm Vol}_R(X,g;e^{2\varphi}h)$ varies in the same exact way as do the Liouville functional and the logarithm of the determinant of the Laplacian viewed as functionals on $[h]$. Among metrics in the conformal class $[h]$ of constant volume $2\pi|\chi(M)|$, it is maximized at the hyperbolic metric $h^{\rm hyp}\in[h]$, and we define the \emph{renormalized volume of $(X,g)$} by 
\[{\rm Vol}_{R}(X,g):={\rm Vol}_R(X,g; h^{\rm hyp}).\]
We remark that the renormalized volume could equivalently be defined  by  ${\rm Vol}_R(X,g;\hat{h})=a_0$, where $a_0$ is defined by the asymptotic expansion (for some $a_j\in\rr$) as $\epsilon\to 0$
\[ \int_{\hat{\rho}\geq \epsilon}{\rm dvol}_g=a_2\epsilon^{-2}+a_1\log(\epsilon)+a_0+\mc{O}(\epsilon).\]

In this setting, the first general study  was done by Krasnov-Schlenker \cite{KrSc}, although earlier works of Takhtajan-Teo \cite{TaTe} considered this quantity, and for more general Poincar\'e-Einstein manifolds the renormalized volume appeared even earlier in works of Henningson-Skenderis \cite{HeSk} and 
Graham \cite{Gr} in AdS/CFT correspondence. 

When defined in this way, the renormalized volume has many interesting properties:
\begin{itemize} 
\item The renormalized volume is a K\"ahler potential for the Weil-Peterson metric on the Teichm\"uller space of the 
conformal boundary, when viewed as a function on the deformation space of convex co-compact hyperbolic $3$-manifolds. This was proved by Takhtajan-Teo \cite{TaTe} for a class of Kleinian convex co-compact groups, by Krasnov-Schlenker \cite{KrSc} for quasi-Fuchsian manifolds and by Guillarmou-Moroianu \cite{GuMo} for all geometrically finite hyperbolic $3$-manifolds without cusps of rank $1$.
\item ${\rm Vol}_R(X,g)$ can be compared to the volume of the convex core ${\rm Vol}(C(X))$ by 
\[ {\rm Vol}(C(X))-10\chi(M) \leq {\rm Vol}_R(X,g)\leq {\rm Vol}(C(X)).\]
This inequality is proved by Schlenker \cite{Sc} for quasi-Fuchsian manifolds, 
and extended by Bridgeman-Canary \cite{BrCa} to convex co-compact $3$-manifolds with incompressible boundary.
\item Schlenker \cite{Sc} proves that for quasi-Fuchsian manifolds, 
${\rm Vol}_R(X,g)$ is comparable to the Weil-Petersson distance between the two connected components 
$(M,h_\pm)$ of the conformal boundary. Namely he shows that
\begin{equation}\label{compVolWP}
{\rm Vol}_R(X,g)\leq \tfrac{3}{2}\sqrt{2\pi\chi(X)} d_{\rm WP}(h_+,h_-),
 \end{equation}
improving a weaker inequality due to Brock \cite{Br}. Moreover, using \cite{Br}, Schlenker obtains that there exists 
some $k_1,k_2>0$ such that
\[ k_1 d_{\rm WP}(h_+,h_-)-k_2\leq {\rm Vol}_R(X,g).\]
These inequalities have interesting implications about the geometry of hyperbolic $3$-manifolds 
fibering over the circle, cf.\ \cite{KoMc}, \cite{BrBr} .
\item Ciobotaru-Moroianu \cite{CiMo} prove that for almost-Fuchsian manifolds, 
the renormalized volume is positive except at the Fuchsian locus where it 
vanishes\footnote{The normalization to make it $0$ at the Fuchsian locus is actually 
to choose the metric in the conformal boundary to have 
Gaussian curvature $-4$. The same normalization is used in \cite{KrSc}}. 
\item Moroianu \cite{Mo} proves that the renormalized volume has a critical point on the deformation 
space of convex co-compact $3$-manifolds  if the convex core has smooth totally geodesic boundary, 
and the Hessian of ${\rm Vol}_R$ is positive definite there. Another proof appeared recently in \cite{Va}.
\end{itemize}

Like in the estimate \eqref{compVolWP}, it is of interest to understand the properties of ${\rm Vol}_R$ on the deformation space of convex co-compact hyperbolic $3$-manifolds with a given topology. For example, \eqref{compVolWP} shows that ${\rm Vol}_R$ does not explode as one 
approaches the boundary of the Teichm\"uller space viewed as a Bers slice in the quasi-Fuchsian space. 

The first goal of this work is to define the renormalized volume for geometrically finite 
hyperbolic $3$-manifolds, focusing on the rank-$1$ cusps. Contrary to the convex co-compact 
setting, the existence of equidistant foliations via geodesic boundary defining 
functions turns out to be 
quite tricky in the case of rank-$1$ cusps. A geometrically finite hyperbolic manifold
$(X,g)=\Gamma \backslash \hh^3$ with rank-$1$ cusps is the interior of a smooth non-compact manifold $\bbar{X}=\Gamma\backslash (\hh^3\cup\Omega_\Gamma)$ with boundary, where $\Omega_{\Gamma}\subset \mathbb{S}^2$ 
is the discontinuity set of the Kleinian group $\Gamma\subset {\rm PSL}_2(\cc)$. 
The smooth manifold with boundary $\bbar{X}$ has a non-compact boundary $M=\Gamma\backslash \Omega_{\Gamma}$ equipped with a conformal class $[h]$ induced from the hyperbolic metric $g$.
On this conformal boundary $(M,[h])$, we show in Proposition \ref{uniform} that there exists a unique complete hyperbolic metric $h^{\rm hyp}\in [h]$ with finite volume and cusps.
\begin{theo}\label{thintro:1}
Let $(X,g)$ be a geometrically finite hyperbolic $3$-manifold with rank-$1$ cusps, let $(M,[h])$ be its conformal boundary and let $h^{\rm hyp}$ be the complete hyperbolic metric with finite volume in the conformal class $[h]$.
Then there exists a non-negative smooth boundary defining function $\rho$ on 
$\bbar{X}$ such that $\rho^2g|_{TM}=h^{\rm hyp}$ and, outside a finite volume region $\mc{V}\subset X$, 
$|d\log(\rho)|_g=1$. The function $z\to \int_{X\setminus \mc{V}}\rho^z{\rm dvol}_g$ 
admits a meromorphic extension from ${\rm Re}(z)>2$ to a neighborhood of $z=0$.
\end{theo}

We define the renormalized volume by  
\[ {\rm Vol}_R(X,g):={\rm Vol}_g(\mc{V})+{\rm FP}_{z=0}\int_{X\setminus \mc{V}}\rho^z{\rm dvol}_g.\]

In fact, in Proposition \ref{deffunccusp}, we show a stronger statement: we prove that 
for each conformal representative in $[h]$ with certain asymptotic properties near the cusp, 
there is an associated geodesic boundary defining function and an equidistant foliation, 
allowing to view ${\rm Vol}_R$ as a function on $[h]$ like in the convex co-compact case. 
In Proposition \ref{vrv.1}, we show a variation formula similar to that of 
the determinant of the Laplacian \cite[Theorem~2.9]{AAR} or the Liouville functional:
\[{\rm Vol}_R(X, g; e^{2\varphi}h^{\rm hyp})= {\rm Vol}_R(X,g; h^{\rm hyp}) -\frac14 \int_{M} (|\nabla\varphi|^2_{h^{\rm hyp}}-2\varphi) \rm dvol_{h^{\rm hyp}}.\]

There is a diffeomorphism $\psi:[0,\eps)_x\x M\to \bbar{X}\setminus \mc{V}$ such that 
$\psi^*\rho=x$, and the metric has a finite expansion in powers of $x$:
\begin{equation}\label{devofg}
\psi^*g= \frac{dx^2+h_0+x^2h_2+x^4h_4}{x^2}\end{equation}
where the coefficients $h_0$, $h_2$ and $h_4$ are symmetric tensors on $M$ and such that $h_0=h^{\rm hyp}$, 
$h_2^0:=h_2-\demi h^{\rm hyp}$ is trace-free and divergence-free with respect to $h^{\rm hyp}$, 
and $h_4=\frac{1}{4}h_0(A^2\cdot,\cdot)$ if $A$ is the endomorphism defined by 
$h_2=h_0(A\cdot,\cdot)$. The tensor $h_2^0$ can thus be identified to a cotangent vector to the Teichm\"uller space 
$\mc{T}(M)$ of $M$ at the metric $h_0=h^{\rm hyp}$, and the pair $(h_0,h^0_2)\in T^*\mc{T}(M)$ characterizes uniquely $g$. 
We call $h_2^0$ the \emph{second fundamental form of $g$ at $M$}.
\begin{theo}\label{thintro:2}
For $t\in(-1,1),$ let $(X,g^t)$ be a smooth family of geometrically finite hyperbolic metrics with cusps of rank $1$ and 
let $h^{t}$ be the unique finite volume hyperbolic representative of the conformal boundary of  
$(X,g^t)$. Then, 
$$
  \left. \pa_t {\rm Vol}_R(X,g^t) \right|_{t=0} = -\frac14 \int_{M} \langle \dot{h}, h_2^0\rangle_{h} 
  {\rm dvol}_{h},
$$
with $\dot{h}=\pl_{t}h^{t}|_{t=0}$, $h=h^{t}|_{t=0}$,  and $h_2^0$ is the second fundamental form of $g=g^t|_{t=0}$ at $M$.
\end{theo}
The equivalent result was shown by Krasnov-Schlenker \cite{KrSc} 
(see also \cite{GuMo} for another proof) in the convex co-compact setting. 
This implies, using a theorem of Marden \cite{Ma2}, that the deformation space 
of a geometrically finite hyperbolic $3$-manifold $(X,g)$ with rank-$1$ cusps 
can be viewed as a Lagrangian submanifold $\mc{H}\subset T^*\mc{T}(M)$, 
the graph of the exact $1$-form on $\mc{T}(M)$ given by the exterior differential of
the renormalized volume functional ${\rm Vol}_R(X,\cdot):\mc{T}(M)\to \rr$. Equivalently,
the restriction to $\mc{H}$ of the Liouville $1$-form on $T^*\mc{T}(M)$ is exact, 
and a primitive is given by ${\rm Vol}_R(X,\cdot )$ if we identify $\mc{H}$ with $\mc{T}(M)$
by the canonical projection; we refer to \cite{GuMo} for details 
in the convex co-compact setting.

Ahlfors-Bers simultaneous uniformisation theorem shows that if $(M,h_-)$ and $(M,h_+)$ are two hyperbolic surfaces of finite volume with $n$ cusps, there exists a unique (up to diffeomorphism) complete 
hyperbolic metric $g_{(h_-,h_+)}$ on the cylinder $X:=\rr_t\x M$, which is realized as a quotient $\Gamma\backslash\hh^3$ for some quasi-Fuchsian group $\Gamma\subset {\rm PSL}_2(\cc)$ and the obtained 
manifold is geometrically finite with cusps of rank-$1$. 
The quasi-Fuchsian space is the deformation space of such quasi-Fuchsian groups and it identifies to $\mc{T}(M)\x \mc{T}(M)$ where $\mc{T}(M)$ is the Teichm\"uller space of $M$.
Fixing $h_-$, the map $h_+\mapsto g_{(h_-,h_+)}$ provides an embedding of $\mc{T}(M)$ into the quasi-Fuchsian deformation space  
and we view the renormalized volume as a function on $\mc{T}(M)$: $h_+\mapsto {\rm Vol}_R(X,g_{(h_-,h_+)})$.
We extend to the case with punctures the result proved in the convex co-compact case by Takhtajan-Teo \cite{TaTe}, and later by Krasnov-Schlenker \cite{KrSc}, Guillarmou-Moroianu \cite{GuMo}.
\begin{theo}\label{Kahlerpot}
Set $h_-=h_0\in \mc{T}(M)$, the map $V_{h_0}: h_+\mapsto {\rm Vol}_R(X,g_{(h_0,h_+)})$ is a K\"ahler potential for Weil-Petersson metric on $\mc{T}(M)$, more precisely $\bbar{\pl}\pl  V_{h_0}=\frac{i}{16}\omega_{\rm WP}$ where $\omega_{\rm WP}$ is the Weil-Petersson symplectic form.
\end{theo}
Notice that the same result was proved independently by Park-Takhtajan-Teo \cite{PTT}. 

Our last result consists in analyzing the renormalized volume of families of convex co-compact hyperbolic $3$-manifolds degenerating to a geometrically finite manifold with rank-$1$ cusps. We define precisely an \emph{admissible degeneration of convex co-compact metrics} in Definition \ref{sec:assumptions}, but essentially such a family of metrics $(g_\eps)_{\eps>0}$ on $X$ corresponds to having a disjoint union $H=\cup_{j=1}^{j_1}H_j$ of $j_1$ simple curves in $M=\pl\bbar{X}$ such that
\begin{enumerate}
\item outside a uniform neighborhood $\mc{U}$ of $H$, $\rho^2g_\eps$ converges smoothly to a metric on $\bbar{X}\setminus \mc{U}$ if $\rho$ is a fixed boundary defining function of $\pl\bbar{X}$;
\item in $\mc{U}$ near $H_j$, the metric $g_\eps$ is isometric to a certain region of $\cjg \gamma^\eps_j\cjd\backslash \hh^3$ where $\gamma_j^\eps\in {\rm PSL}_2(\cc)$ is a loxodromic element converging as $\eps\to 0$ to a parabolic element $\gamma_j$ in such a way that $\alpha_j(\eps)/\ell_j(\eps)$ converges , where $\ell_j(\eps)$ and $\alpha_j(\eps)$ are respectively  the translation length and the holonomy angle of $\gamma_j^\eps$ (i.e., $\gamma_j^\eps$ is conjugated to $z\mapsto e^{\ell(\eps)+i\alpha_j(\eps)}z$).
\end{enumerate}

Our last theorem is 
\begin{theo}
Assume $g_\eps$ is an admissible degeneration of convex co-compact hyperbolic metrics on $X$, in the sense of Definition \ref{sec:assumptions}, to a geometrically finite hyperbolic metric $g_0$ with rank-$1$ cusps on $X$. Then 
\[ \lim_{\eps\to 0}{\rm Vol}(X,g_\eps)={\rm Vol}_R(X,g_0).\]
\end{theo}
We show in Proposition \ref{SchottkyOK} that such admissible degenerations happen 
for instance on the boundary of the classical Schottky space. In \cite[Theorem 1.3]{BrCa}, Bridgeman and Canary show that in other asymptotic regime (when the radius of injectivity of the hyperbolic metric in the domain of discontinuity is going to $0$), the renormalized volume tends to $-\infty$.

\begin{figure}\label{fig:foliation1}
\centering
\def\svgwidth{15em}
\scalebox{1.05}{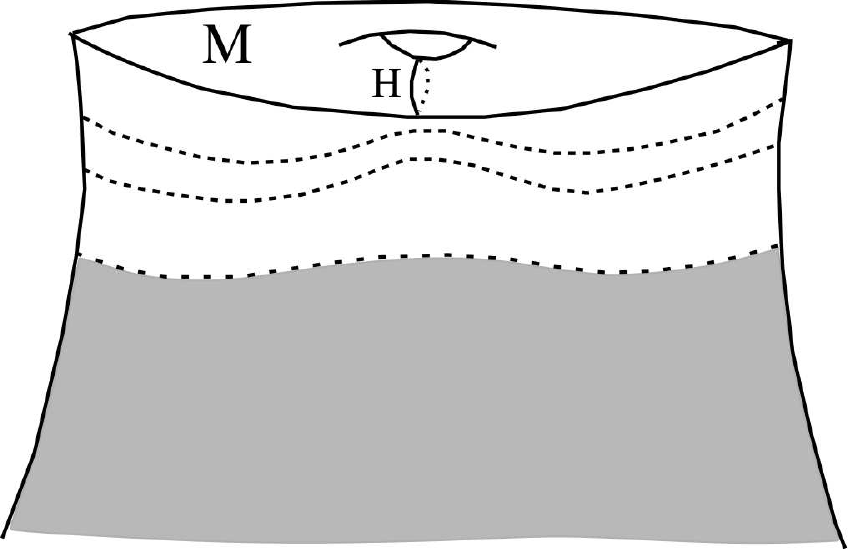}\quad \quad 
\scalebox{0.55}{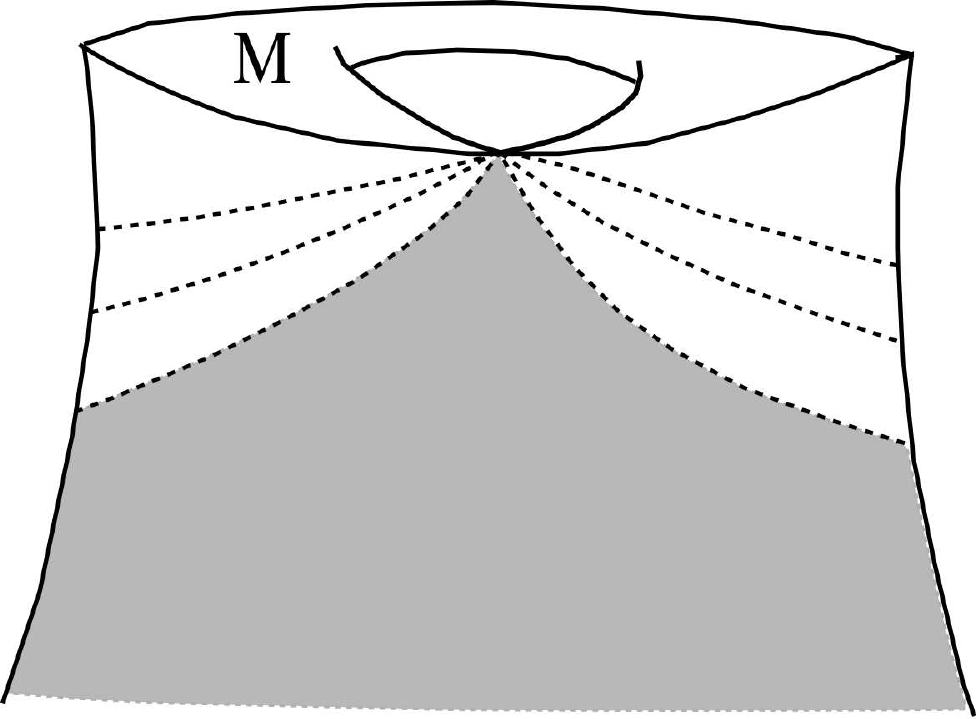}
\caption{We consider a case when the curve $H$ is being pinched in the boundary $M$. 
The equidistant foliation is represented by the dotted lines. 
The first picture corresponds to the convex co-compact case and the second picture is 
the hyperbolic $3$-manifold with a rank-$1$ cusp. The dark regions are the convex cores.} 
\end{figure}

\begin{Acknowledgement}
F.~R.\ was partially supported by NSERC, FRQNT and a Canada research chair. 
C.~G.\ and F.~R.\ were partially supported by grant ANR-13-BS01-0007-01. 
S.~M.\ was partially supported by the CNCS project PN-II-RU-TE-2012-3-0492.
We thank R.\ Canary and J.-M.\ Schlenker for helpful discussions.
\end{Acknowledgement}

\section{Renormalized volume for geometrically finite hyperbolic $3$-manifolds} \label{rv.1}

\subsection{Geometrically finite hyperbolic $3$-manifolds}\label{geofinite}
In this Section, we recall the geometry of geometrically finite hyperbolic manifolds 
of dimension $3$. For more details, we refer to the paper of Bowditch \cite{Bo}, 
see also Mazzeo-Phillips \cite{MaPh} or Guillarmou-Mazzeo \cite{GuMa}.
A manifold $X$ of dimension $3$ is said to be geometrically finite hyperbolic if it can be realized as a quotient
$X=\Gamma\backslash \hh^{3}$ by a Kleinian group $\Gamma\subset {\rm PSL}_{2}(\cc)\simeq {\rm PSO}(3,1)$, 
so that its action on $\hh^3$ has a fundamental domain with finitely many side. In higher dimension, 
this definition is not very natural and the correct one is given by Bowditch, however we shall restrict here to the $3$-dimensional case. 
If we view $\hh^3$ as the open unit ball in $\rr^3$, it can be naturally compactified into the closed unit ball 
$\bbar{\hh^3}=\hh^3\cup \mathbb{S}^2$, and elements of ${\rm PSL}_2(\cc)$ acts on $\bbar{\hh^3}$.
We say that $X$ has cusps if $\Gamma$ contains parabolic elements in ${\rm PSL}_2(\cc)$,  i.e. elements which fix only one 
point in the closed unit ball $\bbar{\hh^3}$. If for each point $p\in \mathbb{S}^2$ fixed by a parabolic transformation $\gamma_p\in \Gamma$, the subgroup 
$\Gamma_p\subset \Gamma$ fixing $p$ is the cyclic group generated by the element $\gamma_p$, then we say that $X$ has only \emph{cusps of rank $1$}, and we will make this assumption for what follows.\footnote{cusps of rank $2$ are trivial to deal with for what concerns renormalized volume questions, since they generate ends with finite volume in $X$.}

 We view $\hh^3$ as the unit ball in $\rr^3$. We can add to $X$ a conformal boundary by defining 
\[\bbar{X}:=\Gamma\backslash (\hh^3\cup \Omega)\]
where $\Omega\subset S^2$ is the domain of discontinuity of the group $\Gamma$, ie. the complement in $S^2$ of the limit set $\Lambda_\Gamma$ consisiting of accumulation points in the closed unit ball of the orbit of any given point $m\in\hh^3$. The manifold $\bbar{X}$ is a smooth manifold with boundary and its 
boundary 
\[M:=\Gamma\backslash\Omega=\pl\bbar{X}\] 
is a union of smooth Riemann surfaces, which has cusps if and only
$\Gamma$ has rank-$1$ cusps. It inherits a conformal class  which is defined to be 
the conformal class of $\rho^2g|_{TM}$ where $g$ is the hyperbolic metric on $X$ and $\rho$ is 
any smooth boundary defining function in $\bbar{X}$ (ie. $\rho\ge 0$, $M=\{\rho=0\}$ and $d\rho|_{M}$ never vanishes on 
$M$). Note that $\bbar{X}$ is not compact if $\Gamma$ has cusps.

The important geometric fact that we shall use is the following:   
there exists a compact set $\mc{K}\subset \bbar{X}$ such that $\bbar{X}\setminus \mc{K}=\cup_{j=1}^{j_1}\mc{U}^c_j$ where $\mc{U}^c_j$ are disjoint open sets of $\bbar{X}$, called \emph{cusp neighbourhoods}, so that $g$ on $\mc{U}^c_j\cap X$ is isometric through a map $\iota_j$ to 
\begin{equation}\label{Uj}
 \begin{gathered}
 \{(z=y+ix,w)\in \hh^2 \x (\rr/\tfrac{1}{2}\zz);\, |z|>R_j\},\,\, \\
  \textrm{ with metric } g=\frac{dx^2+dy^2+dw^2}{x^2} 
 \end{gathered}
 \end{equation}  
for some $R_j>0$; here $\hh^2=\{z\in \cc; {\rm Im}(z)>0\}$ is viewed as the upper half plane. 
We shall therefore identify $\mc{U}_j^c$ with the region in \eqref{Uj}.
Here $j_0$ is the number of rank-$1$ cusps.
The compact $\mc{K}$ in $\bbar{X}$ decomposes further into $\mc{K}=\mc{K}_0\cup \mc{U}^r$ where $\mc{K}_0$ is compact in $X$ and $\mc{U}^r$ is a compact set of $\bbar{X}$ such that  the hyperbolic metric $g$ in the interior of  $\mc{U}^r$ near $M$ is of the 
form $g=\bbar{g}/\rho^2$ where $\rho$ is a smooth boundary defining function of $M$ 
and $\bbar{g}$ is a smooth metric on $\mc{K}$.   
The boundary $M$ is a non compact Riemann surface with $2j_0$ cusps, and $M$ equipped with the conformal class $[\rho^2g|_{TM}]$ is called the \emph{conformal boundary} of $X$. 
Notice that, using an inversion $(v+iu)=-1/(y+ix)$ in the $\hh^2$ factor of \eqref{Uj}, 
the neighborhood $\mc{U}_j^c\cap X$ with metric $g$ is also isometric to 
\begin{equation}\label{Uj'} 
\begin{gathered}
 \{(z=v+iu,w)\in \hh^2\x (\rr/\tfrac{1}{2}\zz);\, |z|<R_j^{-1}\},\,\,\\  \textrm{ with metric } 
 g=\frac{du^2+dv^2+(u^2+v^2)^2dw^2}{u^2}. 
\end{gathered}
\end{equation}
Using this model for $\mc{U}_j^c$, we see that we can compactify $\bbar{X}$ into a compact smooth manifold with boundary, 
denoted $\bbar{{\bf X}}$, by compactifying the open set \eqref{Uj'} to 
\begin{equation}\label{Uj'compact}  
\{(z=v+iu,w)\in \bbar{\hh^2}\x (\rr/\tfrac{1}{2}\zz);\, |z|<R_j^{-1}\}
\end{equation} 
if $\bbar{\hh^2}$ is the closed upper half-plane of $\cc$, and with the smooth structure given by 
the smooth structure on $\bbar{\hh^2}\x (\rr/\tfrac{1}{2}\zz)$. This compactification amounts to adding 
a circle at each cusp of $\bbar{X}$, and clearly the interior of $\bbar{{\bf X}}$ is $X$ and $\bbar{X}$ is an open set in $\bbar{{\bf X}}$. 
We denote by $H_j$ each of these circles defined by $\{u=v=0\}$ in \eqref{Uj'compact}, 
and let $H:=\cup_{j=1}^{j_1}H_j$. 

There is another natural compactification of $\bbar{X}$ (and $X$) that arises, which corresponds to the 
real blow-up of $H$ in $\pl\bbar{{\bf X}}$ in $\bbar{{\bf X}}$, which we denote by $\bbar{X}_c$. 
This is a smooth manifold with corners of codimensions $2$ defined as follows: 
by taking the representation \eqref{Uj'} of  $\mc{U}_j^c$, we see 
that this has closure  in $\bbar{X}$ diffeomorphic to 
\[ \{ (u,v,w)\in \rr_+\x \rr\x (\rr/\tfrac{1}{2}\zz); u^2+v^2<R_j^{-2}\}\]
and to define $\bbar{X}_c$, we replace this chart by the chart
\[ \{ (r,\theta,w)\in [0,R_j^{-1})\x [-\tfrac{\pi}{2},\tfrac{\pi}2]\x (\rr/\tfrac{1}{2}\zz)\}\]
where $r:=\sqrt{u^2+v^2}$ and $\theta:=\arctan \frac{v}{u}$. This corresponds to a real blow-up 
(denoted $\bbar{X}_c=[\bbar{{\bf X}};H]$ in \cite[Chap. 5]{Me}) of the 
submanifold 
$\{(u,v,w)\in \bbar{\mc{U}_j^c}; u=v=0\}$, which is a circle. 
In this way, the manifold with corners $\bbar{X}_c$ has two boundary hypersurfaces. 
One, given by $\theta=\pm\frac{\pi}2$, is denoted $\bbar M$ and is a compactification 
of $M$ to a smooth surface with boundary,  while the other, the \emph{cusp face}, 
denoted $\cf$, is given by $r=0$ and is diffeomorphic to a cylinder
$[-\frac{\pi}2,\frac{\pi}2]_{\theta}\times (\rr/\tfrac12 \zz)_{w}$ if there is only one cusp 
of rank 1.  More generally,  the connected components of $\cf$ are 
in one-to-one correspondence with the cusps of rank 1 of $X$ with each connected component 
diffeomorphic to $[-\frac{\pi}2,\frac{\pi}2]_{\theta}\times (\rr/\frac12 \zz)_{w}$. 
We will use this extended space $\bbar{X}_c$ to describe the 
analytic structure of the geodesic boundary defining function of $M$ in $\bbar{X}$ 
near the cusps, which allows us to define the renormalized volume in that setting. 
To summarize, we have the following manifolds and inclusions  
\[ \pl\bbar{X}=M\subset \bbar{M},\quad  X \subset \bbar{X} \subset \bbar{{\bf X}} , \quad X \subset \bbar{X} \subset \bbar{X}_c=[\bbar{{\bf X}};H].\]

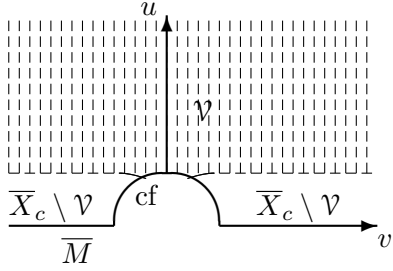
\begin{figure}
\setlength{\unitlength}{0.7cm}
\begin{picture}(7,5)(0,-0.5)
\thicklines
\put(3,0){\oval(2,2)[t]}
\put(2.4,0.4){$\cf$}

\put(2.5,4){$u$}

\put(4,0){\vector(1,0){3}}
\put(0,0){\line(1,0){2}}
\put(7,-0.4){$v$}
\put(1,-0.7){$\bbar M$}

\put(3,1){\vector(0,1){3}}

\put(0,0.2){$\bbar{X}_c\setminus\mathcal{V}$}
\put(4.7,0.2){$\bbar{X}_c\setminus\mathcal{V}$}
\put(3.5,2){$\mathcal{V}$}

\thinlines
\put(0,1){\line(1,0){0.2}}
\put(0.3,1){\line(1,0){0.2}}
\put(0.6,1){\line(1,0){0.2}}
\put(0.9,1){\line(1,0){0.2}}
\put(1.2,1){\line(1,0){0.2}}
\put(1.5,1){\line(1,0){0.2}}
\put(1.8,1){\line(1,0){0.2}}
\qbezier(2.1,1)(2.3,1)(2.6,0.9)
\qbezier(3.9,1)(3.7,1)(3.4,0.9)
\put(4.0,1){\line(1,0){0.2}}
\put(4.3,1){\line(1,0){0.2}}
\put(4.6,1){\line(1,0){0.2}}
\put(4.9,1){\line(1,0){0.2}}
\put(5.2,1){\line(1,0){0.2}}
\put(5.5,1){\line(1,0){0.2}}
\put(5.8,1){\line(1,0){0.2}}
\put(6.1,1){\line(1,0){0.2}}
\put(6.4,1){\line(1,0){0.2}}
\put(6.7,1){\line(1,0){0.2}}

\put(0,1){\line(0,1){0.2}}
\put(0,1.3){\line(0,1){0.2}}
\put(0,1.6){\line(0,1){0.2}}
\put(0,1.9){\line(0,1){0.2}}
\put(0,2.2){\line(0,1){0.2}}
\put(0,2.5){\line(0,1){0.2}}
\put(0,2.8){\line(0,1){0.2}}
\put(0,3.1){\line(0,1){0.2}}
\put(0,3.4){\line(0,1){0.2}}
\put(0,3.7){\line(0,1){0.2}}

\put(0.2,1){\line(0,1){0.2}}
\put(0.2,1.3){\line(0,1){0.2}}
\put(0.2,1.6){\line(0,1){0.2}}
\put(0.2,1.9){\line(0,1){0.2}}
\put(0.2,2.2){\line(0,1){0.2}}
\put(0.2,2.5){\line(0,1){0.2}}
\put(0.2,2.8){\line(0,1){0.2}}
\put(0.2,3.1){\line(0,1){0.2}}
\put(0.2,3.4){\line(0,1){0.2}}
\put(0.2,3.7){\line(0,1){0.2}}

\put(0.4,1){\line(0,1){0.2}}
\put(0.4,1.3){\line(0,1){0.2}}
\put(0.4,1.6){\line(0,1){0.2}}
\put(0.4,1.9){\line(0,1){0.2}}
\put(0.4,2.2){\line(0,1){0.2}}
\put(0.4,2.5){\line(0,1){0.2}}
\put(0.4,2.8){\line(0,1){0.2}}
\put(0.4,3.1){\line(0,1){0.2}}
\put(0.4,3.4){\line(0,1){0.2}}
\put(0.4,3.7){\line(0,1){0.2}}

\put(0.6,1){\line(0,1){0.2}}
\put(0.6,1.3){\line(0,1){0.2}}
\put(0.6,1.6){\line(0,1){0.2}}
\put(0.6,1.9){\line(0,1){0.2}}
\put(0.6,2.2){\line(0,1){0.2}}
\put(0.6,2.5){\line(0,1){0.2}}
\put(0.6,2.8){\line(0,1){0.2}}
\put(0.6,3.1){\line(0,1){0.2}}
\put(0.6,3.4){\line(0,1){0.2}}
\put(0.6,3.7){\line(0,1){0.2}}

\put(0.8,1){\line(0,1){0.2}}
\put(0.8,1.3){\line(0,1){0.2}}
\put(0.8,1.6){\line(0,1){0.2}}
\put(0.8,1.9){\line(0,1){0.2}}
\put(0.8,2.2){\line(0,1){0.2}}
\put(0.8,2.5){\line(0,1){0.2}}
\put(0.8,2.8){\line(0,1){0.2}}
\put(0.8,3.1){\line(0,1){0.2}}
\put(0.8,3.4){\line(0,1){0.2}}
\put(0.8,3.7){\line(0,1){0.2}}

\put(1,1){\line(0,1){0.2}}
\put(1,1.3){\line(0,1){0.2}}
\put(1,1.6){\line(0,1){0.2}}
\put(1,1.9){\line(0,1){0.2}}
\put(1,2.2){\line(0,1){0.2}}
\put(1,2.5){\line(0,1){0.2}}
\put(1,2.8){\line(0,1){0.2}}
\put(1,3.1){\line(0,1){0.2}}
\put(1,3.4){\line(0,1){0.2}}
\put(1,3.7){\line(0,1){0.2}}

\put(1.2,1){\line(0,1){0.2}}
\put(1.2,1.3){\line(0,1){0.2}}
\put(1.2,1.6){\line(0,1){0.2}}
\put(1.2,1.9){\line(0,1){0.2}}
\put(1.2,2.2){\line(0,1){0.2}}
\put(1.2,2.5){\line(0,1){0.2}}
\put(1.2,2.8){\line(0,1){0.2}}
\put(1.2,3.1){\line(0,1){0.2}}
\put(1.2,3.4){\line(0,1){0.2}}
\put(1.2,3.7){\line(0,1){0.2}}

\put(1.4,1){\line(0,1){0.2}}
\put(1.4,1.3){\line(0,1){0.2}}
\put(1.4,1.6){\line(0,1){0.2}}
\put(1.4,1.9){\line(0,1){0.2}}
\put(1.4,2.2){\line(0,1){0.2}}
\put(1.4,2.5){\line(0,1){0.2}}
\put(1.4,2.8){\line(0,1){0.2}}
\put(1.4,3.1){\line(0,1){0.2}}
\put(1.4,3.4){\line(0,1){0.2}}
\put(1.4,3.7){\line(0,1){0.2}}

\put(1.6,1){\line(0,1){0.2}}
\put(1.6,1.3){\line(0,1){0.2}}
\put(1.6,1.6){\line(0,1){0.2}}
\put(1.6,1.9){\line(0,1){0.2}}
\put(1.6,2.2){\line(0,1){0.2}}
\put(1.6,2.5){\line(0,1){0.2}}
\put(1.6,2.8){\line(0,1){0.2}}
\put(1.6,3.1){\line(0,1){0.2}}
\put(1.6,3.4){\line(0,1){0.2}}
\put(1.6,3.7){\line(0,1){0.2}}

\put(1.8,1){\line(0,1){0.2}}
\put(1.8,1.3){\line(0,1){0.2}}
\put(1.8,1.6){\line(0,1){0.2}}
\put(1.8,1.9){\line(0,1){0.2}}
\put(1.8,2.2){\line(0,1){0.2}}
\put(1.8,2.5){\line(0,1){0.2}}
\put(1.8,2.8){\line(0,1){0.2}}
\put(1.8,3.1){\line(0,1){0.2}}
\put(1.8,3.4){\line(0,1){0.2}}
\put(1.8,3.7){\line(0,1){0.2}}

\put(2,1){\line(0,1){0.2}}
\put(2,1.3){\line(0,1){0.2}}
\put(2,1.6){\line(0,1){0.2}}
\put(2,1.9){\line(0,1){0.2}}
\put(2,2.2){\line(0,1){0.2}}
\put(2,2.5){\line(0,1){0.2}}
\put(2,2.8){\line(0,1){0.2}}
\put(2,3.1){\line(0,1){0.2}}
\put(2,3.4){\line(0,1){0.2}}
\put(2,3.7){\line(0,1){0.2}}

\put(2.2,1){\line(0,1){0.2}}
\put(2.2,1.3){\line(0,1){0.2}}
\put(2.2,1.6){\line(0,1){0.2}}
\put(2.2,1.9){\line(0,1){0.2}}
\put(2.2,2.2){\line(0,1){0.2}}
\put(2.2,2.5){\line(0,1){0.2}}
\put(2.2,2.8){\line(0,1){0.2}}
\put(2.2,3.1){\line(0,1){0.2}}
\put(2.2,3.4){\line(0,1){0.2}}
\put(2.2,3.7){\line(0,1){0.2}}

\put(2.4,1){\line(0,1){0.2}}
\put(2.4,1.3){\line(0,1){0.2}}
\put(2.4,1.6){\line(0,1){0.2}}
\put(2.4,1.9){\line(0,1){0.2}}
\put(2.4,2.2){\line(0,1){0.2}}
\put(2.4,2.5){\line(0,1){0.2}}
\put(2.4,2.8){\line(0,1){0.2}}
\put(2.4,3.1){\line(0,1){0.2}}
\put(2.4,3.4){\line(0,1){0.2}}
\put(2.4,3.7){\line(0,1){0.2}}

\put(2.6,1){\line(0,1){0.2}}
\put(2.6,1.3){\line(0,1){0.2}}
\put(2.6,1.6){\line(0,1){0.2}}
\put(2.6,1.9){\line(0,1){0.2}}
\put(2.6,2.2){\line(0,1){0.2}}
\put(2.6,2.5){\line(0,1){0.2}}
\put(2.6,2.8){\line(0,1){0.2}}
\put(2.6,3.1){\line(0,1){0.2}}
\put(2.6,3.4){\line(0,1){0.2}}
\put(2.6,3.7){\line(0,1){0.2}}

\put(2.8,1){\line(0,1){0.2}}
\put(2.8,1.3){\line(0,1){0.2}}
\put(2.8,1.6){\line(0,1){0.2}}
\put(2.8,1.9){\line(0,1){0.2}}
\put(2.8,2.2){\line(0,1){0.2}}
\put(2.8,2.5){\line(0,1){0.2}}
\put(2.8,2.8){\line(0,1){0.2}}
\put(2.8,3.1){\line(0,1){0.2}}
\put(2.8,3.4){\line(0,1){0.2}}
\put(2.8,3.7){\line(0,1){0.2}}

\put(3,1){\line(0,1){0.2}}
\put(3,1.3){\line(0,1){0.2}}
\put(3,1.6){\line(0,1){0.2}}
\put(3,1.9){\line(0,1){0.2}}
\put(3,2.2){\line(0,1){0.2}}
\put(3,2.5){\line(0,1){0.2}}
\put(3,2.8){\line(0,1){0.2}}
\put(3,3.1){\line(0,1){0.2}}
\put(3,3.4){\line(0,1){0.2}}
\put(3,3.7){\line(0,1){0.2}}

\put(3.2,1){\line(0,1){0.2}}
\put(3.2,1.3){\line(0,1){0.2}}
\put(3.2,1.6){\line(0,1){0.2}}
\put(3.2,1.9){\line(0,1){0.2}}
\put(3.2,2.2){\line(0,1){0.2}}
\put(3.2,2.5){\line(0,1){0.2}}
\put(3.2,2.8){\line(0,1){0.2}}
\put(3.2,3.1){\line(0,1){0.2}}
\put(3.2,3.4){\line(0,1){0.2}}
\put(3.2,3.7){\line(0,1){0.2}}

\put(3.4,1){\line(0,1){0.2}}
\put(3.4,1.3){\line(0,1){0.2}}
\put(3.4,1.6){\line(0,1){0.2}}
\put(3.4,1.9){\line(0,1){0.2}}
\put(3.4,2.2){\line(0,1){0.2}}
\put(3.4,2.5){\line(0,1){0.2}}
\put(3.4,2.8){\line(0,1){0.2}}
\put(3.4,3.1){\line(0,1){0.2}}
\put(3.4,3.4){\line(0,1){0.2}}
\put(3.4,3.7){\line(0,1){0.2}}

\put(3.6,1){\line(0,1){0.2}}
\put(3.6,1.3){\line(0,1){0.2}}
\put(3.6,1.6){\line(0,1){0.2}}
\put(3.6,1.9){\line(0,1){0.2}}
\put(3.6,2.2){\line(0,1){0.2}}
\put(3.6,2.5){\line(0,1){0.2}}
\put(3.6,2.8){\line(0,1){0.2}}
\put(3.6,3.1){\line(0,1){0.2}}
\put(3.6,3.4){\line(0,1){0.2}}
\put(3.6,3.7){\line(0,1){0.2}}

\put(3.8,1){\line(0,1){0.2}}
\put(3.8,1.3){\line(0,1){0.2}}
\put(3.8,1.6){\line(0,1){0.2}}
\put(3.8,1.9){\line(0,1){0.2}}
\put(3.8,2.2){\line(0,1){0.2}}
\put(3.8,2.5){\line(0,1){0.2}}
\put(3.8,2.8){\line(0,1){0.2}}
\put(3.8,3.1){\line(0,1){0.2}}
\put(3.8,3.4){\line(0,1){0.2}}
\put(3.8,3.7){\line(0,1){0.2}}

\put(4,1){\line(0,1){0.2}}
\put(4,1.3){\line(0,1){0.2}}
\put(4,1.6){\line(0,1){0.2}}
\put(4,1.9){\line(0,1){0.2}}
\put(4,2.2){\line(0,1){0.2}}
\put(4,2.5){\line(0,1){0.2}}
\put(4,2.8){\line(0,1){0.2}}
\put(4,3.1){\line(0,1){0.2}}
\put(4,3.4){\line(0,1){0.2}}
\put(4,3.7){\line(0,1){0.2}}

\put(4.2,1){\line(0,1){0.2}}
\put(4.2,1.3){\line(0,1){0.2}}
\put(4.2,1.6){\line(0,1){0.2}}
\put(4.2,1.9){\line(0,1){0.2}}
\put(4.2,2.2){\line(0,1){0.2}}
\put(4.2,2.5){\line(0,1){0.2}}
\put(4.2,2.8){\line(0,1){0.2}}
\put(4.2,3.1){\line(0,1){0.2}}
\put(4.2,3.4){\line(0,1){0.2}}
\put(4.2,3.7){\line(0,1){0.2}}

\put(4.4,1){\line(0,1){0.2}}
\put(4.4,1.3){\line(0,1){0.2}}
\put(4.4,1.6){\line(0,1){0.2}}
\put(4.4,1.9){\line(0,1){0.2}}
\put(4.4,2.2){\line(0,1){0.2}}
\put(4.4,2.5){\line(0,1){0.2}}
\put(4.4,2.8){\line(0,1){0.2}}
\put(4.4,3.1){\line(0,1){0.2}}
\put(4.4,3.4){\line(0,1){0.2}}
\put(4.4,3.7){\line(0,1){0.2}}

\put(4.6,1){\line(0,1){0.2}}
\put(4.6,1.3){\line(0,1){0.2}}
\put(4.6,1.6){\line(0,1){0.2}}
\put(4.6,1.9){\line(0,1){0.2}}
\put(4.6,2.2){\line(0,1){0.2}}
\put(4.6,2.5){\line(0,1){0.2}}
\put(4.6,2.8){\line(0,1){0.2}}
\put(4.6,3.1){\line(0,1){0.2}}
\put(4.6,3.4){\line(0,1){0.2}}
\put(4.6,3.7){\line(0,1){0.2}}

\put(4.8,1){\line(0,1){0.2}}
\put(4.8,1.3){\line(0,1){0.2}}
\put(4.8,1.6){\line(0,1){0.2}}
\put(4.8,1.9){\line(0,1){0.2}}
\put(4.8,2.2){\line(0,1){0.2}}
\put(4.8,2.5){\line(0,1){0.2}}
\put(4.8,2.8){\line(0,1){0.2}}
\put(4.8,3.1){\line(0,1){0.2}}
\put(4.8,3.4){\line(0,1){0.2}}
\put(4.8,3.7){\line(0,1){0.2}}

\put(5,1){\line(0,1){0.2}}
\put(5,1.3){\line(0,1){0.2}}
\put(5,1.6){\line(0,1){0.2}}
\put(5,1.9){\line(0,1){0.2}}
\put(5,2.2){\line(0,1){0.2}}
\put(5,2.5){\line(0,1){0.2}}
\put(5,2.8){\line(0,1){0.2}}
\put(5,3.1){\line(0,1){0.2}}
\put(5,3.4){\line(0,1){0.2}}
\put(5,3.7){\line(0,1){0.2}}

\put(5.2,1){\line(0,1){0.2}}
\put(5.2,1.3){\line(0,1){0.2}}
\put(5.2,1.6){\line(0,1){0.2}}
\put(5.2,1.9){\line(0,1){0.2}}
\put(5.2,2.2){\line(0,1){0.2}}
\put(5.2,2.5){\line(0,1){0.2}}
\put(5.2,2.8){\line(0,1){0.2}}
\put(5.2,3.1){\line(0,1){0.2}}
\put(5.2,3.4){\line(0,1){0.2}}
\put(5.2,3.7){\line(0,1){0.2}}

\put(5.4,1){\line(0,1){0.2}}
\put(5.4,1.3){\line(0,1){0.2}}
\put(5.4,1.6){\line(0,1){0.2}}
\put(5.4,1.9){\line(0,1){0.2}}
\put(5.4,2.2){\line(0,1){0.2}}
\put(5.4,2.5){\line(0,1){0.2}}
\put(5.4,2.8){\line(0,1){0.2}}
\put(5.4,3.1){\line(0,1){0.2}}
\put(5.4,3.4){\line(0,1){0.2}}
\put(5.4,3.7){\line(0,1){0.2}}

\put(5.6,1){\line(0,1){0.2}}
\put(5.6,1.3){\line(0,1){0.2}}
\put(5.6,1.6){\line(0,1){0.2}}
\put(5.6,1.9){\line(0,1){0.2}}
\put(5.6,2.2){\line(0,1){0.2}}
\put(5.6,2.5){\line(0,1){0.2}}
\put(5.6,2.8){\line(0,1){0.2}}
\put(5.6,3.1){\line(0,1){0.2}}
\put(5.6,3.4){\line(0,1){0.2}}
\put(5.6,3.7){\line(0,1){0.2}}

\put(5.8,1){\line(0,1){0.2}}
\put(5.8,1.3){\line(0,1){0.2}}
\put(5.8,1.6){\line(0,1){0.2}}
\put(5.8,1.9){\line(0,1){0.2}}
\put(5.8,2.2){\line(0,1){0.2}}
\put(5.8,2.5){\line(0,1){0.2}}
\put(5.8,2.8){\line(0,1){0.2}}
\put(5.8,3.1){\line(0,1){0.2}}
\put(5.8,3.4){\line(0,1){0.2}}
\put(5.8,3.7){\line(0,1){0.2}}

\put(6,1){\line(0,1){0.2}}
\put(6,1.3){\line(0,1){0.2}}
\put(6,1.6){\line(0,1){0.2}}
\put(6,1.9){\line(0,1){0.2}}
\put(6,2.2){\line(0,1){0.2}}
\put(6,2.5){\line(0,1){0.2}}
\put(6,2.8){\line(0,1){0.2}}
\put(6,3.1){\line(0,1){0.2}}
\put(6,3.4){\line(0,1){0.2}}
\put(6,3.7){\line(0,1){0.2}}

\put(6.2,1){\line(0,1){0.2}}
\put(6.2,1.3){\line(0,1){0.2}}
\put(6.2,1.6){\line(0,1){0.2}}
\put(6.2,1.9){\line(0,1){0.2}}
\put(6.2,2.2){\line(0,1){0.2}}
\put(6.2,2.5){\line(0,1){0.2}}
\put(6.2,2.8){\line(0,1){0.2}}
\put(6.2,3.1){\line(0,1){0.2}}
\put(6.2,3.4){\line(0,1){0.2}}
\put(6.2,3.7){\line(0,1){0.2}}

\put(6.4,1){\line(0,1){0.2}}
\put(6.4,1.3){\line(0,1){0.2}}
\put(6.4,1.6){\line(0,1){0.2}}
\put(6.4,1.9){\line(0,1){0.2}}
\put(6.4,2.2){\line(0,1){0.2}}
\put(6.4,2.5){\line(0,1){0.2}}
\put(6.4,2.8){\line(0,1){0.2}}
\put(6.4,3.1){\line(0,1){0.2}}
\put(6.4,3.4){\line(0,1){0.2}}
\put(6.4,3.7){\line(0,1){0.2}}

\put(6.6,1){\line(0,1){0.2}}
\put(6.6,1.3){\line(0,1){0.2}}
\put(6.6,1.6){\line(0,1){0.2}}
\put(6.6,1.9){\line(0,1){0.2}}
\put(6.6,2.2){\line(0,1){0.2}}
\put(6.6,2.5){\line(0,1){0.2}}
\put(6.6,2.8){\line(0,1){0.2}}
\put(6.6,3.1){\line(0,1){0.2}}
\put(6.6,3.4){\line(0,1){0.2}}
\put(6.6,3.7){\line(0,1){0.2}}

\put(6.8,1){\line(0,1){0.2}}
\put(6.8,1.3){\line(0,1){0.2}}
\put(6.8,1.6){\line(0,1){0.2}}
\put(6.8,1.9){\line(0,1){0.2}}
\put(6.8,2.2){\line(0,1){0.2}}
\put(6.8,2.5){\line(0,1){0.2}}
\put(6.8,2.8){\line(0,1){0.2}}
\put(6.8,3.1){\line(0,1){0.2}}
\put(6.8,3.4){\line(0,1){0.2}}
\put(6.8,3.7){\line(0,1){0.2}}

\end{picture}
\caption{The manifold with corners $\bbar{X}_c$ (the circle variable $w$ is not represented). The region $\mc{V}$
has finite volume and appears in the statement of Proposition \ref{deffunccusp}. It corresponds to the region where the geodesic boundary defining function is well-defined.}\label{tx.1}
\end{figure}

\subsection{Renormalized volume in the convex co-compact case}\label{convcocomp}
A geometrically finite hyperbolic $3$-manifold $X=\Gamma\backslash \hh^3$ with no cusps is called \emph{convex co-compact}. Such a manifold $X$
can be decomposed as $X=\mc{K}\cup \mc{U}$ where $\mc{K}\subset X$ is a compact region with smooth boundary and  
$\mc{U}$ is isometric to 
\begin{equation}\label{productform} 
M\x (0, \delta)_\rho, \textrm{ with metric } g=\frac{d\rho^2+h(({\rm Id}+\demi \rho^2A)^2\cdot,\cdot)}{\rho^2}
\end{equation}
where $M=\Gamma\backslash \Omega$ is a compact surface (not necessarily connected), $h$ is a metric on $M$, $A$ is a symmetric endomorphism of $TM$ satisfying the trace and divergence properties 
\begin{equation}\label{constraint}
{\rm Tr}_h(A)=-\tfrac{1}{2}{\rm Scal}_{h}, \quad \delta_{h}(A)=\tfrac{1}{2} d \,{\rm Scal}_{h},
\end{equation}
see \cite[Th 7.4]{FeGr} or \cite{KrSc} for details. The product form \eqref{productform} will also be written
\[g=\frac{d\rho^2+h_0+\rho^2h_2+\rho^4h_4}{\rho^2} \]
with $h_0=h$, $h_2(\cdot,\cdot)=h(A\cdot,\cdot)$ and $h_4(\cdot,\cdot):=\frac{1}{4}h(A^2\cdot,\cdot)$.
The manifold $M$ is compact and, when equipped with the conformal class $[\rho^2g|_{TM}]=[h]$, is the conformal 
boundary of $X$. As above, $X$ can be compactified smoothly into $\bbar{X}$ with boundary $\pl\bbar{X}=M$ and $\rho$, viewed as a function on $X\setminus \mc{K}$ is a smooth boundary defining function. The function $\rho$ in $\mc{U}$ so that the metric has the form \eqref{productform} is not unique and is characterized by the property
\[  \Big|\frac{d\rho}{\rho}\Big|_g=1\,\, \textrm{ in }\,\, \mc{U} , \quad \textrm{ and } (\rho^2g)|_{TM}=h.\] 
In fact, for each metric $\hat{h}$ conformal to $h$, there is a unique function $\hat{\rho}$ near $\pl\bbar{X}$ 
satisfying $|d\hat{\rho}/\hat{\rho}|_{g}=1$ and $\hat{\rho}^2g|_{\hat{\rho}=0}=\hat{h}$, and we call $\hat{\rho}$ the \emph{geodesic boundary defining function associated to the conformal representative $\hat{h}$}. We just recall briefly the argument of Graham \cite{Gr}, as it will be useful later for the cusp case: take $\rho$ a boundary defining function of $\bbar{X}$, then the structure of the hyperbolic metric on $\hh^3$ near its boundary implies that 
$\bar{g}=\rho^2g$ is smooth up to $\pl\bbar{X}$ and $|d\rho/\rho|_g$ is smooth on 
$\bbar{X}$ and equal to $1$ at $\pl\bbar{X}$ (that follows from the fact that $g$ has curvature $-1$ in $\mc{U}$, 
see \cite{MM}), then writing 
$h:=(\rho^2g)|_{TM}$ and 
$\hat{\rho}=\rho e^{\omega}$ the equation $|d\hat{\rho}/\hat{\rho}|_g=1$ with the condition $\hat{\rho}^2g|_{\pl\bbar{X}}=
\hat{h}=e^{2\varphi}h$ for some $\varphi\in \mc{C}^\infty(M)$  is equivalent to the equation
\[ 2\cjg d\omega,d\rho\cjd_{\bar{g}}+\rho|d\omega|^2_{\bar{g}}=\frac{1-|d\rho^2|_{\bar{g}}}{\rho}, \textrm{ with boundary condition }\omega|_{\pl\bbar{X}}=\varphi.\]
This is a non-characteristic Hamilton-Jacobi equation with smooth coefficients which can be solved near the boundary by the method of characteristics and  the solution is unique. We then extend $\hat{\rho}$ smoothly outside 
this neighborhood as a positive function in any fashion. 
The form of the metric $g$ in the collar neighborhood of $M=\pl\bbar{X}$ induced by the gradient flow of $\hat{\rho}$ with respect to 
the metric $\hat{\rho}^2g$ is then of the form $g=(d\hat{\rho}^2+\hat{h}(\hat{\rho}))/\hat{\rho}^2$ for some one-parameter smooth family $\hat{h}(\hat{\rho})$ of metrics on $M$ parametrized by $\hat{\rho}$, and 
the constant sectional curvature $-1$ implies the form \eqref{productform} with \eqref{constraint} (using Gauss and Codazzi constraint equations).\\

If $\rho$ is a geodesic boundary defining function near $\pl\bbar{X}$ associated to a conformal representative $h\in[h]$, extended smoothly as a positive function on 
$X$, then the form  \eqref{productform} of the metric in $\mc{U}$ implies that the Riemannian volume measure in $\mc{U}$ has the form 
$\rho^{-3}{\rm dvol}_g= G(\rho)d\rho\, d{\rm vol}_{h}$ for some smooth function $G\in \mc{C}^\infty([0,\delta))$. 
It is then direct to see (see \cite{Al,GMS} for details) that 
\[ H(z):= \int_{X} \rho^z {\rm dvol}_g\] 
has a meromorphic extension from $\{z\in \cc; {\rm Re}(z)>2\}$ to $\cc$, with a simple pole at $z=0$  and the value of the finite part of $H(z)$ at $z=0$ is independent of the value of $\rho$ in any fixed compact set $\mc{K}\subset X$: in fact 
\[ {\rm FP}_{z=0}H(z)=\Big({\rm FP}_{z=0}\int_{X\setminus \mc{K}} \rho^z {\rm dvol}_g\Big)+ {\rm Vol}_g(\mc{K}).\]
We define the renormalized volume of $X$ with respect to the conformal representative $h\in [h]$ by
\[ {\rm Vol}_R(X,h):={\rm FP}_{z=0} \int_{X} \rho^z {\rm dvol}_g.\]
As a function on the set of metrics in the conformal class $[h]$ with volume $2\pi\chi(\pl\bbar{X})$, 
the functional ${\rm Vol}_R(X,h)$ has a unique maximum at $h=h^{\rm hyp}$, the unique hyperbolic metric in the conformal class, 
see for instance \cite[Prop. 3.1]{GMS}.

\begin{defi}\label{renormvol}
Let $X$ be a convex co-compact hyperbolic $3$-manifold with conformal boundary $(M,[h])$ a Riemann surface
admitting a hyperbolic metric, ie. $M$ does not contain genus-$1$ connected components. 
Let $h^{\rm hyp}\in [h]$ be the unique hyperbolic representative in the conformal class $[h]$, and let $\rho$ be the geodesic boundary defining function 
associated to $h^{\rm hyp}$, defined uniquely near $M$  and extended smoothly as a positive function in $X$. The renormalized volume of $X$ is defined to be 
\[{\rm Vol}_R(X):= {\rm FP}_{z=0} \int_{X} \rho^z {\rm dvol}_g ={\rm Vol}_R(X,h^{\rm hyp})
\]
where $g$ is the hyperbolic metric on $X$.
\end{defi}
The choice of the conformal representative $h^{\rm hyp}\in[h]$ to be hyperbolic is important and 
yields particularly interesting properties of the renormalized volume related to quasi-Fuchsian reciprocity and 
construction of K\"ahler potential for the Weil-Peterson metric; see \cite{KrSc, GMS}.

\subsection{Uniformisation of Riemann surfaces with cusps}\label{Sec:uniform}

\begin{defi}
A \emph{hyperbolic cusp} is a region $\{y>R\}$ of the quotient $\cjg z\to z+\tfrac{1}{2}\cjd\backslash \hh^2$ for some $R>0$, 
where $z=y+iw$ are coordinates on the hyperbolic half-plane $\hh^2$. 
\end{defi}
This set is 
also isometric to 
\[ \Big((R,\infty)_y\x (\rr/\tfrac{1}{2}\zz)_w , 
h=\frac{dy^2+dw^2}{y^2}\Big)\simeq \Big((0,\tfrac{1}{R})_v\x (\rr/\tfrac{1}{2}\zz)_w , h=\frac{dv^2}{v^2}+v^2dw^2\Big).\]
A \emph{surface with hyperbolic cusps} $(M,h)$ is a surface isometric outside a compact set to a finite disjoint union of
hyperbolic cusps.

We can compactify $M$ into a smooth compact surface $\bbar{M}$ with boundary by replacing each cusp end 
$(0,\tfrac{1}{R})_v\x (\rr/\tfrac{1}{2}\zz)_w$ by $[0,\tfrac{1}{R})_v\x (\rr/\tfrac{1}{2}\zz)_w$, that is, by adding circles at infinity 
of the cusp end. 

We can also compactify $M$ to a compact surface $\Sigma$ by adding a finite number of points, 
one for each cusp. Define a conformal coordinate
near such a point by $\zeta=\exp(4\pi(-y+iw))$. (The factor $4\pi$ is needed in order for
$e^{4\pi iw}$ to be well-defined for $w\in\rr/\tfrac{1}{2}\zz$.)
We compute
\begin{equation}\label{hdz}
|d\zeta|^2= (4\pi)^2 |\zeta|^2 (dy^2+dw^2) = (4\pi)^2 |\zeta|^2 y^2 h.
\end{equation}
Since $h$ is conformal to $|d\zeta|^2$, we get in this way a conformal structure on $\Sigma$. 
If $M$ is oriented, $\Sigma$ becomes a compact Riemann surface.

If we take a boundary defining function $\rho$ in a geometrically finite hyperbolic $3$-manifold
with a certain behaviour near the cusps, we see that the conformal infinity $M =\pl\bbar{X}$ will have
a metric with a hyperbolic cusp in the conformal class: indeed, set $\rho$ to be a smooth boundary defining function for $M$ in $\bbar{X}$ such that 
\[ \rho=\frac{x}{\sqrt{x^2+y^2}} \,\, \textrm{ in  }\,\,\mc{U}^c_j .\] 
Then the metric $h:=\rho^2g|_{TM}$ is a smooth metric on $M$ which is given near the cusps, 
that is in $\mc{U}_j^c\cap M\simeq\{y\in \rr;  |y|>R_j\}\x (\rr/\tfrac{1}{2}\zz)$,  by 
\[
h=\frac{dy^2 + dw^2}{y^2},\quad |y|>R_j, \,\, w\in \rr/\tfrac{1}{2}\zz
\]
or equivalently, using the coordinates $(u,v,w)$ of \eqref{Uj'}, $\rho=u/\sqrt{u^2+v^2}$ and
\[ h=\frac{dv^2}{v^2} + v^2dw^2,\quad 0<|v|<1/R_j, \,\, w\in \rr/\tfrac{1}{2}\zz.\]

We define on $\bbar{M}$ the space 
$\dot{\mc{C}}^\infty(\bbar{M})$ to be the Fr\'echet subspace of $\mc{C}^\infty(\bbar{M})$ consisting of functions vanishing 
to infinite order at $\pl \bbar{M}$. We also define $\mc{C}_r^\infty(\bbar{M})$
to be the subspace of $\mc{C}^\infty(\bbar{M})$ consisting of functions $f$ such that 
$\pl_w f\in \dot{\mc{C}}^\infty(\bbar{M})$.
This corresponds to smooth functions whose Taylor series at the boundary
is of the form  
\begin{equation}\label{Taylor}
f(x,\theta)\sim\sum_{k=0}^\infty a_k v^k
\end{equation} 
where $a_k$ are constants, rather than functions of $w$. 
On a surface $(M,h)$ with hyperbolic cusps, we say that a symmetric tensor $h'\in C^\infty(M;S^2(T^*M))$ 
is a \emph{cusp symmetric tensor} if there exist $A\in C^\infty(\bbar{M};{\rm End}(T\bbar{M}))$ 
self-adjoint with respect to $h$ such that $h'(\cdot,\cdot)=h(A\cdot,\cdot)$.

We first claim the following uniformisation theorem, see \cite[Theorem 3]{Rochon-Zhang} 
for a related result for K\"ahler-Einstein metrics.
\begin{prop}\label{uniform}
Let $h$ be a metric on a surface $M$ with hyperbolic cusps and let $\bbar{M}$ be the compactification as above. 
There exists a unique conformal factor $\varphi\in \mc{C}^\infty(M)\cap L^\infty(M)$ 
such that $h^{\rm hyp}=e^{2\varphi}h$ has constant curvature $-1$ on $M$. Moreover,
$\varphi\in \mc{C}_r^\infty(\bbar{M})$ and $\varphi|_{\pl \bbar{M}}=0$. More precisely,
in every cusp of $M$,
\[\varphi(v,w)+\log(1+av)\in \dot{\mc{C}}^\infty(\bbar{M})\]
for some $a\in\rr$ depending on the cusp.
\end{prop} 
\begin{proof} 
The surface $(M,h)$ is conformal to the compact Riemann surface $\Sigma$ with a finite number of points removed, hence its 
fundamental group is non-commutative and free.
The Poincar\'e--Koebe uniformization theorem implies that 
$M$ with its induced conformal structure is conformal to a complete 
hyperbolic quotient. In other words, there exists a unique conformal factor $\varphi\in \mc{C}^\infty(M)$
such that the Riemannian metric $h^{\rm hyp}=e^{2\varphi}h$ is hyperbolic and complete.
The complete hyperbolic metric on a punctured Riemann surface 
is known to have hyperbolic cusps near the punctures, 
hence there exist isometries $\Phi$ between the hyperbolic cusps 
of $h$ and $h^{\rm hyp}$ near the punctures. Such a $\Phi$ is 
a holomorphic self-map of $\Sigma$ defined only near the punctures,
and $\Phi^*h=e^{2\varphi}h$. 

Note that $\Phi$ is an isometry, hence it is proper. It follows that it extends continuously
as the identity on the punctures. The possible singularities of $\Phi$ at the punctures are thus 
removable since the target surface $\Sigma$ is compact, so in terms of the complex variable
$\zeta=\exp(4\pi(-y+iw))$, we have $\Phi(\zeta)=\zeta f(\zeta)$ with $f(0)=\Phi'(0)\neq 0$. 
Then using \eqref{hdz}
\[h=\frac{|d\zeta|^2}{16\pi^2 |\zeta|^2 y^2}= \frac{1}{|\zeta|^2\log^2|\zeta|}|d\zeta|^2. \]
This implies
\begin{align*}
\Phi^*h ={}& \frac{|\Phi'(\zeta) d\zeta|^2}{|\zeta f(\zeta)|^2\log^2|\zeta f(\zeta)|}=
\frac{|1+\zeta \frac{f'(\zeta)}{f(\zeta)}|^2}{\left(1+\frac{\log|f(\zeta)|}{\log|\zeta|}\right)^2}h.
\end{align*}
In terms of the boundary-defining function $v=1/y$, 
\begin{align*}
\log|\zeta|=-\frac{4\pi}{v},&& \zeta \frac{f'(\zeta)}{f(\zeta)} \in \dot{\mc{C}}^\infty(\bbar{M}),&&
\log|f(\zeta)|-\log|f(0)|\in \dot{\mc{C}}^\infty(\bbar{M}).
\end{align*}
Thus the conformal factor $e^{2\varphi}=\frac{\Phi^* h}{h}$ satifies 
$\varphi+\log(1-\frac{\log|f(0)|}{4\pi}v)\in \dot{\mc{C}}^\infty(\bbar{M})$ near the cusp.
\end{proof}

\subsection{The renormalized volume of geometrically finite hyperbolic $3$-manifolds}

We now wish to define a renormalized volume for a geometrically finite hyperbolic 
$3$-manifold $X$ with rank-$1$ cusps. We proceed like in the convex co-compact case, 
by first uniformizing the conformal infinity $(M,[h])$ with the choice of the finite volume 
hyperbolic metric $h$ in the conformal class
and then constructing a geodesic boundary defining function $\rho$ in $\bbar{X}$ associated to $h$. The difficulty here is that the conformal boundary is non-compact 
and it is not clear what the behavior of the function $\rho$ near the cusp in $\bbar{X}$ is.
We proceed as in Section \ref{Sec:uniform}: we start by choosing $\rho$ as a smooth boundary defining function near $\pl\bbar{X}=M$ which is equal to $\rho=u/\sqrt{u^2+v^2}$ in the model \eqref{Uj'} of each cusp neighborhood $\mc{U}_j^c$, the metric $h\in [h]$ 
obtained by $h=\rho^2g|_{TM}$ in the conformal infinity is then hyperbolic outside a compact region of $M$. Then by Proposition \ref{uniform} we know that there exists a hyperbolic metric $h^{\rm hyp}=e^{2\varphi}h$, with $\varphi\in \mc{C}^\infty(\bbar{M})$ and $\varphi|_{\pl \bbar{M}}=0$. 
We obtain the following Proposition, whose proof is done in Section \ref{hamilton}.
\begin{prop}\label{deffunccusp}
Let $X$ be a geometrically finite hyperbolic $3$-manifold with rank-1 cusps. Let $(M, [h])$ be the conformal infinity and $h^{\rm hyp}$ be the complete hyperbolic metric with cusps in the conformal class obtained from Proposition \ref{uniform}.   For each $\psi\in {\mathcal{C}}^{\infty}_r(\bbar{M})$, consider the conformal representative $\hat{h}:=e^{2\psi}h^{\rm hyp} $. 
There exists a smooth boundary defining function $\hat{\rho}\in \mc{C}^\infty(\bbar{X}_c)$ 
of the boundary hypersurface $\bbar{M}$ in $\bbar{X}_c$ and a closed set $\mc{V}\subset \bbar{X}_c$
with finite volume, intersecting $\pl \bbar{X}_c$ only at ${\rm cf}$, such that 
\begin{equation}\label{geodforXc}
\Big| \frac{d\hat{\rho}}{\hat{\rho}}\Big|_g=1 \textrm{ in }\bbar{X}_c\setminus \mc{V}, \quad \hat{\rho}^2g|_{TM}= \hat{h}.\end{equation}
The function $\hat{\rho}$ is defined uniquely near $\bbar{M}$.
There is a smooth diffeomorphism $\phi: \bbar M\times [0,\epsilon)_x\to \bbar{X}_c\setminus \mc{V}$ such that $\phi^*\hat{\rho}=x$ and $\phi^*g$ admits a finite expansion in powers of $x$,
\begin{equation}
     \phi^*g=  \frac{dx^2+ \hat{h}_0 + x^2\hat{h}_2 + \frac14 x^4 \hat{h}_4}{x^2}
\label{mai.1}\end{equation}
where the coefficients $\hat{h}_0=\hat{h}$ and $\hat{h}_2,\hat{h}_4$ are cusp symmetric tensors such that 
$$
\Tr_{\hat{h}_0}(\hat{h}_2)= -\frac12 \Scal_{\hat{h}_0} \quad \mbox{and} \quad  \delta_{\hat{h}_0}(\hat{h}_2)= \frac12 d \, \Scal_{\hat{h}_0}
$$
and $\hat{h}_4(\cdot,\cdot)= \frac14 \hat{h}_0(A^2\cdot,\cdot)$ if $\hat{h}_2(\cdot,\cdot)= \hat{h}_0(A\cdot,\cdot)$.
Finally, the function  
$H(z):= \int_{X} \hat{\rho}^z {\rm dvol}_g$
admits a meromorphic extension from ${\rm Re}(z)>2$ to a neighborhood of $z=0$, with pole of order $1$ at $z=0$.
\end{prop}
A smooth boundary defining function of $\bbar{M}$ in $\bbar{X}_c$ is called 
\emph{geodesic boundary defining function} associated to $\hat{h}$ if it satisfies 
\eqref{geodforXc}. Similarly to the convex co-compact case, the value of the finite part 
at $z=0$ of the integral in any compact subset $\mc{V}\subset \bbar{X}_c\setminus \bbar{M}$ 
is independent of the value of $\rho$ in $\mc{V}$:
\[{\rm FP}_{z=0}H(z)=\Big({\rm FP}_{z=0}\int_{X\setminus 
\mc{V}} \rho^z {\rm dvol}_g\Big)+ {\rm Vol}_g(\mc{V}).\]
This is a consequence of the fact that $\mc{V}$ has finite volume. In other words, ${\rm FP}_{z=0}H(z)$ 
depends only on the values of $\rho$ in arbitrarily small neighborhoods of $\bbar{M}$ in $\bbar{X}_c$, 
and thus it depends only on the conformal representative $\hat{h}\in[h]$ in the conformal infinity.

Now we can define the renormalized volume in this setting:
\begin{defi}\label{renormvolbis}
Let $X$ be a  geometrically finite hyperbolic $3$-manifold with rank-$1$ cusps, and 
with conformal boundary $(M ,[h])$ admitting a complete hyperbolic metric. 
Let $h^{\rm hyp}\in [h]$ be the unique hyperbolic representative in the conformal class $[h]$. 
For $\psi\in\mathcal{C}_r^{\infty}(\bbar{M})$, let $\hat{\rho}\in \mc{C}^\infty(\bbar{X}_c)$ 
be the geodesic boundary defining function of $\bbar{M}$ associated to 
$\hat{h}:=e^{2\psi}h^{\rm hyp}$ defined uniquely by Proposition \ref{deffunccusp} 
in a neighborhood of $\bbar{M}$ in $\bbar{X}_c$ and extended in any fashion as 
a positive smooth function in $\bbar{X}_c\setminus \bbar{M}$. 
The renormalized volume of $X$ associated to the conformal representative 
$\hat{h}=e^{2\psi}h^{\rm hyp}$ is defined to be 
\[{\rm Vol}_R(X,\hat{h}):= {\rm FP}_{z=0} \int_{X} \hat{\rho}^z \,{\rm dvol}_g\]
where $g$ is the hyperbolic metric on $X$.  We define the renormalized volume of $X$ by
$$  {\rm Vol}_R(X):= {\rm Vol}_R(X,h^{\rm hyp}).$$  
\end{defi}

\section{Formation of a cusp from Schottky groups}\label{sec:Schottky}

\subsection{Notations} We shall use mainly the representation of $\hh^3$ as a half-space $\rr^+_x\x \cc_z$ in $\rr^3$, 
the boundary then becomes $\pl\hh^3=\{0\}\x \cc\simeq \cc$. 

The intersection of $\hh^3=\rr^+_x \x \cc_z$ 
with the Euclidean ball of radius $r>0$ centered at a boundary point
$z\in \pl\hh^3$ is called a \emph{half-ball} of $\hh^3$, and we denote it by $B(z,r)$.
The boundary of a half-ball $B(z,r)$ in $\hh^3$ is called a \emph{half-sphere} of center $z\in\cc$ and radius $r>0$, and is denoted by $H(z,r)$. 
In terms of hyperbolic geometry, $B(z,r)$ is an unbounded domain with totally geodesic boundary given by the half-sphere $H(z,r)=\pl B(z,r)\cap \hh^3$
We say that the ball is \emph{supported} by the disc $D(z,r)\subset \cc$ of center $z$ and radius $r$ 
(this corresponds to $\pl\hh^3\cap \bbar{B(z,r)}$). Similarly we say that $H(z,r)$ is \emph{supported} by the circle
$C(z,r)=\pl D(z,r)$ in $\pl\hh^3\simeq \cc$ (this corresponds to $\pl\hh^3\cap \bbar{H(z,r)}$).

\subsection{Schottky groups}\label{Schottky}
We shall analyze the behaviour of  $\VolR(X^{\eps})$ for a family $(X^{\eps})_{\eps>0}$ 
of convex co-compact hyperbolic $3$-manifolds such that, as $\eps \to 0$,
$X^\eps$ converges to a hyperbolic $3$-manifold $X^0$ with rank-$1$ cusps. Here, we take 
$\eps\geq 0$ to be a continuous parameter, but one could of course also consider sequences.
The case that we consider is a smooth (in $\eps>0$) family of Schottky groups $\Gamma^\eps$ with certain loxodromic generators of ${\rm PSL}_2(\cc)$ 
converging to parabolic transformations of ${\rm PSL}_2(\cc)$.

We recall that a \emph{marked Schottky group} $\Gamma\subset {\rm PSL}_2(\cc)$ of genus $g$ 
with \emph{standard generators} $\gamma_1,\dots,\gamma_g\in {\rm PSL}_2(\cc)$ is a group 
generated by these generators  such that there exist $2g$ disjoint Jordan curves 
$(C_{\pm j})_{j=1\dots,g}$ in $\mathbb{S}^2=\pl \hh^{3}$  bounding a connected open domain 
$D\subset \mathbb{S}^2$ with $\gamma_j(D)\cap D=\emptyset$ and $\gamma_j(C_{-j})=C_{+j}$. 
Then $\Gamma$ is free and contains only  loxodromic elements, with fundamental domain $D\cup_{j}C_{\pm j} \subset \mathbb{S}^2$ for the action of $\Gamma$ on the discontinuity set $\Omega\subset \mathbb{S}^2$ (which is connected open and dense set in $\mathbb{S}^2$). It is shown by Chuckrow \cite{Ch} that every set of $g$ free generators of a Schottky group is in fact a set of standard generators.
A Schottky group is said to be a \emph{classical Schottky group} if there is some set 
of free generators for which the curves $C_{\pm j}$ can be taken to be circles. 
A family of circles associated to the generators satisfying the conditions as above 
will be called a set of \emph{adapted circles}. Such a set is of course not unique.
We can view $\Gamma$ as a discrete group of isometries acting freely and discontinuously 
on $\hh^3$, and as a group of conformal transformations acting freely and discontinuously 
on the discontinuity set $\Omega\subset \mathbb{S}^2$. To define the Schottky space $\mc{S}_g$, 
we follow Chuckrow \cite{Ch}:  ${\rm PSL}_2(\cc)$ 
identifies with  $P_3(\cc)\setminus Y$ where $P_3(\cc)$ is the 
$3$-dimensional complex projective space, and $Y$ the algebraic submanifold 
$\{\gamma\in{\rm PSL}_2(\cc);\det\gamma=0 \}$. Consider the subset $\mc{U}_g$ of 
$({\rm PSL}_2(\cc))^g$
consisting of groups with $g$ generators $\gamma_1,\dots,\gamma_g$ such that there are 
at least 3 distincts fixed points among those of $\gamma_j$. Then  $\mc{U}_g$ 
is an open and connected subset of $(P_3(\cc)\setminus Y)^g$.
There is an action of ${\rm PSL}_2(\cc)$ on $\mc{U}_g$ by conjugation:  
\[ (B,( \gamma_1,\dots,\gamma_g) )\mapsto ( B\gamma_1 B^{-1},\dots, B\gamma_g B^{-1})
\]
and $\mc{U}_g/{\rm PSL}_2(\cc)$ is a complex manifold of dimension $3g-3$ with 
the natural topology inherited from $({\rm PSL}_2(\cc))^g$. One way of fixing 
coordinates on this space is to fix $3$ distincts fixed points of the generators 
by conjugating the group with an appropriate element of ${\rm PSL}_2(\cc)$. 
More precisely near a $\Gamma \in \mc{U}_g/{\rm PSL}_2(\cc)$, up to conjugation, 
we can choose the generators $\gamma_j$ so that $0$, $1$ and $\infty$ are 
the three distinct fixed points among the generators, then there are local 
complex coordinates on $\mc{U}_g/{\rm PSL}_2(\cc)$ near $[\Gamma]$ 
given by the coefficients $a_j,b_j,c_j,d_j\in \cc$ so that 
$\gamma_j(z)=(a_jz+ b_j)/(c_jz + d_j)$ with $a_jd_j-b_jc_j=1$ 
(notice that $3$ complex parameters among the  $\gamma_j$ are fixed).
The Schottky space $\mc{S}_g$ is the subset of $\mc{U}_g/{\rm PSL}_2(\cc)$ 
corresponding to equivalence classes of marked Schottky groups. 
For a group $\Gamma \in \mc{S}_g$, we can always choose the three distinct 
fixed points to be the repulsive and attractive fixed point of $\gamma_1$ 
and the repulsive fixed point of $\gamma_2$, and one then gets global complex 
coordinates by conjugating the groups so that $0$ and $\infty$ are the attractive 
and repulsive fixed point of $\gamma_1$ and $1$ is the repulsive point of $\gamma_2$. 
This system of coordinates is not well adapted to the description of groups 
tending to the boundary of $\mc{S}_g$ with $\gamma_1$ becoming parabolic.
Chuckrow \cite[Lemma~5]{Ch} showed that $\mc{S}_g$ is a connected open subset of 
$\mc{U}_g/{\rm PSL}_2(\cc)$.  The Schottky classical space $\mc{S}_g^0$ 
is the open subset of those groups which are classical. 
Chuckrow \cite[Theorems~4 and 5]{Ch} also showed that boundary points in $\pl\mc{S}_g$ are free groups 
with $g$ generators, without elliptic transformations, which either have a 
parabolic element or are not discontinuous, and both cases may happen. Marden \cite{Ma} 
proved that groups in $\pl\mc{S}_g$ are discrete, that 
$\mc{S}_g\setminus \bbar{\mc{S}_g^0}$ is a non-empty open set, and groups in 
$\pl\mc{S}^0_{g}$ are discontinuous. Later Jorgensen-Marden-Maskit \cite{JMM} proved that 
all points in $\pl \mc{S}_g^0$ are geometrically finite Kleinian group with parabolic 
elements. Thus $\mc{S}_g^0$ is better behaved and we will only 
focus on this space.

\subsection{Admissible Schottky groups}
We consider a sequence of classical Schottky groups $\Gamma^\eps\in \mc{S}_{g}^0$ where $(\gamma_1^\eps,\dots,\gamma_g^{\eps})$ 
converge to $(\gamma^0_1,\dots,\gamma^0_g)$ with $\Gamma^0$ generated by these elements in 
$\pl\mc{S}_g^0\cap \pl\mc{S}_g$ as $\eps\to 0$, so that $\Gamma^0$ 
is a geometrically finite free group with $g$ generators, with parabolic elements. We assume that 
$\gamma_j^\eps$ is smooth in $\eps\in [0,1]$ for $j\leq g$.
We use the coordinates on $\mc{S}_g$ as above, so that the fixed points of $\gamma_1^\eps$ are $0$ and $\infty$, and the repulsive point of $\gamma_2^\eps$ is $1$.
For notational simplicity, we shall sometime remove the $0$ superscript for the limiting objects at $\eps=0$, for instance we shall use $\gamma_j$ for $\gamma_j^0$.
We write these elements of ${\rm PSL}_2(\cc)$ as
\[\begin{gathered}
\gamma_j(z)=\frac{a_jz+b_j}{c_jz+d_j}, \quad \gamma^\eps_j(z)=\frac{a_j(\eps)z+b_j(\eps)}{c_j(\eps)z+d_j(\eps)}, \\
\textrm{ with } a_jd_j-b_jc_j=1  \textrm{ and } 
a_j(\eps)d_j(\eps)-b_j(\eps)c_j(\eps)=1.
\end{gathered}\]
For $j<j_0$, the fixed points of $\gamma_j^\eps$ are denoted $p_{\pm j}(\eps)$ and given by 
\begin{equation}
\label{fixedpoints} 
p_{\pm j}(\eps)= \tfrac{a_j(\eps)-d_j(\eps)}{2c_j(\eps)}\pm \tfrac{\sqrt{\Tr(\gamma_j^\eps)^2-4}}{2c_j(\eps)}. 
\end{equation}
(we use the determination of $\sqrt{\cdot}$ on $\cc\setminus \rr_-$). 
Up to possibly exchanging $\gamma_j^\eps$ by its inverse in our choice of generators, 
we can assume that 
$p_{+j}(\eps)$ is the attractive, and $p_{-j}(\eps)$ the repulsive fixed point. 
The geodesic in $\hh^3$ relating $p_{-j}(\eps)$ to $p_{+j}(\eps)$ is called the \emph{axis} of $\gamma_j^\eps$. The Euclidean distance in $\cc$ between the two fixed points of $\gamma_j^\eps$ is 
\begin{equation}\label{p_+-p_-} 
|p_{+j}(\eps)-p_{-j}(\eps)|= \frac{|\Tr(\gamma_j^\eps)^2-4|^\demi}{|c_j(\eps)|}.
\end{equation}
Take a family of adapted circles $C_{\pm j}^\eps$ bounding a fundamental domain $D^\eps$. 
Notice that $D^\eps$ has compact closure contained in the region bounded by $C_{-1}^\eps$ 
and $C_{+1}^\eps$ in $\cc$. If the circles $C_{\pm 1}^\eps$ are not contained in a compact 
set of $\cc$ independent of $\eps$, then all fixed points of a subsequence of $\gamma^\eps_j$ 
for $j>1$ tend to $\infty$, and that is not possible since 
the limiting transformations $\gamma_j$ and $\gamma_k$ can not have common fixed 
points if $j\not= k$, according to \cite[Lemma 2.3]{Ma}. For the same reason, $D^\eps$ can not shrink to $0$ and more generally to a point of $\cc$. Up to extraction of a  subsequence $\eps_n\to 0$, 
the circles $C_{\pm j}^{\eps_n}$ then converge to circles or points, and for 
$j=1$ the limits $C_{\pm 1}$ are circles. If they are disjoint then the limiting domain $D^0$ is non-empty
and thus, if some circle $C_{\pm j}^{\eps_n}$ converge to a point $p$, we obtain a contradiction since $\gamma_j$ would have to map $D$ to $p$. This shows in that case that all $C_{\pm j}^{\eps_n}$ converge to circles $C_{\pm j}$. If $C_{+1}=C_{-1}$, then since $\gamma_1^\eps\to \gamma_1$, we necessarily have that 
$\gamma_1$ is elliptic or the Identity, but this can not happen by \cite{Ch} since $\Gamma^0$ can not contain elliptic elements and must be a free group with $g$ generators given by $\gamma_1,\dots \gamma_g$.
We thus conclude that $D^{\eps_n}\to D^0$ non-empty, bounded by circles $C_{\pm j}$. Necessarily, at least two of the circles $C_{\pm j}$ must intersect at a point since we assume that $\Gamma^0$ is not in $\mc{S}_g^0$.
We will make the assumption that the limiting circles satisfy
\begin{equation}\label{assumpcircles}
C_{\pm j}\cap (C_{+k}\cup C_{-k})=\emptyset,   \textrm{ if }  j\not=k.
\end{equation}
Thus there are $g-j_0$ of the generators $\gamma_j$ that are parabolic for some $0<j_0\leq g-1$. Without loss of generality, we choose them to be $\gamma_{j}$ for $j=j_0+1,\dots,g$.
For $j\leq j_0$ the $\gamma_j$ are loxodromic. If $j>j_0$, we have $\Tr(\gamma_j)=2$ at the limit 
and the unique fixed point of the parabolic transformation is $p_j=\frac{a_j-d_j}{2c_j}$, and  
we have that $|\Tr(\gamma_j^\eps)-2|^\demi/|c_j(\eps)|\to 0$ as $\eps\to 0$. 
The fundamental domain $D^\eps$ for $\Gamma^\eps$ is uniformly bounded
in $\cc$, and $c_j(\eps)\not=0$ for $j>1$ in order to have adapted circles $C_{\pm j}^{\eps}$ associated to $\gamma_j^{\eps}$. To be adapted, the disk bounded by the circle $C_{+j}^\eps$
has to contain the point $z=a_j(\eps)/c_j(\eps)$ (which is the image of $\infty$ by $\gamma_j^\eps$) and the disk bounded by the 
circle $C_{-j}^\eps$ has to contain $z=-d_j(\eps)/c_j(\eps)$ (which is mapped to $\infty$ by $\gamma_j^\eps$); since $C_{\pm j}^\eps$ also contains 
$p_{\pm j}(\eps)$, we deduce that when $j>j_0$ and $\eps\to 0$, the fact that $p_{\pm j}(\eps)\to p_j$ implies that 
the radius of $C_{\pm j}^\eps$ is bounded below by $(|\Tr(\gamma_j^\eps)|-|\Tr(\gamma_j^\eps)-2|^\demi)/4|c_j(\eps)|$ for small $\eps>0$. In particular $c_j(\eps)$ converge to $c_j\not=0$ as otherwise the radius
of the adapted circles would tend to $\infty$.
There is a subsequence $\eps_n$ where for each $j$ there are adapted circles $C_{\pm j}^{\eps_n}$ 
associated to the $\gamma_j^{\eps_n}$ that converge to circles $C_{\pm j}^0$ (also denoted $C_{\pm j}$),
 which are tangent if and only if $j>j_0$ and with $C_{+j}\cap C_{-j}=p_j$ being the fixed point of $\gamma_j$.  Then the limiting (tangent) adapted circles for $j>j_0$ have radius bounded below by 
$1/2|c_j|$. Moreover, an easy computation shows that $\gamma_j$ preserve the line 
$c_j^{-1}(a_j+\rr)\subset \cc$ that we call the \emph{axis} of $\gamma_j$. 

The element $\gamma_1^\eps$ is of the form 
\[\gamma_1^\eps(z)=q_1(\eps)z,\quad q_1(\eps)\in \cc ,\,\, |q_1(\eps)|>1\]
and each $\gamma_j^{\eps}$ for $j>1$ can be written 
as the transformation of $\cc\simeq \mathbb{S}^2\setminus{\{\infty\}}$ satisfying
\[\theta_j^\eps\circ \gamma_j^{\eps}(z)=q_j(\eps)\theta_j^\eps(z), \quad \theta_j^\eps(z):=-\frac{z-p_{-j}(\eps)}{z-p_{+j}(\eps)}\]  
where $p_{\pm j}(\eps)\in \cc$ are the two fixed points of $\gamma_j^{\eps}$ and $q_j(\eps)\in\cc$ is the complex multiplier with $|q_j(\eps)|>1$ (we take $p_{+j}(\eps)$ to be the attractive fixed point).
The multiplier will also be written as 
\begin{equation}\label{qk} 
q_j(\eps)=e^{\ell_j(\eps)+i\alpha_j(\eps)}\end{equation}
for some $\ell_j(\eps)>0$ and $\alpha_j(\eps)\in [0,2\pi)$. Since for $j> j_0$, $\gamma_j^{\eps}$ converge to a parabolic element $\gamma_j$ with fixed point $p_j$, then $q_j(\eps)\to 1$ since  
$\Tr(\gamma_j^{\eps})^2-4=(q_j(\eps)-1)^2/q_j(\eps)$ must converge to $0$. 
The axis of $\gamma_j^\eps$ is mapped to $\rr^+\x \{0\}\subset \hh^3$ by 
$(\theta_j^{\eps})^{-1}$ in the half-space model $\hh^3=\rr_x^+\x \cc_z$

The set $D^\eps$ is a fundamental domain for the action 
of $\Gamma^\eps$ on the discontinuity set $\Omega^\eps\subset \cc$. 
The group acts properly discontinuously on $\Omega^\eps$ by conformal transformations 
and the quotient $M^\eps=\Gamma^\eps\backslash D^\eps=\Gamma^\eps\backslash \Omega^\eps$ 
is a closed Riemann surface of genus $g$, with conformal structure induced by that 
of $\cc$. It is the conformal boundary of the hyperbolic $3$-manifold 
$X^\eps:=\Gamma^\eps\backslash \hh^3$. We denote by $F^\eps\subset \hh^3$ 
the fundamental domain for the group $\Gamma^\eps$ with totally geodesic boundary 
satisfying $\pl F^\eps\cap \pl\hh^3=D^\eps$; 
in particular $X^\eps=\Gamma^\eps\backslash F^\eps$. 
Up to extraction of a subsequence, these fundamental domains converge to $D^0$ and $F^0$ (that we also denote $D$ and $F$) 
and $X^0=\Gamma^0\backslash F^0$ is a geometrically finite hyperbolic manifold 
(that we also denote $X$).

We define the parameters 
\begin{equation}
\label{parameter}
\la_j(\eps):=\frac{|p_{+j}(\eps)-p_{-j}(\eps)|}{\ell_j(\eps)}, \quad \nu_j(\eps):=\frac{\alpha_j(\eps)}{\ell_j(\eps)}.
\end{equation}
Then since $c_j(\eps)$ is smooth in $[0,1]$ and ${\rm Tr}(\gamma_j^\eps)=q_j(\eps)^\demi+q_j(\eps)^{-\demi}$, 
we see from \eqref{p_+-p_-} that if $\nu_j(\eps)$ is smooth, then $\la_j(\eps)$ extends smoothly in 
$\eps\in[0,1]$ and
\begin{equation}\label{limitlaj}
\la_j:=\lim_{\eps\to 0}\la_j(\eps)= (1+\nu_j^2)^\demi/|c_j|.
\end{equation}
In fact, 
if $\nu_j(\eps)$ is smooth, then by \eqref{fixedpoints} we have that $\frac{p_{+j}(\eps)-p_{-j}(\eps)}{\ell_j(\eps)}$ extends smoothly to $\eps=0$ with
\begin{equation}\label{convline}
\frac{p_{+j}(\eps)-p_{-j}(\eps)}{\ell_j(\eps)} \to (1+i\nu_j)/c_j \,\, \textrm{ as }\eps\to 0.
\end{equation}
\begin{defi}\label{qadmissible}
For a smooth family of multipliers $\eps\to q(\eps)\in C^\infty([0,1];\cc \setminus D(0,1))$ 
with $q(0)=1$ and $|q(\eps)|>1$ if $\eps>0$, we say that  $q(\eps)$ 
is \emph{admissible} if $q(\eps)=e^{\ell(\eps)(1+i\nu(\eps))}$ for some real valued 
functions $\ell(\eps),\nu(\eps)$ smooth in $\eps\in[0,1]$.
We say that a smooth family $\Gamma^\eps$ of classical Schottky groups of genus $g$ 
is admissible if each multiplier $q_j(\eps)$ of the generator $\gamma_j^\eps$
is either admissible or $\gamma_j=\gamma_j^0$ is loxodromic, and if  there is a subsequence $\eps_n\to 0$ 
for which there are $2g$ adapted circles $C_{\pm j}^{\eps_n}$ converging to $C_{\pm j}$
such that if two of the limiting $2g$ circles $C_{\pm j}$ intersect, this can only be 
$C_{+j}\cap C_{-j}=\{p_j\}$ for $j$ so that $\gamma_j$ is parabolic with fixed point $p_j$.
\end{defi}

\subsection{Canonical circles}
The adapted circles $C_{\pm j}^\eps$ associated to the elements $\gamma_j^\eps$ can actually be taken smoothly in $\eps>0$, but they are not in general of the best form to get a local model description of the geometry with respect to $\eps\to 0$. In addition it is not clear if they can be taken smoothly down to 
$\eps=0$, but we will show below that if the family of Schottky groups $\Gamma^\eps$ is admissible, then 
we can find a smooth family of fundamental domains down to $\eps=0$, which are bounded by 
pieces of circles near the punctures.

For this purpose and to obtain a nice description of the degeneration near the punctures, we define the notion of \emph{canonical circles} for a loxodromic transformation.

If $\gamma\in {\rm PSL}_2(\cc)$ is loxodromic with fixed points $p_-$ and $p_+$ and multiplier 
$q=e^{\ell(1+i\nu)}$ with $\ell>0$ and $\theta \circ \gamma \circ \theta^{-1}(z)=qz$ for some $\theta\in {\rm PSL}_2(\cc)$, the canonical circles for $\gamma$ will be the circles 
\begin{equation}\label{goodloxo} 
\begin{gathered}
\til{C}_{\pm}:=  \theta^{-1}(\{ z; |z|=e^{\pm \frac{\ell}2} \}) =\{z\in \cc; |z-z_{\pm}|=r\},\,\,\, \textrm{with } \\
z_{\mp}:= p_{\pm}+\frac{p_{\mp}-p_{\pm}}{1-e^{-\ell}},\quad 
r= \frac{|p_{+}-p_{-}|}{2\sinh(\ell/2)}.
\end{gathered}
\end{equation}

\begin{lem}\label{reasonforgoodcircle}
Let $\gamma\in {\rm PSL}_2(\cc)$ be loxodromic with multiplier $q=e^{\ell(1+i\nu)}$ and fixed points 
$p_-,p_+\in\cc$, and let $\til{C}_\mp$ be its associated canonical circles, defined by \eqref{goodloxo}.
Then the transformation $\gamma$ maps the exterior of the disk $\til{D}_{-}$ bounded by $\til{C}_{-}$ 
to the interior of the disk $\til{D}_{+}$ bounded by $\til{C}_{+}$.
\end{lem}
\begin{proof}
Consider now the two concentric circles $S_{\pm}:=\{|z|=e^{\pm\demi \ell}\}$ 
and let $m(q)$ be the complex dilation by $q$ in $\cc$. 
Consider the transformations 
\begin{equation}\label{phipsieta}
\phi(z)=\frac{z-1}{z+1}, \quad  \psi(z)=z+\frac{p_++p_-}{p_+-p_-},\quad 
\eta(z)=\frac{p_+-p_-(\eps)}{2} z.\end{equation} 
The composition $\eta\psi\phi$ maps $\{0,\infty\}$ to $\{p_-,p_+\}$ and 
$\gamma=(\eta\psi\phi) m(q)(\eta\psi\phi)^{-1}$. The circles $\til{C}_\pm$ are mapped by 
$\theta:=(\eta\psi\phi)^{-1}$ to $\{|z|=e^{\pm \ell/2}\}$ and the Lemma follows directly.
\end{proof} 

We shall denote by $\til{C}^\eps_{\pm j}$ the canonical circles of $\gamma_j^\eps$; a priori they are not adapted circles for the group. For $j>j_0$, assuming that $\nu_j(\eps)$ is smooth in $\eps\in[0,1]$, we see by \eqref{goodloxo} and \eqref{convline} that $\til{C}^\eps_{\pm j}$ extends smoothly to $\eps=0$ with limiting circles $\til{C}^0_{\pm j}$, tangent at $p_j$, with  radius $r_j$ and center $z_{\pm j}$ given by   
\[ z_{\pm j}= p_j\mp \frac{(1+i\nu_j)}{c_j}, \quad r_j=\la_j=\frac{\sqrt{1+\nu_j^2}}{|c_j|}.\] 
Note that the lines passing through the centers $z_{+j}$ and $z_{-j}$ intersect 
the axis of the parabolic transformation $\gamma_j$ at an angle $\arctan(\nu_j)$.

\subsection{Good fundamental domains}
For $\Gamma^\eps$ a family of admissible Schottky groups we have a subsequence 
of fundamental domains $F^{\eps_n}$ with totally geodesic boundary and with 
$D^{\eps_n}=\pl\hh^3\cap \pl F^{\eps_n}$ bounded by the adapted circles 
$C_{\pm j}^{\eps_n}$, and 
$F^{\eps_n}$ and $D^{\eps_n}$ are converging to $F^0$ and $D^0$, where $D^0$ is bounded by 
the limiting circles $C_{\pm j}$. From the limiting domain $F^0$,
we shall construct new fundamental domains $\til{F}^\eps$ for $\Gamma^\eps$ for small $\eps\geq 0$, 
called good fundamental domains and constructed by combining canonical circles with the limiting 
adapted circles $C_{\pm j}$. The domain $\til{D}^\eps=\pl \til{F}^\eps \cap \pl \hh^3$
will be bounded by Jordan curves instead of circles, but their form near the parabolic points $p_j$ will 
be a good model for the geometry as $\eps\to 0$ near the punctures. 

Notice that we can always choose $\delta_0>0$ and $\eps_0>0$ small enough  so that for each $j>j_0$ and $\delta\in (0,\delta_0)$, for all $0<\eps\le \eps_0$ the half-ball $B(p_j,\delta)\subset \hh^3$ is at positive Euclidean distance from 
all connected components of $\pl F^0\setminus \pl F^0\cap \pl \hh^3$ except those half-spheres supported by $C_{\pm j}$.  

Recall that $\til{C}_{\pm j}^\eps$ are the canonical circles of $\gamma_j^\eps$, and denote by $\til{D}_{\pm j}^\eps\subset \cc$ the disk bounded by $\til{C}_{\pm j}^\eps$. 
We then show the existence of good fundamental domains:
\begin{lem}\label{goodFD}
There exists $\delta \in (0,\delta_0)$ such that for all $\eps \in [0,\eps_0]$, 
there exist fundamental 
domains $\til{F}^\eps$ for $\Gamma^\eps$ acting on $\hh^3$ with the following properties:
\begin{itemize}
\item the boundary $\pl \widetilde{F}^\eps$ is a smooth in $\eps\in[0,\eps_0]$ collection of $2g$  smooth hypersurfaces  $(H_{\pm j}^\eps)_{j=1,\dots,g}$ homeomorphic to half-spheres: more precisely 
$H_{\pm j}^\eps$ is the image of $H_{\pm j}^0$
by a smooth in $\eps\in[0,\eps_0]$ family of diffeomorphisms of $\bbar{\hh^3}$ 
equal to ${\rm Id}$ at $\eps= 0$. 
The closures of $H_{\pm j}^\eps$ in $\bbar{\hh^3}$ are all disjoint except when $\eps=0$ 
where $\bbar{H^0_{-j}}\cap \bbar{H^{0}_{+j}}=\{p_j\}$ for $j>j_0$.
\item each $\gamma_j^\eps$ maps the exterior of the compact domain bounded by $H_{-j}^\eps$ 
in $\bbar{\hh^3}$ to the interior of the compact domain bounded by $H_{+j}^\eps$ 
in $\bbar{\hh^3}$.
\item For each $j>j_0$, 
$\til{F}^\eps \cap B(p_j,\delta)=\til{F}_j^\eps\cap B(p_j,\delta)$ 
if $\til{F}_j^\eps$ is the fundamental domain 
with totally geodesic boundary for the cyclic group $\cjg \gamma_j^\eps\cjd$ and
satisfying $\pl \til{F}_j^\eps\cap \pl \hh^3=\cc\setminus (\til{D}_{+j}^\eps\cup \til{D}_{-j}^\eps)$.
\end{itemize}
\end{lem}
\begin{proof} 
For $j\leq j_0$, we can take $H_{-j}^\eps=H_j^0$ to be the half-spheres supported by the circles $C_{-j}^0$ and 
$H_{+j}^\eps=\gamma_j^\eps(H_j^0)$. Since $\gamma_j^\eps\to \gamma_j$ and $C_{\pm j}$ are adapted circles for the limiting $\gamma_j$, clearly for small enough $\eps$ the hypersurfaces  $H_{\pm j}^\eps$ satisfy the desired properties. Now we deal with the more delicate part, that is when $j>j_0$ and the limiting $\gamma_j$ is parabolic.
We take $\delta$ small, but independent of $\eps\geq 0$, and take a large $C>0$ so that  
$B(p_j,C\delta)$ is at positive distance from 
the half-spheres supported by the limiting $C_{\pm k}$ at $\pl\hh^3$ for $k\not =j$. Note that $C$ can be taken large by taking $\delta$ small (for instance $C\simeq \delta^{-1/2}$ works).
We start with $\eps=0$, where we will modify $F^0$ to $\til{F}^0$ near the parabolic points $p_j$.
 We may assume that 
$C_{-j}^0\not=\til{C}_{-j}^0$ as otherwise it suffices to take $H_{\pm j}^\eps$ to be the half-spheres supported by $\til{D}_{\pm j}^\eps$, which satisfy the desired properties.
By conjugating  by $\phi: z\mapsto 1/(z-p_j)$ the parabolic element $\gamma_j$ becomes a parabolic transformation fixing $\infty$, thus of the form $z\mapsto z+c$ for some $c\in \cc$, 
which in $\hh^3$ acts by $T_c:(x,z)\to (x,z+c)$. 
The half-balls $B(p_j,\delta)$ and $B(p_j,C\delta)$ are mapped by the Poincar\'e extension $\Phi$ of $\phi$ to (the interior of) $\hh^3\setminus B(0,\delta^{-1})$ and $\hh^3\setminus B(0,(C\delta)^{-1})$. The circles $C_{-j}^0$ and $\til{C}_{-j}^0$ are mapped to lines $L$ and $\til{L}$ of $\cc$ with respective tangent vector $\tau\in \cc$ and $\til{\tau}\in\cc$,  
and $\phi(C_{+j}^0)$ and $\phi(\til{C}_{-j}^0)$ are images of these lines by $z\mapsto z+c$, that is $L+c$ and 
$\til{L}+c$. The half-spheres of $\hh^3$ supported by $C_{-j}^0$ and $\til{C}_{-j}^0$ are mapped to 
vertical planes $\rr^+\x L$ and $\rr^+\x \til{L}$ by $\Phi$, and the image of the half-spheres supported by 
$C_{+j}^0$ and $\til{C}_{+j}^0$ are $\rr^+\x (L+c)$ and $\rr^+\x (\til{L}+c)$. Note that 
$S:=\bbar{\Phi(F^0 \cap B(p_j,C\delta))}
\cap \pl\hh^3$ and $\til{S}:=\bbar{\Phi(\til{F}_j^0\cap B(p_j,C\delta))}\cap \pl\hh^3$ are strips in 
$\cc \setminus D(0,(C\delta)^{-1})$ bounded by 
$L$ and $L+c$ (resp. $\til{L}$ and $\til{L}+c$).
For the following part of the proof, we recommend the reader to see Figure~\ref{fig:FD} while reading the argument.

\begin{figure}
\centering
\def\svgwidth{25em}
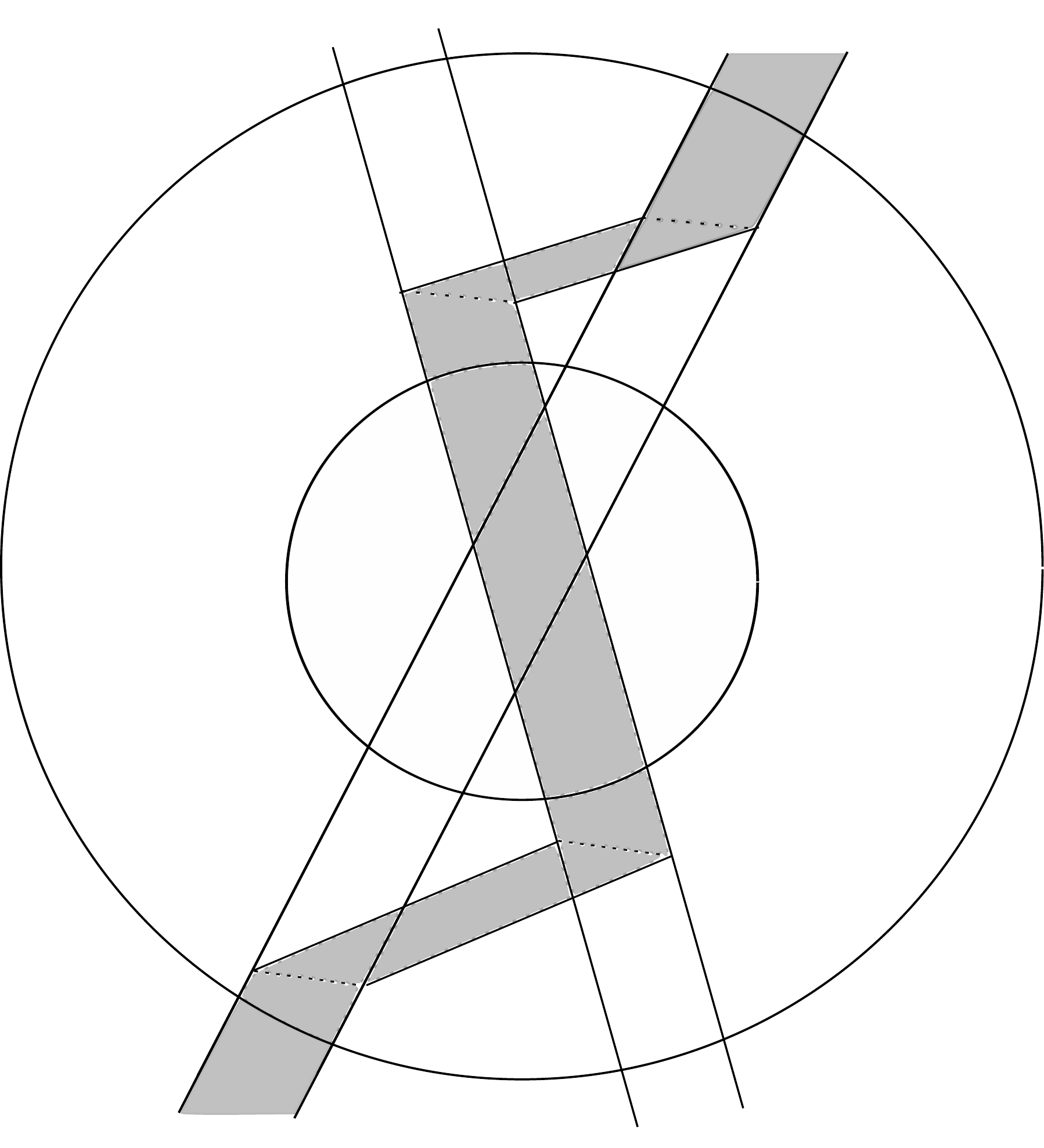
\caption{The new fundamental domain $\mc{D}$ in $\pl \hh^3\simeq \cc$ 
before smoothing is given by the dark region} 
\label{fig:FD}
\end{figure}
For $C>0$ large and $\delta>0$ small, consider the annulus 
$A_\delta:=\{(C\delta)^{-1}<|z|<\delta^{-1}\}$ in $\cc$ viewed as the boundary of the half-space $\hh^3$.  If $\delta>0$ is chosen very small, then in $A_\delta$ 
the strips bounded by $L$ and $L+c$ and the strips bounded by $\til{L}$ and $\til{L}+c$ are at a positive 
distance. We can then take two segments $T_1$ and $T_2$ in $A_\delta$ with extremities on $L$ 
and $\til{L}$, which are transverse to the lines with tangent vector $c\in\cc$. Then 
$P_i:=\cup_{t=\in[0,1]}(T_i+tc)$ for $i=1,2$ are two parallelograms with vertices on $L$, $L+c$, $\til{L}$, $\til{L}+c$.  
Then there is a unique fundamental domain $\mc{D}\subset \cc$ for the translation $z\to z+c$, 
with boundary made of two piecewise linear curves $Z$ and $Z+c$, 
with $Z$ containing $5$ segments, and such that $\mc{D}$ is equal to $\til{S}$ 
outside $|z|<\delta^{-1}$, to $S$ inside $|z|<(C\delta)^{-1}$, and contains 
the parallelograms $P_1$ and $P_2$. The two points of $\mc{D}$ at the largest distance from $\til{S}$ are 
vertices $v_1$ and $v_2$ of $P_1$ and $P_2$ (we choose $v_1$ to be the one on 
$L$), and there is a homotopy $h_t$ (for $t\in[0,1]$) between $\mc{D}$ and $\til{S}$ 
which can be done in the obvious way by moving $v_1$ along $L$ 
toward $v_1':=L\cap \til{L}$ and $v_2$ along $L+c$ toward $v_2':=(\til{L}+c)\cap(L+c)$
linearly in $t$. By choosing $C>0$ large enough, 
there exists a height $x_0\in ((C\delta)^{-1},\delta^{-1})$ so that 
in the half-space $\hh^3$, $\cup_{t\in]0,1]} (\{tx_0\}\x h_t(P_1\cup P_2))$ is contained in 
 $B(0,\delta^{-1})\setminus B(0,(C\delta)^{-1})$. 
We thus take the fundamental domain $\mc{F}\subset \rr^+_x\x \cc=\hh^3$ for the quotient 
$\cjg T_c\cjd\backslash \hh^3$ given by 
\[\mc{F}= \Big(\bigcup_{t\in(0,1]} (\{tx_0\}\x h_{t}(\mc{D}))\Big) \cup ( [x_0,+\infty)\x \til{S}).\]
This has a piecewise smooth boundary, and can be smoothed out by an arbitrarily small 
perturbation in $B(0,\delta^{-1})\setminus B(0,(C\delta)^{-1})$. For convenience we keep 
the same notation for the smoothed fundamental domain. 
By construction, $\Phi^{-1}(\mc{F})\cap B(p_j,C\delta)$ gives the desired modification 
of $F^0$ inside $B(p_j,C\delta)$ to produce $\til{F}^0$. This construction defines 
the hypersurfaces $H_{\pm j}^0$, which are the connected components of 
$\pl \til{F}^0\setminus \pa\til{F}^0\cap \pl\hh^3$.

Next we want to use a perturbation argument to construct $H_{\pm j}^\eps$ from $H_{\pm j}^0$. For each $j>j_0$, there exists a smooth family in $\eps\in [0,\eps_0]$ of M\"obius 
transformations $A_j^\eps\in {\rm PSL}_2(\cc)$ which map $\til{C}_{-j}^0$ onto $\til{C}_{-j}^\eps$. 
 It is just a composition of a translation and a dilation, and equals ${\rm Id}$ at $\eps=0$. 
Then $H_{-j}^\eps:= A_j^\eps H_{-j}^0$ is a smooth hypersurface and  
define $H^\eps_{+j}:=\gamma_j^\eps(H^\eps_{-j})$; both hypersurfaces are disjoint from other $H_{\pm k}^\eps$ for small $\eps$ since it is the case for $\eps=0$. 
The point $d_j/c_j\in\cc$ that is mapped to $\infty$ by  $\gamma_j$ is in the disk $D_{-j}$ bounded by $C_{-j}$, and since $A_j^\eps\to {\rm Id}$ as $\eps\to 0$, 
we see that for $\eps>0$ small enough $d_j(\eps)/c_j(\eps)$ is in the domain bounded by the curve 
$\pl H^\eps_{-j}\cap \pl\hh^3 \subset \cc$ and thus property 2) in the Lemma is satisfied for this choice of $H_{\pm j}^\eps$. By construction, in $B(p_j,\delta)$ the hypersurfaces $H_{\pm j}^\eps$ are given by pieces of half-spheres supported by $\til{C}^\eps_{\pm j}$, and  outside $B(p_j,\delta)$ they are arbitrarily close 
to $H_{\pm j}^0$ since $A_j^\eps\to {\rm Id}$ in $\mc{C}^k$-norms, thus we deduce that 
$H_{-j}^\eps\cap H_{+j}^\eps=\emptyset$ for $\epsilon>0$ small enough. These conditions ensure that the domain $\til{F}^\eps$ bounded by the hypersurfaces $H_{\pm j}^\eps$ is a fundamental domain for $\Gamma^\eps$ satisfying all the desired properties of the Lemma.
\end{proof}

\section{Analysis of the model degeneration} \label{analysis model}
In this section, we shall describe more precisely the model geometry for the degeneracy to a rank-$1$ cusp.
Let $\gamma_L\in {\rm PSL}_2(\cc)$ be loxodromic with multiplier $q=e^{\ell(1+i\nu)}$ and 
fixed points $p_-=0$ and $p_+=\la \ell$ for some $\la>0$; we write $L=(\ell,\nu,\la)$ and we take
\begin{equation}\label{regionforq}
L\in \mc{Q}:=(0,1]\x [-N,N]\x [N^{-1},N]
\end{equation} 
for some $N>0$ fixed. Using \eqref{convline}, the set of those $\gamma_L$ such that $L\in \mc{Q}$ has closure such that the boundary $\{\ell=0\}$ corresponds to parabolic elements 
\[ \gamma_{L}(z)=\frac{z}{cz+1} \textrm{ with }c=\frac{1+i\nu}{\la}, \,\, L=(0,\nu,\la)\]
fixing $p_-=p_+=0$.
We denote by $\bbar{\mc{Q}}$ the closure of $\mc{Q}$, and we define
the \emph{parabolic boundary} of $\mc{Q}$ as the set $\{\ell=0\}$. There is a smooth fibration 
\begin{equation}\label{fibrationPi}
\begin{gathered}
\Pi: \mc{X}\to \bbar{\mc{Q}}, \quad \textrm{with fibers the manifolds }\\ 
\Pi^{-1}(L)=\bbar{X}_L:=\cjg \gamma_L\cjd \backslash (\hh^3\cup \Omega_L)
\end{gathered}
\end{equation} 
where $\Omega_L=\pl\bbar{\hh^3}\setminus \{0,\la\ell\}$ the discontinuity set of the cyclic group
$\cjg \gamma_L\cjd$. A \emph{cusp region} of $\bbar{X}_L$ is the image of a neighborhood $B(0,\delta)$ 
of $0\in\pl\bbar{\hh^3}$ by the covering map $\pi_{\gamma_L}: (\hh^3\cup \Omega_L) \to \bbar{X}_L$ 
and we say that $\cup_{(0,\la,\nu)\in \bbar{\mc{Q}}}\pi_{\gamma_L}(B(0,\delta))$ is the \emph{cusp region} of 
$\mc{X}$.

If $|q|>1$, consider the isometry of the hyperbolic space 
$\hh^3=\rr^+_x\x \cc_z$
\begin{equation}\label{Lq}
m(q): (x,z)\mapsto (|q|x,qz)
\end{equation} 
and the quotient of $\hh^3$ by the elementary group $\cjg m(q)\cjd$ generated by $m(q)$ 
\begin{equation}\label{modelLq}
 X_{m(q)}:=\cjg m(q)\cjd \backslash \hh^3  \textrm{ with covering map } \pi_{m(q)}:  \hh^3 \to \cjg m(q)\cjd \backslash \hh^3.
\end{equation}

\begin{lem}\label{admissible}
For $L=(\ell,\nu,\la)\in \mc{Q}$, let $\gamma_L\in {\rm PSL}_2(\cc)$ be loxodromic with multiplier $q=e^{\ell(1+i\nu)}$ and fixed points 
$p_-=0$ and $p_+=\la \ell\in(0,\infty)$, and let $\til{C}^L_\pm$ be its associated canonical circles, defined by \eqref{goodloxo}.
Let $\til{D}^L_{\pm}\subset \cc$ be the disk  bounded by $\til{C}^L_{\pm}$ and $\til{F}_L\subset \hh^3$ the fundamental domain 
for the cyclic group $\cjg \gamma_L\cjd$ with totally geodesic boundary  satisfying
$\pl \til{F}_L\cap \pl \hh^3=\cc\setminus (\til{D}^L_{+}\cup \til{D}^L_{-})$.
Let $\pi_{\gamma_L}:\hh^3\to \cjg \gamma_L\cjd\backslash \hh^3$ 
denote the covering map, then  for $\delta>0$ small and for $\lambda \ell <\delta$, the set 
\begin{equation}\label{ULqk} 
\mc{U}_L^\delta:=\pi_{\gamma_L}(B(0,\delta)\cap \til{F}_L)
\end{equation}
is isometric to 
\begin{equation}\label{quotientLq} 
\pi_{m(q)}\Big(\{(x,z)\in \hh^3\setminus B(e,\rho); \,\, e^{-\demi\ell}
\leq \sqrt{x^2+|z|^2} \leq e^{\demi\ell}\}\Big)
\end{equation}
where  $e(L)\in \cc\subset\pl\hh^3$ and $\rho(L)>0$ 
have asymptotics for small $\ell$ 
\begin{equation}\label{ckrhok}
\begin{gathered}
e(L)=-1-\tfrac{\la^2\ell^2}{\delta^2}
+\mc{O}(\tfrac{\la^4\ell^4}{\delta^4}),\quad 
\rho(L)=\tfrac{\la\ell}{\delta}
+\mc{O}(\tfrac{\la^2\ell^2}{\delta^2}).
\end{gathered}\end{equation} 
The isometry from \eqref{ULqk} to \eqref{quotientLq} is given by 
\begin{equation}\label{Psij} 
\Theta_L(x,z):= \Big(\frac{x\la \ell}{|z-\la\ell|^2+x^2}, 
\frac{-x^2-|z|^2+\la \ell z}{|z-\la\ell|^2+x^2}\Big).
\end{equation}
\end{lem}
\begin{proof} We use the notations of the proof of Lemma \ref{reasonforgoodcircle}.
We have a composition $\eta\psi\phi$ which maps $\{0,\infty\}$ to $\{0,\la\ell\}$ and 
$\gamma=\eta\psi\phi \, m(q)(\eta\psi\phi)^{-1}$. We define 
$\Theta$ to be the Poincar\'e extension of $\theta=(\eta\psi\phi)^{-1}$ to the half-space 
$\hh^3$, thus given by \eqref{Psij}.
We check that the image of $B(0,\delta)$ under $(\eta\psi\phi)^{-1}$ is the complement of the half-ball 
$B(e,\rho)$ as claimed in the statement of the Lemma: 
 $\psi^{-1}\eta^{-1}$ maps $B(0,\delta)$ to the half-ball centered at $(x, z)=(0,-1)$ 
 and radius $2\delta/\la\ell$, then 
$\phi^{-1}$ maps to the complement of the half-ball with center and radius
\[(x,z)=(0,-1-\tfrac{\la^2\ell^2}
{\delta^2-\la^2\ell^2}),\quad  
\rho=\tfrac{\delta \la\ell}
{|\delta^2-\la^2\ell^2|} \]
which proves the claim.
\end{proof}

The following Proposition describes the model manifold $X_{m(q)}$ with more appropriate coordinates; the proof involves a sequence of tedious (and not very enlightning) computations, we thus have deferred its proof in the Appendix.
 \begin{prop}\label{model2} 
Assume that $L=(\ell,\nu,\la)\in \mc{Q}$ with the notation \eqref{regionforq}, then
there is an isometry $\Phi_L$ between the solid torus
\eqref{modelLq} and the manifold $(\rr/\tfrac{1}{2}\zz)_w\x \hh^2_{\zeta=v+iu}$ 
equipped with the metric
\begin{equation}\label{geps}
\begin{gathered}
g_L=  \frac{{du}^2+{dv}^2+((1+\nu^2)R^4-4\nu^2\ell^2u^2)dw^2
+ 2\nu(R^2-2u^2)dwdv+4\nu uv dudw}{u^2}\end{gathered} 
\end{equation}
where $R:=\sqrt{u^2+v^2+\ell^2}$. 
With $e(L),\rho(L)$ given by \eqref{ckrhok}, the neighborhood 
\begin{equation}\label{quotientLq_2} 
\pi_{m(q)}\Big(\{(x,z)\in \hh^3\setminus B(e,\rho); \,\, e^{-\demi\ell}
\leq \sqrt{x^2+|z|^2} \leq e^{\demi\ell}\}\Big)
\end{equation}
is mapped by $\Phi_L$ to the set
\begin{equation}\label{W}  
\mc{W}^\delta_{L}:=\pi_{w}\Big(\{ (w,\zeta)\in [-\tfrac{1}{4},\tfrac{1}{4})\x \hh^2;  
\, |\zeta-v_{L}(w)|<\tau_{L}(w)\}\Big)
\end{equation}
where $\pi_w: \rr\x \hh^2\to (\rr/\tfrac{1}{2}\zz)\x \hh^2$ is the covering map,
 and $\tau_{L}(w),v_{L}(w)$ are smooth functions of $w\in[-\frac{1}{4},\frac{1}{4})$ 
 which converge uniformly as $\ell\to 0$ to 
some $\tau_{\la,\nu}(w)$ and $v_{\la,\nu}(w)$ satisfying $v_{\la,\nu}(w)=\mc{O}(\delta^3)$ and 
$\tau_{\la,\nu}(w)=2\delta/\la+\mc{O}(\delta^3)$ uniformly in $|w|<1/4$. 
Finally,  the map 
\[(L,x,z)\mapsto (L,\Phi_L\circ \Theta_L(x,z))\in (\mc{Q}\x (\rr/\tfrac{1}{2}\zz)\x \bbar{\hh^2})\]
extends smoothly to a neighborhood of the cusp region of  $\mc{X}$ and is a diffeomorphism with image $V\setminus \{\ell=0, \zeta=0\}$ where $V$ is some neighborhood of $\{\ell=0, \zeta=0\}$.
\end{prop}

Notice that when $\ell\to 0$, the limiting model in Proposition \ref{model2} is $(\rr/\demi \zz)_w\x \hh^2_{\zeta=v+iu}$ equipped with the metric 
\begin{equation}\label{model1g0}
\begin{gathered}
g_0=  \frac{{du}^2+{dv}^2+(1+\nu^2)(u^2+v^2)^2dw^2+ 2\nu(v^2-u^2)dwdv+4\nu uv dudw}{u^2}\end{gathered}. 
\end{equation}
Writing $x:=\frac{u}{u^2+v^2},y:=-\frac{v}{u^2+v^2}$, this becomes
\[g_0=  \frac{{dx}^2+{dy}^2+(1+\nu^2)dw^2+ 2\nu dwdy}{x^2}\] 
and thus taking $(x',y',w')=(\frac{x}{\sqrt{1+\nu^2}},\frac{y}{1+\nu^2},w+\frac{\nu y}{1+\nu^2})$ and 
the inverted coordinates
$(u'=\frac{x'}{x'^2+y'^2},v'=-\frac{y'}{x'^2+y'^2},w')$, we obtain 
\begin{equation}\label{model2g0}
g_0=\frac{{dx'}^2+{dy'}^2+dw'^2}{x'^2}=\frac{{du'}^2+{dv'}^2+(u'^2+v'^2)^2dw'^2}{u'^2}
\end{equation}
which is exactly the model metric of \eqref{Uj'}. We can then write this change of variable  
\begin{equation}\label{uvtou'v'}
\begin{gathered}
u'=(1+\nu^2)^{3/2}u(1-\tfrac{\nu^2u^2}{u^2(1+\nu^2)+v^2}), \quad v'=(1+\nu^2)v(1-\tfrac{\nu^2u^2}{u^2(1+\nu^2)+v^2})\\
w'=w-\frac{\nu}{1+\nu^2}\frac{v}{v^2+u^2}
\end{gathered}
\end{equation}
and if we take the fundamental domain $[-\tfrac{1}{4},\tfrac{1}{4}]_w\x \hh^2_{iu+v}$ for $(\rr/\demi \zz)_w\x \hh^2$, 
we see that the corresponding fundamental domain in the coordinates $(u',v',w')$ for the action $w'\mapsto w'+\demi$ becomes 
\begin{equation}\label{mcD}
\mc{D}:=\Big\{(w', iu'+v')\in \rr \x \hh^2; w'+\tfrac{\nu}{1+\nu^2}\tfrac{v}{v^2+u^2}\in [-\tfrac{1}{4},\tfrac{1}{4}]\Big\}.\end{equation} 
This explicit isometry will be used later since it is sometime more convenient to work in the model \eqref{model2g0} than in the model \eqref{model1g0}.

The function $U:=\tfrac{u}{R}$ in $\mc{W}^\delta_{L}$ defines the boundary corresponding to 
$\pl \bbar{X}_L$ via $\Phi_L\circ\Theta_L$. 
We will see later that, near the cusp, this function is a boundary defining function on a space that compactifies $\mc{X}$ as a manifold with corners. 
This function will essentially give the form of the equidistant foliation near the pinched geodesic.  
\begin{lem}\label{lem:hyperb}
Let $U:=\tfrac{u}{R}$ be the chosen boundary defining function in $\mc{W}^\delta_{L}$, 
then the metric $h_{L}:=(1+\nu^2)(U^2g_L)|_{U=0}$ in the conformal infinity induced by the defining function $U\sqrt{1+\nu^2}$ 
is given by 
\[ \begin{split}
h_L:=(1+\nu^2)\Big( \frac{dv^2}{v^2+\ell^2}+(v^2+\ell^2)(1+\nu^2)dw^2
+2\nu dvdw\Big)\end{split}\]
and has constant Gaussian curvature $-1$.
\end{lem}
\begin{proof} First we notice that $h^{\flat}_L:=\tfrac{\ell^2}{(1+\nu^2)(v^2+\ell^2)}h_L$ is flat, since it is given by
\[h^\flat_L= d\theta^2+\ell^2(1+\nu^2)dw^2+2\ell\nu
d\theta dw\]
with $\theta:=\arctan(v/\ell)$, and thus the Gaussian curvature of $h_L$ is given by 
\[\frac{1}{2(1+\nu^2)} \frac{\ell^2}{v^2+\ell^2}\Delta_{h^\flat_{L}}\Big(\log\Big(\frac{v^2+\ell^2}{\ell^2}\Big)\Big)=
(\cos \theta)^2\pl_\theta^2(\log \cos(\theta))=-1\]
which finishes the proof.
\end{proof}

\section{Formation of a cusp on surfaces}\label{formationofcusp}

In this section, we discuss the uniformisation on a Riemann surface when there is a degeneration to a surface with cusps.

We start by setting the assumptions.
Let $N$ be a compact surface of genus $g\ge 2$ and $h_\eps$ a family of smooth metrics on $N$ for $\eps>0$. 
Assume that there is a finite family of disjoint smooth embedded circles $(H_j)_{j=1,\dots,j_1}$ on $N$ (for some $j_1\in \nn$) 
which satisfies the following properties: 
there exist $A>a>0$ and some connected open neighborhoods $\mc{Z}_j^\eps\subset (\rr/\tfrac{1}{2}\zz)\x (-A,A)$ 
of $(\rr/\tfrac{1}{2}\zz)\x\{0\}$ and some neighborhood $\mc{Y}_j^\eps$ of $H_j$ in $N$
such that $(\rr/\tfrac{1}{2}\zz)\x (-a,a) \subset \mc{Z}_j^\eps$, 
and there exist some  smooth diffeomorphisms
\[\psi_j^\eps : \mc{Z}_j^\eps \to \mc{Y}_j^\eps\]
and some parameters $\nu_j(\eps),\ell_j(\eps)$ converging to $\nu_j\in\rr$ and $0$ as $\eps\to 0$, such that 
\begin{equation}\label{heps} 
{\psi_j^\eps}^* h_{\eps}= (1+\nu_j(\eps)^2)\left(\tfrac{dv^2}{v^2+ \ell_j(\eps)^2} + 
(v^2+\ell_j(\eps)^2)(1+\nu_j(\eps)^2)dw^2+2\nu_j(\eps) dvdw\right),
\end{equation}
where $w\in (\rr/\tfrac{1}{2}\zz)$ is an angle variable and $v$ is the coordinate obtained 
by projecting on the second factor. Moreover, we ask that $\psi_j^\eps$ is converging in 
$\mc{C}^k$-norms for all $k\in\nn_0$ to some smooth diffeomorphisms 
$\psi_j^0: \mc{Z}_j^0\setminus ((\rr/\tfrac{1}{2}\zz)\x\{0\}) \to 
\mc{Y}_j^0 \setminus H_j$ where 
$\mc{Z}_j^0={\rm Int}(\cap_{\eps>0}\mc{Z}_j^\eps)$ 
and $\mc{Y}_j^0={\rm Int}(\cap_{\eps>0}\mc{Y}_j^\eps)$. We finally assume 
that the metric $h_{\eps}$ converges in $\mc{C}^k$-norms on compact sets of $M:=N\setminus H$ 
for all $k\in\nn_0$ to a smooth metric $h_{0}$ defined on $M $ where 
$H:=\bigcup_{j=1}^{j_0}H_j$.
Thus, for $\eps>0$, the metric $h_{\eps}$ is smooth on $N$, while for $\eps=0$, $h_{0}$ 
is a complete metric on $M$ of finite volume with cusp ends. 

Notice that near $H_j$ the metric \eqref{heps} can be rewritten under the more standard form 
\begin{equation}\label{isometrictomodel} 
\begin{gathered}
h_\eps= \tfrac{dv^2}{v^2+\ell_j(\eps)^2}+(v^2+\ell_j(\eps)^2){dw'}^2, \\ 
\textrm{ with }w':=w(1+\nu_j(\eps)^2)-\tfrac{\nu_j(\eps)}{\ell_j(\eps)}
\arccos\left(\tfrac{v}{\sqrt{v^2+\ell_j(\eps)^2}}\right)
\end{gathered}
\end{equation}
which shows that 
$H_j=\{v=0\}$ is a closed geodesic of length $\demi\ell_j(\eps)(1+\nu_j(\eps)^2)$ in this neighborhood. Since 
$\ell^{-1}\arccos(v/\sqrt{v^2+\ell^2})=-\int_{v}^\infty 1/(t^2+\ell^2)dt$, we see that for the limiting case $\ell_j(\eps)=0$, the change of coordinates above is only well defined (and smooth) in the region $\{v>0\}$. But changing $\arccos(v/\sqrt{v^2+\ell_j(\eps)^2})$ to $\arccos(-v/\sqrt{v^2+\ell_j(\eps)^2})$, we get a 
smooth change of coordinates at $\eps=0$ in $\{v<0\}$.
We use the model \eqref{heps} instead of \eqref{isometrictomodel} since it is more suitable to our $3$-dimensional model of Proposition \ref{model2} for the rank-$1$ cusp formation.  

We can compactify smoothly $M$ into $\bbar{M}$ by using $\psi_j^0$: it suffices to compactify the 
charts $\mc{W}_j^0\setminus ((\rr/\tfrac{1}{2} \zz)\x \{0\})$ made of two disjoint connected components
$\{v>0\}$ and $\{v<0\}$ by attaching a circle at $v=0$ on each  connected component 
and defining the smooth structure by saying that $v$ and $w$ are smooth functions.  
The obtained surface is a smooth surface with $2j_1$ boundary components and 
with interior given by $M$. It is important to notice that the isometry between \eqref{heps} and \eqref{isometrictomodel} at $\ell_j=0$ (ie. $\eps=0$) is not smooth at $v=0$ since $F_0(v)=-1/v$, thus the smooth 
compactification we take for $M$ using $\psi_j^0$ is not the same as the one used in the beginning of Section \ref{Sec:uniform}, which corresponds rather to compactifying by using the coordinates  $(w',v)$ putting metric under the form \eqref{isometrictomodel}.

By the uniformisation theorem, we can find for each $\eps>0$ a unique function 
$\varphi_{\eps}\in\CI(N)$ such that the conformal metric 
\[ h^{\rm hyp}_{\eps}=e^{2\varphi_{\eps}}h_{\eps}\] 
is hyperbolic.  Similarly, for $\eps=0$, Proposition \ref{uniform} insures that 
we can find $\varphi_0\in \CI(M)$ such that $h^{\rm hyp}_{0}= e^{2\varphi_0}h_{0}$ 
is a complete hyperbolic metric of finite volume with cusps on $M$. In fact, if  $(w',v)$ 
are the coordinates above putting the metric under the form \eqref{isometrictomodel}, 
Proposition \ref{uniform} shows that $\varphi_0(w',v)$ admits a smooth extension 
from each connected component of $\{v\not=0\}$ to both $\{v\geq 0\}$ and $\{v\leq 0\}$ 
(it is smooth from each side but not globally on an open interval containing $0$) with 
$\varphi_0|_{v=0}=0$ and $\pl_{w'}\varphi_0=\mc{O}(|v|^\infty)$ near $v=0$. Viewing now 
$\varphi_0$ as a function of $(w,v)$, we get 
$\varphi_0(w',v)=\varphi(w(1+\nu^2)-\frac{\nu}{v}, v)$ in $v>0$ and 
$\varphi_0(w',v)=\varphi(w(1+\nu^2)+\frac{\nu}{v}, v)$
in $v<0$, and we easily see that $\varphi_0$ admits a smooth 
extension to $\bbar{M}$ such that $\varphi_0|_{\pl\bbar{M}}=0$ and $\pl_w\varphi_0$ 
vanishes to infinite order at 
 $\pl\bbar{M}=\{v=0\}$.

\begin{prop}\label{cf.1}
Under the assumptions above, we have, as $\eps\to 0$,
$$
       \| \varphi_{\eps}-\varphi_0\|_{\cC^0} \to 0.
$$
\end{prop} 
\begin{proof}
Let $\tvarphi$ be a continuous function on $N\x [0,\eps_0)_\eps$ whose restriction to $\eps=0$ is given by $\varphi_0$
and such that $\tvarphi$ is smooth on $(N\x [0,\eps_0))\setminus (H\x \{0\})$. Moreover, we ask that  
\begin{equation}\label{plw}
\pl_w\tvarphi=\mc{O}((\ell_j(\eps)+|v|)^\infty) 
\end{equation} 
near $H_j\x \{0\}$; for instance this can be achieved by writing $\varphi_0=\varphi_{0,1}+\varphi_{0,2}$ with ${\rm supp}(\varphi_{0,1})\cap H=\emptyset$, and $\supp(\varphi_{0,2})\subset \cup_j\mc{Y}_j^\eps$ (where 
$\mc{Y}_j^\eps$ is the collar neighborhood with coordinates $v,w$ as above) and then taking $\tvarphi=\tvarphi_{1}+\tvarphi_{2}$ where $\tvarphi_{2}$ is supported in $\cup_j\mc{Y}_j^\eps$ and given in $\mc{Y}_j^\eps$
by 
\begin{equation}\label{tvarphi02}
\tvarphi_{2}(v,w,\eps)= \frac{1}{\ell_j(\eps)}\int \chi(\tfrac{v-v'}{\ell_j(\eps)})\varphi_{0,2}(v',w)dv'
\end{equation} 
where $\chi\in \mc{C}_0^\infty(\rr)$ satisfies $\int \chi=1$, $\chi\geq 0$ and $\chi(0)=1$. Using that, near $H_j$, 
$\pl_w\varphi_{0,2}=\pl_w\varphi_0=\mc{O}(|v|^\infty)$, we obtain the claim.
Consider the new family of metrics
$$
    \til{h}_{\eps}= e^{2\tvarphi} h_{\eps}, \quad \eps\in [0,\eps_0),
$$
and set $\tvarphi_{\eps}:= \varphi_{\eps}-\tvarphi(\cdot,\eps)$ so that $h^{\rm hyp}_{\eps}= e^{2\varphi_{\eps}}h_{\eps}= 
e^{2\tvarphi_{\eps}}\til{h}_{\eps}$.
Notice that $\tvarphi_0=0$ and that $R_{\til{h}_0}=-2$ where $R$ denotes the scalar curvature.  Thus, outside any fixed open set containing $H$, we will have that $R_{\til{h}_{\eps}}=-2+o(1)$ as $\eps\to 0$, by the fact that $h_\eps\to h_0$ on $M$ in $\mc{C}^k$-norms on compact sets of $M$.  On the other hand, near $H$, $R_{h_{\eps}}=-2$ by Lemma \ref{lem:hyperb}, so by the formula for the scalar curvature under conformal changes of metrics, we have that
$$
    R_{\til{h}_{\eps}}= e^{-2\tvarphi(\cdot,\eps)}( -2 + 2\Delta_{h_{\eps}}\tvarphi(\cdot,\eps)) \quad \mbox{near} \; H.
$$
The Laplacian $\Delta_{h_\eps}$ near $H_j$ is given by 
\[ \Delta_{h_\eps}=-\pl_v (v^2+\ell_j(\eps)^2)\pl_v -\frac{(1+\nu_j(\eps)^2)^{-1}}{v^2+\ell_j(\eps)^2}\pl_w^2+2\frac{\nu_j(\eps)}{1+\nu_j(\eps)^2}\pl_v\pl_w\]
therefore using \eqref{plw}, \eqref{tvarphi02} and the fact that $\varphi_0\in \mc{C}^\infty(\bbar{M})$, 
we deduce that $R_{\til{h}_{\eps}}$ converges uniformly to $R_{\til{h}_0}$ near $H$ as $\eps\to 0$ and thus $R_{\til{h}_\eps}=-2+o(1)$ uniformly.
In particular, for $\eps$ sufficiently small, $R_{\til{h}_{\eps}}$ will be negative.

Now, again by the formula for the curvature under conformal changes of metrics, we have that
\begin{equation}
     -2= e^{-2\tvarphi_{\eps}}( R_{\til{h}_{\eps}}+ 2\Delta_{\til{h}_{\eps}}\tvarphi_{\eps}).
\label{cf.2}\end{equation}
Thus, for $\eps>0$ sufficiently small so that $R_{\til{h}_{\eps}}$ is negative, 
we see that if $\tvarphi_{\eps}$ attains its maximum at $p$, then
$$
    -2\ge e^{-2\varphi_{\eps}(p)}R_{\til{h}_{\eps}(p)} \; \Longrightarrow \; e^{2\tvarphi_{\eps}(p)}\le \frac{R_{\til{h}_{\eps}}(p)}{-2} \; 
    \Longrightarrow \; \tvarphi_{\eps}(p) \le \frac12 \log \left( \frac{R_{\til{h}_{\eps}}(p)}{-2} \right)=o(1).
$$
Similarly, if $\tvarphi_{\eps}$ attains its minimum at $q$, then
$$
            \tvarphi_{\eps}(q)\ge \frac12 \log \left( \frac{R_{\til{h}_{\eps}}(q)}{-2}\right)=o(1).
$$
Consequently, $\tvarphi_{\eps}\to 0$ uniformly on $M$ as $\eps\to 0$.   Since 
$\varphi_{\eps}-\varphi_0= \tvarphi_{\eps} + \tvarphi(\cdot,\eps)-\varphi_0$
and $||\tvarphi(\cdot,\eps)-\varphi_0||_{L^\infty}=o(1)$ as $\eps\to 0$, the result follows.
\end{proof}

\begin{rem}
A recent result of Melrose-Zhu \cite{MelZhu} shows that in fact $\varphi_{\epsilon}$ admits a polyhomogeneous expansion on the manifold with corners obtained from $N\times [0,1)_{\eps}$ by blowing up  $H\times \{0\}$.
\end{rem}

The following corollary will be useful to deal with the limit of the renormalized volume under the formation of a rank-1 cusp.

\begin{cor} \label{cor:limitenergy}
Let $I_\eps\subset [-1,1]$ with size $|I_\eps|\to 0$, and let $e^{2\varphi_{\eps}}$ be the uniformisation factor for $h_\eps$ on $M$ 
so that $h^{\rm hyp}_\eps=e^{2\varphi_\eps}h_\eps$ is hyperbolic. We have in each collar neighborhood $\mc{C}_j$ of 
$H_j$,
\begin{equation}\label{cf.3}
\begin{gathered}
\lim_{\eps\to 0} \int_{\rr/\demi\zz}\int_{I_\eps} |d\varphi_{\eps}|^2_{h_{\eps}} dv dw =0, \\
 \lim_{\eps\to 0} \int_{M} |d\varphi_{\eps}|^2_{h_{\eps}} {\rm dvol}_{h_\eps} = 
\int_{M} |d\varphi_{0}|^2_{h_{0}} {\rm dvol}_{h_0}.\\
 \end{gathered}
 \end{equation}
\end{cor}  
\begin{proof}
Since $\varphi_0\in \mc{C}^\infty(\bbar{M})$ with $\pl_w\varphi_0\in\dot{\mc{C}}^\infty(\bbar{M})$, we see from the form of the 
metric $h_\eps$ in \eqref{heps}
that $|d\varphi_0|^2_{h_\eps}\in L^\infty$ with uniform bound with respect to $\eps$, and so
\begin{equation}
   \lim_{\eps\to 0} \int_{\rr/\demi\zz} \int_{I_\eps} |d\varphi_0|^2_{h_\eps} dv dw=0.
\label{cf.4}\end{equation}
On the other hand, we know that
\begin{equation}
      2\Delta_{h_{\eps}}\varphi_{\eps}= -2 e^{2\varphi_{\eps}} -R_{h_{\eps}}, \quad \mbox{for} \;
       \eps \in [0,\eps_0).
\label{cf.4a}\end{equation}
By the previous proposition, we therefore have that $\| \Delta_{h_{\eps}}\varphi_{\eps}-
\Delta_{h_0}\varphi_0\|_{\cC^0(M)}=o(1)$.  Moreover, the form of the metric \eqref{heps} and the fact that 
$\varphi_0\in \mc{C}^\infty(\bbar{M})$ and $\pl_w\varphi_0\in\dot{\mc{C}}^\infty(\bbar{M})$ imply that
 $\| \Delta_{h_{\eps}}\varphi_0 -\Delta_{h_0}\varphi_0\|_{\cC^0(M)}= o(1)$. Now
we combine these facts and use integration by parts to show that
\begin{equation}
  \int_{M} | d(\varphi_{\eps}-\varphi_0)|_{h_{\eps}}^2 {\rm dvol}_{h_{\eps}}= \int_{M} (\varphi_{\eps}\Delta_{h_{\eps}}
  \varphi_{\eps} + \varphi_0\Delta_{h_{\eps}}\varphi_0 -2 \varphi_0 \Delta_{h_{\eps}}\varphi_{\eps}){\rm dvol}_{h_\eps} =o(1).
\label{cf.3a}\end{equation}
The boundary terms at $H$ are $0$ by the properties of $h_\eps$ and $\varphi_\eps$.
In particular, as $\eps\to 0$
\begin{equation}
  \int_{\rr/\demi \zz} \int_{I_\eps} | d(\varphi_{\eps}-\varphi_0)|^2_{h_{\eps}} dv dw =o(1).
\label{cf.5}\end{equation}
The first result in the Corollary then follows by combining \eqref{cf.4} and \eqref{cf.5} 
and using the triangle inequality. 
The second result follow from \eqref{cf.3a} and using the fact that 
\begin{equation}
\begin{aligned}
   \lim_{\eps\to 0} \int_{M}  | d\varphi_0|^2_{h_{\eps}} {\rm dvol}_{h_\eps} &= \int_{M} | d\varphi_0|^2_{h_{0}}
   {\rm dvol}_{h_0}.
\end{aligned}
\end{equation}
This ends the proof.
\end{proof}

\section{The boundary defining function used to define the renormalized volume}\label{bdf.0}

In this Section, we analyze the geodesic boundary defining function corresponding to the hyperbolic representative in the conformal infinity when we have a family of convex co-compact hyperbolic $3$-manifold converging to a 
geometrically finite hyperbolic $3$-manifold with rank-$1$ cusps.

\subsection{Geometric assumptions on the family of metrics}\label{sec:assumptions} 
We fix a compact manifold with boundary $\bbar{{\bf X}}$ and a family of hyperbolic convex co-compact metrics 
$g_\eps$, with $\eps>0$, on the interior $X$ of $\bbar{{\bf X}}$. 
\begin{defi}\label{admisdeg}
We say that \emph{the family $g_\eps$ is an admissible degeneration of convex co-compact hyperbolic metrics on $X$} if $g_\eps$ are convex co-compact hyperbolic metrics satisfying  
the following properties (below, $\hh^2$ denotes the open upper half-plane in $\cc$ and $\bbar{\hh^2}$ the closed upper half-plane; we use the topology of $\cc$ to define bounded sets in $\bbar{\hh^2}$):

\textbf{Assumption 1 (Model near the cusp)}. There exists a family of $j_1$ disjoint simples curves $H_1,\dots,H_{j_1}$ in 
$\pl\bbar{{\bf X}}$, and disjoint open neighborhoods $\mc{U}_j^\eps\subset \bbar{{\bf X}}$ of $H_j$, 
there are diffeomorphisms 
$\Psi_j^\eps: \mc{W}^\eps_j\to \mc{U}^\eps_j$ where
$\mc{W}_j^\eps\subset (\rr/\demi\zz)_w\x \bbar{\hh^2_{\zeta}}$  are bounded open sets  
containing $(\rr/\demi\zz)\x \{\zeta\in\bbar{\hh^2}; |\zeta|<r_j\}$ for some $r_j>0$, and for $\zeta=v+iu$
\begin{equation}\label{metricmodel}
\begin{gathered} {\Psi_j^\eps}^*g_\eps= 
 \frac{{du}^2+{dv}^2+((1+\nu_j(\eps)^2)R^4-4\ell_j(\eps)^2\nu_j(\eps)^2u^2)dw^2}{u^2}\\
+\frac{2\nu_j(\eps)(R^2-2u^2)dwdv+4\nu_j(\eps)uv dudw}{u^2}
\end{gathered} 
\end{equation} 
for some $\ell_j(\eps)\to 0$ and $\nu_j(\eps)\to \nu_j\in\rr$ as $\eps\to 0$, with $R:=\sqrt{u^2+v^2+\ell_j(\eps)^2}$.

\textbf{Assumption 2 (Convergence outside the cusp)}. 
There exists a hyperbolic metric $g_0$ on $X$ such that  for any fixed boundary defining function $\rho\in \mc{C}^\infty(\bbar{{\bf X}})$, $\rho^2g_\eps\to \rho^2g_0$ in all $\mc{C}^k$-norms on compact sets of $\bbar{\bf{X}}\setminus \cup_{j}H_j$ as $\eps\to 0$. If $\mc{W}_j^0:={\rm Int}(\cap_{\eps>0}\mc{W}_j^\eps)\subset (\rr/\demi\zz)_w\x \bbar{\hh^2}$
and $\mc{U}_j^0:={\rm Int}(\cap_{\eps>0}\, \mc{U}_j^\eps)\subset \bbar{{\bf X}}$, then  $\Psi^\eps_j$ converge to a smooth diffeomorphism $\Psi_j^0:\mc{W}_j^0\setminus (\rr/\demi\zz)\x \{0\}\to \mc{U}_j^0\setminus H_j$ in all $\mc{C}^k$-norms.
\end{defi}

Under these assumptions, the metric $g_0$ has rank-$1$ cusps. This follows from the convergence of $\Psi_j^\eps,\mc{U}_j^\eps$, $\mc{W}_j^\eps$ and the fact that \eqref{metricmodel} has a limiting metric as 
$\eps\to 0$ which is isometric to a neighborhood \eqref{Uj'} of a rank-$1$ cusp.
The \emph{degenerating curve} $H\subset \pl\bbar{X}$ is the submanifold given by $H:=\cup_{j=1}^{j_1}H_j$.
\begin{prop}\label{SchottkyOK}
Let $\Gamma^\eps\subset {\rm PSL}_{2}(\cc)$ be an admissible family of classical Schottky groups of genus ${\bf g}$ in the sense of Definition \ref{qadmissible}. Then for each $\eps>0$, $X^\eps:=\Gamma^\eps\backslash \hh^3$ 
is isometric to $(X,g_\eps)$ where $X$ is the interior of a solid torus of genus ${\bf g}$ and $g_\eps$ is an admissible degeneration of convex co-compact hyperbolic metrics on $X$ in the sense of Definition \ref{admisdeg}.
\end{prop}
\begin{proof} 
We can write the hyperbolic manifold as 
$X^\eps=\Gamma^\eps\backslash \til{F}^\eps$ where $\til{F}^\eps$ 
are good fundamental domains constructed in Lemma \ref{goodFD}. 
The metric on $X^\eps$ is the hyperbolic metric $g_{\hh^3}$ on $\til{F}^\eps$, 
which descends smoothly to the quotient by $\Gamma^\eps$. In fact, we can also consider the closure  
$\bbar{X^\eps}$ obtained from the action of $\Gamma^\eps$ on the closure of 
$\til{F}^\eps$ in $\hh^3\cup\Omega^\eps$ where $\Omega^\eps\subset \mathbb{S}^2$ is the set of discontinuity of $\Gamma^\eps$.   
 These can be  put together into a smooth fibration
\begin{equation}
    \Pi: \mathcal{X}\to [0,1]
\label{fibre.1}\end{equation}
such that $\Pi^{-1}(\eps)= \bbar{X^{\eps}}$ has interior equipped with the complete hyperbolic metric $\hat{g}_{\eps}$ induced from $g_{\hh^3}$.  For $\eps>0$, $X^{\eps}$ is naturally the interior of a solid torus $\bbar{X^{\eps}}$ of genus ${\bf g}$, while  when $\eps=0$, there are cusps of rank 1. So as we have seen in Section~\ref{rv.1}, the conformal compactification is no longer a solid torus, it is a solid torus with a circle removed for each rank-1 cusp.  In fact, by Lemma \ref{admissible} and Proposition \ref{model2}, 
for each cusp point $p_j$, we have an isometry 
$\Psi_j^\eps:=(\Phi_j^\eps\circ \Theta_j^\eps)^{-1}$ from a neighborhood of $\zeta=0$ in $(\rr/\tfrac{1}{2}\zz)_w\x \hh^2_{\zeta}$ to a neighborhood of $p_j$ in $\til{F}^\eps$, where $\Theta_j^\eps=\Theta_{L(\gamma_j^\eps)}$ 
is given by  
\eqref{Psij} and $\Phi_j^\eps=\Phi_{L(\gamma_j^\eps)}$ is given by  Proposition \ref{model2}
with $L(\gamma_j^\eps)=(\ell_j(\eps),\nu_j(\eps),\la_j(\eps))$ smooth in $\eps \in [0,1]$ (in Section \ref{analysis model} we take the fixed points 
$p_{-j}^\eps=0$ and $p_{+j}^\eps=\la_j(\eps) \ell_j(\eps)$ but we can always reduce to this case by composing with a smooth family of translations are rotations). 
Moreover, these combine to give a smooth diffeomorphism $\Psi_j: \mathcal{W}_j \to \mathcal{U}_j$ from a neighborhood $\mathcal{W}_j$ of $\{\zeta=0,\eps=0\}$ in $(\rr/\tfrac{1}{2}\zz)_w\x \bbar{\hh^2_\zeta}\times [0,1]_{\eps}\setminus \{\zeta=\eps=0\}$ into a neighborhood $\mathcal{U}_j \subset \mathcal{X}$ of the cusp point $p_j$ in $\Pi^{-1}(0)\subset \mathcal{X}$. This follows from the last statement of Proposition \ref{model2}.

The diffeomorphisms $\Psi_j$ give us a natural way to compactify uniformly down to $\eps=0$ by simply replacing $\mathcal{W}_j$ by its closure $\bbar{\mathcal{W}}_j$ in $(\rr/\tfrac{1}{2}\zz)_w\x \bbar{\hh^2_\zeta}\times [0,1]_{\eps}$.  Indeed, we can consider a compactification 
\begin{equation}
 \bbar \Pi: \bbar{\mathcal{X}} \to [0,1]\label{fibre.2}\end{equation}
of \eqref{fibre.1} such that $\bbar{\Pi}^{-1}(\eps)= \bbar{X^{\eps}}$ for $\eps>0$ and $\bbar{\Psi}_j: \bbar{\mathcal{W}}_j\to \bbar{\mathcal{U}}_j$, which restrict to $\Psi_j$ on $\mathcal{W}_j$, is a diffeomorphism from $\bbar{\mathcal{W}}_j$ to a neighborhood $\bbar{\mathcal{U}}_j$ of the circle $H_j\subset {\bbar{\Pi}^{-1}(0)}$ corresponding to the cusp point $p_j$.  Here, $\bbar{\mathcal{X}}$ is now a compact manifold with corners and $\bbar{\Pi}$ is a surjective submersion.  Moreover, the fibres of $\bbar{\Pi}$ are manifolds with boundary, more precisely solid tori of genus ${\bf g}$.  Choosing a horizontal connection for \eqref{fibre.2}, we can then use parallel transport to obtain a commutative diagram
\begin{equation}
\xymatrix{
    \bbar{\mathcal{X}}  \ar[r]^-{G} \ar[dr]^{\bbar{\Pi}} &  \bbar{X}\times [0,1] \ar[d]^{\operatorname{pr}_2}\\
                     & [0,1] 
}
\label{fibre.3}\end{equation}
where $\operatorname{pr}_2: \bbar{X}\times [0,1]\to [0,1]$ is the projection on the second factor,  $\bbar X$ is a fixed manifold with boundary and $G$ is a diffeomorphism of manifolds with corners.  In the statement of Proposition~\ref{SchottkyOK}, it suffices then to take $X$ to be the interior of $\bbar X$ with family of metrics $g_{\eps}=  G_* \hat{g}_{\eps}$ on the slices $\operatorname{pr}_2^{-1}(\eps)$
where $\hat{g}_{\eps}$ is the induced family of hyperbolic metrics on the fibres of $\Pi: \mathcal{X}\to [0,1]$.  The family of diffeomorphisms associated to each cusp point $p_j$ in Definition~\ref{admisdeg} can then be taken to be $G \circ \Psi_j(\cdot,\eps)$ for $\eps\ge 0$.
\end{proof}

\subsection{The Hamilton-Jacobi equation outside the cusps}\label{outsidecusp}
We consider an admissible degenerating family of convex co-compact metrics $g_\eps$ on $X={\rm Int}(\bbar{{\bf X}})$ in the sense of Definition \ref{admisdeg} and we keep the notations of Section \ref{sec:assumptions}. The manifold $(X,g_0)$ is geometrically finite hyperbolic with cusps of rank-$1$ and 
$\bbar{X}=\bbar{{\bf X}}\setminus H$ where $H$ is the degenerating curve in the boundary of $\bbar{{\bf X}}$. Recall that $\mc{K}$ is a compact subset of $\pl\bbar{X}$ where Assumption 2 is satisfied in Definition \ref{admisdeg}. 
Let $h_0^{\rm hyp}$ be the uniformizing metric on the conformal boundary 
$M=\pl\bbar{{\bf X}}\setminus H=\pl\bbar{X}$, given by Proposition \ref{uniform}; it is a complete hyperbolic metric with finite volume.
We define $\rho_0$ to be the geodesic boundary defining function of $M$ in $\mc{K}$ near $\pl\bbar{X}$ to be the solution of the Hamilton-Jacobi equation
\[ 
\left|\frac{d\rho_0}{\rho_0}\right|_{g_0}^2=1, \quad (\rho_0^2g_0)|_{TM}=h_0^{\rm hyp}.
\]
The equation is non-characteristic at $M\cap \mc{K}$ and has a unique solution near $M\cap\mc{K}$, just as in the convex co-compact case (see the discussion of Section \ref{convcocomp}).
We first want to define a geodesic boundary defining function for $g_\eps$ by the equation 
\begin{equation}\label{hamjab0}  
\left|\frac{d\hat{\rho}_\eps}{\hat{\rho}_\eps}\right|_{g_{\eps}}^2=1, \quad \hat{\omega}_\eps|_{\rho=0}=0
\end{equation}
where $\hat{\rho}_\eps= e^{\hat{\omega}_\eps}\rho_0$; notice that $\hat{\omega}_0=0$. We first show 
\begin{lem}\label{lemhatomega}
There exists $\delta>0$ such that for all $\eps\geq 0$, the Hamilton-Jacobi equation \eqref{hamjab0} has a solution $\hat{\omega}_\eps$ in 
$\mc{K}\cap \{\rho_0<\delta\}$ and $\hat{\omega}_\eps$ converges to $0$ in $\mc{C}^k$-norms there for all $k$.
\end{lem}
\begin{proof} The equation can be written as  
\[ 2\cjg d\hat{\omega}_\eps,d\rho_0\cjd_{\bar{g}_\eps}+\rho_0|d\hat{\omega}_\eps|^2_{\bar{g}_\eps}=\frac{1-|d\rho_0^2|_{\bar{g}_\eps}}{\rho_0}, 
\textrm{ with boundary condition }\hat{\omega}_\eps|_{\rho_0=0}=0,\]
where $\bar{g}_\eps:=\rho_0^2g_\eps$ converges in $\mc{C}^\infty(\mc{K};S^2T^*\bbar{X})$ to $\rho^2g_0$ as 
$\eps\to 0$. 
This is a uniform family (with respect to $\eps$) of non-characteristic Hamiton-Jacobi equations, which converge in $\mc{C}^\infty(\mc{K})$ to a non-characteristic Hamiton-Jacobi equation as $\eps\to 0$. This is solved by the method of characteristics and thus it admits a solution in a uniform neighborhood of $\rho_0=0$, converging 
smoothly to $\hat{\omega}_0=0$ as $\eps\to 0$.
\end{proof}
Notice that $\hat{\rho}_\eps$ is not exactly the geodesic boundary function that we would need to compute 
the renormalized volume but we will see later that the renormalized volume there can be expressed easily in terms 
of this boundary defining function. The function we are interested in is 
\begin{equation}\label{rhoepsinK}
\rho_\eps=e^{\omega_\eps}\hat{\rho}_\eps
\end{equation} 
where $\omega_\eps$ is the solution of 
\[\left|\frac{d\rho_\eps}{\rho_\eps}\right|_{g_{\eps}}^2=1, \quad \omega_\eps|_{\rho_0=0}=\varphi_\eps\]
and $\varphi_\eps$ is the uniformization factor such that $h_\eps^{\rm hyp}:=e^{2\varphi_\eps}h_\eps$ is hyperbolic if $h_\eps:=(\rho_0^2g_\eps)|_{\rho_0=0}$;
The Hamilton-Jacobi equation \eqref{rhoepsinK} has a unique solution in $\mc{K}$ near $M$
and in particular one has $\omega_0|_{\mc{K}\cap M}=\varphi_0=0$.

\subsection{The Hamilton-Jacobi equation near the cusps}\label{hamilton}
In the model from Proposition~\ref{model2}, which is a neighborhood of $\{\zeta=0\}$ in $(\rr/\tfrac{1}{2}\zz)_w\x \hh^2_{\zeta=v+iu}$, 
it will be useful to forget the $\eps$ parameter and consider now $(\ell(\eps),\nu(\eps))$ as independent parameters $(\ell,\nu)$ and 
we shall study the geodesic boundary defining function as functions of $(\ell,\nu)$ where 
$\nu\in\rr$ is a bounded parameter and $\ell \in[0,\ell_0]$ for some fixed small $\ell_0>0$. 
We view $\nu$ as a parameter moving in a compact set, and the metric has a uniform behavior in terms of $\nu$ in this set, 
and for this reason we shall not emphasize the dependence in $\nu$ in the notations.
The metric $g_\eps$ of \eqref{geps} will be rewritten as 
\begin{equation}\label{gellnu}
\begin{gathered}
g_{\ell}=  \frac{{du}^2+{dv}^2+((1+\nu^2)R^4-4\ell^2\nu^2u^2)dw^2
+ 2\nu(R^2-2u^2)dwdv+4\nu uv dudw}{u^2}
\end{gathered} \end{equation}
with $R=\sqrt{u^2+v^2+\ell^2}$.

We thus consider for the moment just a neighborhood of cusps, that is we set  
\[ \bbar{\mc{U}} := \{ (w, u,v)\in (\rr/\tfrac{1}{2}\zz) \x [0,1)\x\rr ;  u^2+v^2<1\}.\]
and $\mc{U}$ its interior.
Consider the submanifold 
$H:=\{u=v=\ell=0\}\subset \bX\times [0,\ell_0)$ which corresponds to the cusp, and let $\bbar{\mc{U}}_\ell$ be the blow-up
of $\bbar{\mc{U}}\times [0,\ell_0)$ at $H$, defined to be 
\[\bbar{\mc{U}}_\ell = (\bbar{\mc{U}}\x[0,\ell_0)\setminus H)\sqcup SH 
\]
where $SH\subset T(\bX\x[0,\ell_0))|_H$ is the normal inward pointing spherical bundle of $H$. 
There is a blow-down  map $\beta :\bbar{\mc{U}}_\ell \to \bbar{\mc{U}} \x[0,\ell_0)$, which is the identity outside $SH$ and the projection 
$SH\to H$ on the base when restricted to $SH$.   The set $\bbar{\mc{U}}$
has a natural structure of smooth manifold with corners of codimension $2$ in a way that 
the functions $u,v,\ell,R$ lift by $\beta$ to smooth functions on $\bbar{\mc{U}}_\ell$; we will use the same notations for these functions and for their lifts to $\bbar{\mc{U}}_\ell$.
There are three boundary hypersurfaces in $\bbar{\mc{U}}_\ell$: the face denoted $\mc{F}_{\ell}$ whose interior is diffeomorphic 
to $\{\ell=0,u\not=0\} \subset  \bbar{\mc{U}}\x[0,\ell_0)$, the face denoted $\mc{F}_{u}$ whose interior is diffeomorphic 
to $\{u=0,\ell\not=0\} \subset  \bbar{\mc{U}}\x[0,\ell_0)$,  and the front face $\mc{F}_{R}=SH$ given by the equation $R=0$. See Figure~\ref{ul.1}. We notice that $\mc{F}_\ell$ is naturally diffeomorphic to a neighborhood of $\cf$ in 
the manifold $\bbar{X}_c$ defined in Section \ref{Sec:uniform}, with $\cf$ identified with $\mc{F}_R\cap\mc{F}_\ell$, thus studying what happens on $\mc{F}_\ell$ is equivalent to consider 
a neighborhood of the cusp in $\bbar{X}_c$.

\begin{figure}
\setlength{\unitlength}{0.7cm}
\begin{picture}(6,6)
\thicklines
\put(3,2){\oval(2,2)[t]}
\put(2.5,2.2){$\mc{F}_{R}$}
\qbezier(2,2)(3,1.5)(4,2)

\put(2.4,5.8){$u$}

\put(4,2){\vector(1,0){2}}
\put(0,2){\line(1,0){2}}
\put(6,2.2){$v$}
\put(2.8,1.7){\vector(-1,-2){1}}
\put(0.2,0.2){$\mc{F}_u$}
\put(0.2,3){$\mc{F}_{\ell}$}

\put(3,3){\vector(0,1){3}}
\put(2.2,-0.4){$\ell$}

\end{picture}
\caption{The manifold with corners $\bbar{\mc{U}}_\ell$}\label{ul.1}
\end{figure}

Consider the following smooth variables on $\bbar{\mc{U}}_\ell$:
\begin{equation}
   (U= \frac{u}R,  v , w, \ell).
\label{bdf.2}\end{equation}    
They provide coordinates outside $\mc{F}_R$. In fact, when restricted to $\mc{F}_\ell$, $(U,v,w)$ provides smooth coordinates on $\mc{F}_\ell$ near the corner $\mc{F}_R\cap\mc{F}_u$, with $U$ being a smooth defining function for $\mc{F}_u$ and $v$ being a smooth defining function for $\mc{F}_R$.
We will also sometime use the smooth variable 
\[V=\frac{v}{R}= \pm\sqrt{1-U^2-\frac{\ell^2}{R^2}}\] 
on $\bbar{\mc{U}}_\ell$. Then we have 
$$
\frac{du}u = 
\frac{dU}{U(1-U^2)}+ \frac{Vdv}{R(1-U^2)}.
$$
Hence, we see that the dual vector fields $u\pl_u$ and $\pl_v$ to $du/u$ and $dv$ become in the coordinates $(U,v,w)$  
\begin{equation}\label{uplu}
 u\pl_u \to  U(1-U^2)\pl_U, \quad \pl_v\to  \pl_v-\frac{VU}{R}\pl_U.
 \end{equation}
In terms of the variables $U,v,w$, the  metrics $g_{\ell}$ is given by
\begin{equation}
\begin{aligned}
g_{\ell} = &\frac{dU^2}{U^2(1-U^2)^2}+ \frac{2vdUdv}{U(1-U^2)(v^2+\ell^2)}+ 
\frac{v^2dv^2}{(v^2+\ell^2)^2} + \frac{(1-U^2)dv^2}{U^2(v^2+\ell^2)}+
\frac{4\nu v}{U(1-U^2)} dUdw\\
& + \Big((1+\nu^2)\frac{(v^2+\ell^2)}{U^2(1-U^2)}-4\ell^2\nu^2\Big)dw^2+2\nu
\Big(\frac{2 v^2}{v^2+\ell^2} +\frac{(1-2U^2)}{U^2}\Big)dwdv.
\end{aligned}
\label{bdf.4}\end{equation}
In particular, looking at the conformal family of metrics $\overline{g}_{\ell}= U^2 g_{\ell}$, we see that when pulling-back to $\{U=0\}=\mc{F}_u$ this metric, one has
\begin{equation}\label{hell}
h_\ell:= \left.\overline{g}_{\ell}\right|_{U=0}= \frac{dv^2}{v^2+\ell^2} + 
(1+\nu^2)(v^2+\ell^2)dw^2+2\nu dvdw,  
\end{equation}
which corresponds to the model \eqref{heps} for the formation of a cusp obtained by pinching a closed geodesic. 
In general, as described in the previous section, the global hyperbolic representative of this conformal infinity will be slightly different, of the form
\begin{equation}
    h^{\rm hyp}_{\ell}= e^{2\varphi_{\ell}} h_\ell
\label{bdf.5}\end{equation}
for some family of smooth functions $\varphi_{\ell}$, which is obtained by uniformisation and has the properties of Proposition \ref{uniform} for the case $\ell=0$.  By Proposition \ref{cf.1} (with $\ell$ playing the role of $\eps$ here), the uniformising factor $\varphi_\ell$ will tend to $\varphi_0$ as $\ell\to 0$ on the interior of $\mc{F}_u$. 
Since we want to work in a more general setting than the uniformized metric, we now just fix an arbitrary family of smooth functions $\varphi_\ell$ so that $\varphi_\ell\to \varphi_0$ on $\mc{F}_u$ as $\ell\to 0$ with the requirement that $\varphi_0$ satisfies the properties of Proposition \ref{uniform}, ie. it extends smoothly to the closure of $\mc{F}_u$ in $\mc{F}_\ell$ and $\pl_w\varphi_0$ vanishes to infinite order at $\mc{F}_R=\{v=U=0\}$ in $\mc{F}_\ell\cap\mc{F}_u$. 

To define the renormalized volume associated to such a choice of representative in the conformal class at the boundary, we need to construct a boundary defining function $\rho_\ell$ of the face $\mc{F}_u=\{U=0\}$ such that
\begin{equation}\begin{split}\label{hamjab}
   (1) & \quad \rho_\ell= e^{\omega_\ell}U \textrm{ for some function } \omega_\ell \textrm{ satisfying }\left. \omega_\ell\right|_{U=0}=\varphi_{\ell}\\
    (2) & \quad \left|  \frac{d\rho_\ell}{\rho_\ell}\right|_{g_{\ell}}^2=1.
\end{split}\end{equation}
We solve this first order differential equation near the face $\mc{F}_R$. This is an equation of Hamilton-Jacobi type which we write 
explicitly in terms of the coordinates $U,v,w$.  First, in the coordinates 
$u,v,w$ and in matrix form, the family of dual metrics $g^{-1}_{\ell}$ on the cotangent space
is given by
\begin{equation}\label{gellinv}
    g_{\ell}^{-1}= \frac{u^2}{R^4}\left( \begin{array}{ccc}
    R^4+4\nu^2u^2v^2 & 2\nu^2uv(R^2-2u^2) & -2\nu uv\\
    2\nu^2uv(R^2-2u^2)  &  R^4+\nu^2(R^2-2u^2)^2  & -\nu(R^2-2u^2)   \\
    -2\nu uv & -\nu(R^2-2u^2) &1        
     \end{array} 
     \right).
\end{equation}
Since $\tfrac{d\rho_\ell}{\rho_\ell}= \frac{(R^2-u^2)du}{uR^2} -\tfrac{vdv}{R^2}+ d\omega_\ell$, the equation $|\tfrac{d\rho_\ell}{\rho_\ell}|^2_{g_\ell}=1$ becomes
\begin{equation}\label{HJ}
\begin{gathered}
     \frac{(1-U^2)^2}{u^2} \left| du \right|^2_{g_{\ell}}+ |d\omega_\ell |^2_{g_{\ell}}+\frac{v^2}{R^4}|dv |^2_{g_{\ell}}
   +  2(1-U^2)\langle \frac{du}{u}, d\omega_\ell\rangle_{g_{\ell}} \\
   -   \frac{2v(1-U^2)}{R^2}\langle dv, \frac{du}{u}\rangle_{g_{\ell}}-\frac{2v}{R^2}\cjg dv,d\omega_\ell\cjd_{g_\ell} =1. 
\end{gathered}
\end{equation}
Now, we compute (recall that $V=v/R$) 
\begin{equation}\label{compute1}
\begin{gathered}
 \frac{(1-U^2)^2}{u^2} \left| du \right|^2_{g_{\ell}}+\frac{v^2}{R^4}|dv |^2_{g_{\ell}} -   \frac{2v(1-U^2)}{R^2}\langle dv, \frac{du}{u}\rangle_{g_{\ell}} =\\ 
(1-U^2)^2+V^2U^2\Big( 4\nu^2(1-U^2)^2+1+\nu^2(1-2U^2)^2-4\nu^2(1-U^2)(1-2U^2) \Big),
  \end{gathered}   
\end{equation}
and
\begin{equation}\label{compute2}
\begin{split}
\langle \frac{du}{u}, d\omega_\ell\rangle_{g_{\ell}} = & 
(1+4\nu^2U^2V^2)(u\pl_u \omega_\ell) + 2\nu^2 U^2(1-2U^2)VR \pl_v\omega_\ell
 -2\nu U^2\frac{V}{R}\pl_w\omega_\ell,\\
\langle dv, d\omega_\ell\rangle_{g_{\ell}}=& -2\nu^2U^2(1-2U^2)VR(u\pl_u\omega_\ell)+U^2(1+\nu^2(1-2U^2)^2)R^2\pl_v\omega_\ell \\
& -U^2\nu(1-2U^2)\pl_w\omega_\ell,
\end{split}\end{equation}
and from \eqref{gellinv},  $|d\omega_\ell|^2_{g_\ell}$ is of the form 
\[ |d\omega_\ell|^2_{g_\ell}=(u\pl_u\omega_\ell)^2 + \frac{U^2}{R^2}|\pl_w\omega_\ell|^2 +\nu U^2P_0\Big(U^2,V; u\pl_u\omega_\ell,R\pl_v \omega_\ell,\frac{1}{R}\pl_w\omega_\ell\Big)\]
for some polynomial $P_0(x,y; X,Y,Z)$, which is quadratic in $(X,Y,Z)$ with 
coefficients which are polynomial functions in $(x,y)$, independent of $\ell$ and depending smoothly on $\nu$.
Gathering these computations with \eqref{HJ}, we  obtain that $|d\rho_\ell/\rho_\ell|_{g_\ell}=1$ is equivalent to 
\[\begin{gathered} 
2(1-U^2+2\nu^2V^2U^2(3-4U^2))(u\pl_u\omega_\ell)+U^2VQ_1(U^2)(R\pl_v\omega_\ell)+\nu U^2VQ_2(U^2)(\frac{1}{R}\pl_w\omega_\ell)\\
+(u\pl_u\omega_\ell)^2 +\frac{U^2}{R^2}|\pl_w\omega_\ell|^2+ 
\nu U^2P_0(U^2,V; u\pl_u\omega_\ell,R\pl_v \omega_\ell,\frac{1}{R}\pl_w\omega_\ell)=U^2Q_3(U^2,V)
\end{gathered}\]
where $Q_i$ are polynomials, and thus using  \eqref{uplu} and dividing by $U$  
we get an 
equation of the form 
\begin{equation}\label{equchar} 
\begin{gathered}
2((1-U^2)^2+V^2U^2Q_0(U^2))\pl_U\omega_\ell +UQ_1(U^2)v\pl_v\omega_\ell+U(1-U^2)^2(\pl_U\omega_\ell)^2
+\frac{U}{R^2}|\pl_w\omega_\ell|^2\\
+\nu UVQ_2(U^2)(\frac{1}{R}\pl_w\omega_\ell)+\nu UP_1\big(U^2,V, R; U\pl_U\omega_\ell, R\pl_v \omega_\ell,\frac{1}{R}\pl_w\omega_\ell\big)=UQ_3(U^2,V)
\end{gathered}
\end{equation}
for some polynomials $Q_i$, and $P_1$ having the same properties as $P_0$, and $Q_2$.
This equation is of the form $F(D\omega_\ell,\omega_\ell,z)=0$, where $x=(U,v,w)$, 
$D\omega_\ell=(\pl_U\omega_\ell,\pl_v\omega_\ell,\pl_w\omega_\ell)$ and
\begin{equation}\label{bdf.10}
\begin{aligned}
F(p,z,x) &= F(p_U,p_v,p_w,z,U,v,w) \\
 &= 2\left[ (1-U^2)^2+V^2U^2Q_0(U^2)
 \right]p_U+
 UvQ_1(U^2) p_v + \nu \frac{UV}{R} Q_2(U^2)p_w \\
 & \quad +U(1-U^2)^2p_U^2+ \frac{Up_w^2}{R^2}+\nu UP_1(U^2,V; Up_U, Rp_v, \frac{p_w}{R})-UQ_3(U^2,V).
\end{aligned}
\end{equation}
In this definition, notice that the dependence in $z$ and $w$ is in fact trivial.  Now, since
$\pl_{p_U}F|_{U=0}=2\ne 0,$
the equation with initial condition $\omega_\ell|_{U=0}=\varphi_\ell$ is noncharacteristic.  It can therefore be resolved  using the method of characteristics for $U$ small outside $R=0$. In general, the equations for the characteristics are given by (denoting $(x_1,x_2,x_3)=(U,v,w)$ and 
$(p_1,p_2,p_3)=(p_U,p_v,p_w)$)
\begin{equation} \label{eqcharac}
\begin{split}
\dot{p}_i(s)= & -\pa_{x_i}F(p(s),z(s),x(s))- \pa_z F(p(s),z(s),x(s)),\\
 \dot{z}(s) =& \sum_i \pa_{p_i} F(p(s),z(s), x(s)) p_i(s), \\
 \dot {x}_i(s)= &\pa_{p_i}F(p(s),z(s),x(s)).
\end{split}\end{equation}
where a dot is used to denote a derivative with respect to the parameter $s$. We notice that, when $\nu=0$,
these equations have smooth coefficients 
except for all terms containing $p_w/R$.
Thus they are smooth outside the face $\mc{F}_R=\{R=0\}$, in particular they 
restrict on the face $\mc{F}_{\ell}\setminus \{R\not=0\}$ corresponding to the rank-1 cusp limiting case.
We will need to solve these equations with the following initial conditions on the face $\mc{F}_u=\{U=0\}$ 
(we restrict for the moment to the region $U=0,R\not= 0$)
\begin{equation}
\begin{gathered}
U(0)=0, \quad v(0)= v_0, \quad w(0)= w_0, \quad z(0)=\varphi_{\ell}(v_0,w_0),  \\
\quad p_v(0)=\pl_v \varphi_{\ell}(v_0,w_0), \quad 
p_w(0)= \pl_w \varphi_{\ell}(v_0,w_0), \,\,  \,\, p_U(0)=0,\end{gathered}
\label{bdf.13}\end{equation}
where the last condition follows from the fact $F(p(0),z(0),x(0))=0$.  
The behavior of the solution for $U$ small near the face $R=0$ can possibly be singular because of the singularity of the coefficients containing some $R^{-1}$ in $F$ there. The solution $\omega_\ell$ will be given by
\begin{equation}\label{Domega} 
D\omega_\ell(U(s),v(s),w(s))=(p_U(s),p_v(s),p_w(s)),\quad \omega(U(s),v(s),w(s))=z(s) 
\end{equation}
with initial condition $\omega_{\ell}(0,v_0,w_0)=\varphi_\ell(v_0,w_0)$.
We analyze the solution near $\mc{F}_R$ when $\ell=0$
($\mc{F}_R\cap\mc{F}_\ell$ corresponds to $\{u=v=0\}$ inside $\mc{F}_\ell$).
 \begin{prop}\label{proprho0}  
For each $\varphi_0\in \mc{C}^\infty(\mc{F}_\ell\cap\mc{F}_u)$ with $\pl_w\varphi_0$ vanishing to infinite order at 
$\mc{F}_R$, there exists a unique smooth function $\omega_0$ on $\mc{F}_\ell$ defined in a neighborhood of  $\mc{F}_{u}\cap \mc{F}_R$ in $\mc{F}_\ell$
  such that $\pa_w \omega_0$ vanishes to infinite order at $\mc{F}_R$ and $\rho_0= e^{\omega_0}U$ is a boundary defining function of $\mc{F}_{\ell}\cap \mc{F}_{u}$ with the property that
 $$
          \left. (\rho_0^2g_0)\right|_{\mc{F}_u}=e^{2\varphi_0}h_0,\quad 
     \left| \frac{d\rho_0}{\rho_0} \right|_{g_{0}}=1.  
 $$
\end{prop}
\begin{proof}
We need to investigate if the equation \eqref{eqcharac} can be solved in a uniform way as the initial condition $v_0$ in \eqref{bdf.13} approaches zero. We first change coordinates and use the coordinates $(u',v',w')$ of 
\eqref{uvtou'v'} in which the metric $g_0$ has the simpler form \eqref{model2g0}. In fact, since 
the metric in the new coordinates has the same form as in the original coordinates $(u,v,w)$ but 
with $\nu$ replaced by $0$, we are reduced to solve a Hamilton-Jacobi equation 
which has the same form as \eqref{equchar} but with $\nu=0$, and in the coordinates $(U',v',w')$ where 
$U':=u'/\sqrt{u'^2+v'^2}$. Our first goal is to prove that $\omega_0$ viewed in the $(U',v',w')$ coordinates
is smooth near $U'=0$, and then to come back to the original coordinates using \eqref{uvtou'v'} to deduce the desired result. 

We are reduced to analyze the solution of \eqref{equchar} when $\nu=0$, which we now do (for convenience of notations we keep the expression of this equation with the variable $(U,v,w)$ for the moment, 
having in mind that they really mean $(U',v',w')$). We also allow $w$ and $w_0$ to be in $\rr$ instead of $\rr/\demi\zz$, which is the same as viewing the equation in the universal covering, since we need to work in the domain \eqref{mcD} where the coordinates $(U',v',w')$ are valid.
We notice that since we assume $\nu=0$, each of the singular terms in the equations \eqref{eqcharac}
comes now with a $p_w$ factor. 
From the initial conditions and the independence of $F$ with respect to $w$, 
we have that $p_w(s)= \pl_w \varphi_0(v_0,w_0)$ for all $s$. 
On the other hand, by hypothesis, we know that $\pl_w \varphi_0(v_0,w_0)=\mc{O}(|v_0|^\infty)$ 
when $v_0\to 0$.  To solve the ODE \eqref{eqcharac} uniformly as $v_0\to 0$, we now
check that for $v_0\ne 0$, $v(s)$ cannot approach zero rapidly.  
\begin{lem}\label{bdf.16}
There exists a positive constant $K$ depending on $\varphi_0$ but not on $v_0$ and $w_0$, as well as $C>0$ such that 
\[ |v(s)| \ge |v_0| e^{-Cs} \quad \mbox{and} \quad U(s)\ge s \quad \mbox{for} \quad s\le K.\]
\end{lem}    
\begin{proof}
We will consider the case $v_0>0$, since the case $v_0<0$ can be dealt with in a similar fashion.  
First we use that for $\ell=0$, we have that
\[ R= \frac{v}{\sqrt{1-U^2}}, \quad V=\sqrt{1-U^2}.\]
Set $y= \log v$, then from \eqref{bdf.10} and \eqref{eqcharac}, we can write, as long as $U<1$, 
$$
   \dot{y}= UQ_1(U^2)+ U\left( A_1(U^2)e^y+ A_2(U^2) p_U + A_3(U^2)e^y p_V + A_4(U^2)e^{-y} p_w \right)
$$
for some polynomials $A_i$ in the variable $U^2$.
Consider the vector $\overrightarrow{X}=(p_U(s),p_v(s)-p_v(0), U(s), y(s)-y(0))$, where $y(0)= \log v_0$.  Since $p_w$ is in fact independent of $s$, we see from \eqref{bdf.10} and \eqref{eqcharac} that there exists a positive constant $K_1$ 
depending on $\varphi_0$ such that 
$$
      \frac{d}{ds} | \overrightarrow{X}(s)| \le K_1 
$$
whenever $|\overrightarrow{X}(s)|\le \frac{1}{K_1}$.  This means that 
$$
         |\overrightarrow{X}(s)| \le K_1 s
$$
for $s\le \frac{1}{K_1^2}$.  In particular, there exists $C>0$ such that
$$
   |y(s)-y(0)|\le K_1 s \quad \Longrightarrow \quad v(s)\ge v_0 e^{-C s} 
$$
for $s\le  \frac{1}{K_1^2}$.  This gives the first half of the result for some big constant $K$.  Now, for $|\overrightarrow{X}(s)|$ sufficiently small, notice that $\dot{U}\ge 1$.  Integrating, we get that $U(s)\ge s$ for $s$ sufficiently small.  Taking the constant $K$ smaller if necessary gives the result.     
\end{proof}

We have $p_w(s)=p_w(0)=\pl_w\varphi_0(v_0,w_0)$ which 
decreases rapidly when $v_0$ tends to zero, and $V=v/R$ is close to $1$ when $U$ is 
small, thus using Lemma \ref{bdf.16}, there is $C>0$ such that for $s<1/K$, we have $U(s)\geq s$ and 
\[\Big|\frac{p_w(s)}{R(s)^2}\Big|\leq 2\Big|\frac{\pl_w\varphi_0(v_0,w_0)}{v(s)^2}\Big| \leq Ce^{CU(s)}\Big|\frac{\pl_w\varphi_0(v_0,w_0)}{v_0^2}\Big|=\mc{O}(|v_0|^\infty).\]  
Using this rapid vanishing as $v_0\to 0$, by looking at \eqref{eqcharac} and \eqref{bdf.10}, we deduce 
that the $(U(s),v(s),w(s))$ extend smoothly as the initial condition $x_i(0),p_i(0)$ 
tend to $\{v=0\}$, on a uniform time $s\in[0,s_0]$ with $s_0>0$. 
In fact, with the initial condition $v_0=0$, we have by \eqref{bdf.10} that $v(s)=p_w(s)=0$ for all $s$ and the ODE  simply becomes in the region $U<1$ (using that $V=\sqrt{1-U^2}$ in that case)
\begin{equation}
\dot{p_U} = L_1(U,p_U) , \quad
\dot{p_v}= L_2(U,p_v,p_U), \quad 
\dot{z} = L_3(U,P_U), \quad
\dot{U} = L_4(U,P_U), \quad \dot{w}=0
\label{bdf.17}\end{equation}
for some polynomials $L_j$ with $L_4(0,p_U)=2$. In particular, we see that the curves $U(s),v(s)$ are tangent to the face $v=0$ (as long as $U<1$) and they are transverse to $U=0$; moreover $U(s)=2s+\mc{O}(s^2)$ near $U=0$. We thus obtain that $\psi:(s,v_0,w_0)\to (U(s),v(s),w(s))$ is a smooth local diffeomorphism on $[0,\epsilon)\x [0,\epsilon)\x \rr$ for small $\epsilon>0$ and there is $\epsilon>0$ such that 
for each point $w_0\in\rr$ it is a diffeomorphism from $[0,\epsilon)\x[0,\epsilon)\x (w_0-\epsilon,w_0+\epsilon)$ on its image. Moreover, it is easily seen that $\psi(s,v_0,w_0+\demi)=(U(s),v(s),w(s)+\demi)$.
The same hold in the region $v_0\leq 0$ and this implies that $\omega_0$ given by \eqref{Domega} for $\ell=0$ extends as a smooth function of $(U,v,w)$ near each $(0,0,w_0)$ on $\{U\geq 0,v\geq 0\}$ and on $\{U\geq 0,v\leq 0\}$, in some neighborhood which has uniform size with respect to $w_0$. We also have that 
$\pl_w\omega_0(\psi(s,v_0,w_0))=p_w(s)=\pl_{w}\varphi_0(v_0,w_0)=\mc{O}(|v_0|^\infty)$, thus
$\pl_w\omega_0=\mc{O}(|v|^\infty)$ uniformly where it is defined.

We have thus proved that in the $(U',v',w')$ coordinates, $\omega_0$ lifted to the universal cover 
is smooth in $[0,\epsilon)\x[0,\epsilon)\x \rr$, and $\pl_{w'}\omega_0=\mc{O}(|v'|^\infty)$ .
To conclude the proof, we need to come back to the original coordinates $(U,v,w)$ by using \eqref{uvtou'v'}:
\[ \begin{gathered}
U'=\frac{U\sqrt{1+\nu^2}}{\sqrt{1+\nu^2U^2}}, \quad v'=\frac{v(1+\nu^2)}{1+\nu^2U^2},\quad 
w'=w-\frac{\nu}{1+\nu^2}\frac{1-U^2}{v}
\end{gathered}.\] 
First it is clear that $\omega_0(U',v',w')$ is smooth when viewed as a function of $(U,v,w)$ except possibly at $v=0$ where $w'$ is a singular function of $v$. Similarly to the discussion of Section \ref{formationofcusp} (which corresponds to an analysis in the boundary $U=0$), the fact that $\pl_{w'}\omega_0=\mc{O}(|v'|^\infty)$ actually implies that 
$\omega_0$ is smooth in the variable $(U,v,w)$ since $D\omega_0$ admits a smooth extension to $v=0$. This achieves the proof of the proposition, as $(U,v,w)$ are smooth coordinates near $\mc{F}_R\cap \mc{F}_u$ on the face $\mc{F}_\ell$, and $U'$ is a smooth function of $U$.
\end{proof}
 
\begin{cor}\label{hxsmooth}
Let $\rho_0$ be the function of Proposition \ref{proprho0}. There exists a diffeomorphism
$\phi:[0,\eps)_s\x \mc{O}\to \mc{Q}\subset \mc{F}_\ell$ with $\mc{Q}$ a neighborhood of $\mc{F}_u\cap\mc{F}_R$ in $\mc{F}_\ell$ and $\mc{O}$ a neighborhood of $\mc{F}_u\cap\mc{F}_R$ in $\mc{F}_u\cap\mc{F}_\ell$ such that 
$\phi^*\rho_0=s$ and 
\[\phi^*g_0= \frac{ds^2+ h_0 (A_s\cdot\cdot)}{s^2}\]
with $A_x$ is a one-parameter smooth family of smooth endomorphisms of $T\mc{O}$ up to $\mc{F}_R$, for $s\in[0,\eps)$, so that 
$h_s(\cdot,\cdot):=h_0 (A_s\cdot\cdot)$ is a smooth family of cusp symmetric tensors.
\end{cor}
\begin{proof} The diffeomorphism is given by $\phi(s,v_0,w_0)=\phi_s(v_0,w_0)$ where $\phi_s$ is the flow at time $s$ of the gradient $\nabla^{\rho_0^2g_0}\rho_0$ of $\rho_0$ with respect to $\rho_0^2g_0$.
First, we notice that this flow is exactly the diffeomorphism 
$\phi(s,v_0,w_0)=x(s/2)$ where $x(s)$ is the integral curve studied in the proof of the previous proposition (satisfying \eqref{eqcharac} with initial condition $x(0)=(0,v_0,w_0)$. Since $(\phi^*U)/s$ is a smooth function on $[0,\eps)\x \mc{O}$ for some small neighborhood $\mc{O}$ of $\{0\}\x\{0\}\x (\rr/\demi\zz)$ in $(U,v,w)\in [0,\eps)\x [0,\eps)\x (\rr/\demi\zz)$, 
the metric $s^2\phi^*g_0$ is given by a positive smooth function times $\phi^*(U^2g_0)$ with $g_0$ given in
\eqref{bdf.4} (for $\ell=0$). To prove the statement, it suffices to check that for vector fields 
$Z_1:=v_0\pl_{v_0}$ and $Z_2:=v_0^{-1}\pl_{w_0}$, we have that 
$\phi^*(U^2g_0)(Z_i,Z_j)$ is smooth near $s=v_0=0$ for $i,j\in \{1,2\}$. Since $\phi(s,0,w_0)\subset \{v=0\}$ 
by the analysis of the proof in the previous proposition, writing $\phi(s,v_0,w_0)=(U,v,w)$ we get $v=v_0(1+v_0f(s,v_0,w_0))$ 
and $w=w_0+v_0k(s,v_0,w_0)$ for some smooth functions 
$f,k$, and thus 
\[ \phi_*(v_0\pl_{v_0})=vW_1, \quad \phi_*(v_0^{-1}\pl_{w_0})=v^{-1}\pl_w+W_2\] 
for some smooth vector field $W_1,W_2$ near $v=U=0$. By inspecting \eqref{bdf.4} for $\ell=0$, 
 $\phi^*(U^2g_0)(Z_i,Z_j)$ is smooth near $s=v=0$. The same argument works in the region $v\leq 0$ covering the other neighborhood of $\mc{F}_R\cap\mc{F}_u$ in $\mc{F}_\ell$.
\end{proof} 
\subsection{Proof of Proposition \ref{deffunccusp}}
We decompose the hyperbolic $3$-manifold  with rank-1 cusps $(X,g)$ as in Section \ref{geofinite} into a region $\mc{K}\subset \bbar{X}$ and some cusp neighborhoods $\mc{U}_j^c$ for $j=1,\dots,j_1$. Recall that $\bbar{X}$ can be compactified into $\bbar{X}_c$.
 Then we fix a boundary defining function $\rho$ in a neighborhood of $M =\pl\bbar{X}$, which is equal to $\rho=u/\sqrt{u^2+v^2}$ in the coordinates of the model \eqref{Uj'} of $\mc{U}_j^c$. The hyperbolic metric $g$ there, as given by the model \eqref{Uj'}, corresponds to the case $\ell=0,\nu=0$ in the expression \eqref{gellnu} and $U=u/R$ is the chosen defining function of $\pl\bbar{X}$ in these coordinates. 
Let $h^{\rm hyp}$ be the unique hyperbolic metric on $M$ in the conformal class of 
$h:=(\rho^2g)|_{TM}$. Let $\psi\in C^\infty_r(\bbar{M})$ and 
 $\hat{h}=e^{2\psi}h^{\rm hyp}$. By Proposition \ref{uniform}, we have $e^{2\psi}h^{\rm hyp}= e^{2(\psi+\varphi)}h$ for some $\varphi\in \mathcal{C}^{\infty}_r(\bbar M)$.  Since we still have that $\psi+\varphi\in \mathcal{C}_r^{\infty}(\bbar M)$,  Proposition \ref{proprho0} shows that there exists a smooth defining function $\hat{\rho}$ 
of $\bbar{M}$ on a neighborhood of $\cf\cap\bbar{M}$ in $\bbar{X}_c$ (as explained above, $\bbar{M}$ corresponds to $\mc{F}_u$ and $\cf$ to $\mc{F}_R$ in the model $\mc{F}_\ell$ of $\bbar{X}_c$ near the cusps),
such that  $|d\hat{\rho}/\hat{\rho}|_{g}=1$ with $\rho^2g|_{TM}= e^{2\psi}h^{\rm hyp}$; it is unique where it is defined. On the other hand, outside $\mc{U}_j^c$, this equation is also a smooth non-characteristic  Hamilton-Jacobi type  equation,  thus the solution $\hat{\rho}$ defined near $\bbar{M}\cap \cf$ can be extended uniquely as a solution also in a whole neighborhood 
of $\bbar{M}$ in $\bbar{X}_c$, giving the desired function $\hat{\rho}$.  
Considering the maps $\phi:[0,\epsilon)_x\to \bbar{M}\to \bbar{X}_c$ given by $\phi(x,y)=\phi_x(y)$ where $\phi_s$ the flow at time $s$ of the gradient $\nabla^{\hat{\rho}^2g}\hat{\rho}$ of $\hat{\rho}$ with respect to $\hat{\rho}^2g$, 
we see by using Corollary \ref{hxsmooth} (recall that $\phi$ is the gradient flow in the proof of that Corollary) that on $(0,\epsilon)_x\x M$
\[ \phi^*g= \frac{dx^2+\hat{h}_x}{x^2}\]
for some 1-parameter family $\hat{h}_x$ of smooth metrics on $M$ depending smoothly on $x\in[0,\eps)$,
and $h_x$ are actually a smooth family of cusp symmetric tensors. Since $g$ is hyperbolic, we know (as it is a local computation) from \cite[Theorem~{7.4}]{FeGr} that the dependence of $\hat{h}_x$ is a polynomial of order $2$ in $x^2$
$$
       \hat{h}_x = \hat{h}(({\rm Id}+x^2A)\cdot,\cdot) 
$$
with $\Tr(A)= -\frac12 \Scal_{\hat{h}}$ and $\delta_{\hat{h}}(A)= \frac12 d \Scal_{\hat{h}}$.
It remains to  check that the complement of the region $\phi([0,\epsilon)\x \bbar{M})$, called $\mc{V}$, has finite volume with respect to $g$ in $X$. Clearly, 
$\mc{K}\cap (X\setminus \{\hat{\rho}<\epsilon\})$ is compact in $X$ thus has finite volume. Now we analyze the region $\mc{U}_j^c\setminus\{\hat{\rho}< \epsilon\}$. To show that it has finite volume, it suffices to use that $\hat{\rho}$ is a defining function of $\mc{F}_u\cap \mc{F}_\ell$ in the blown-up space 
$\mc{F}_\ell$ of $\bbar{\mc{U}}_j^c$ around the region $H=\{(u,v)=0\}$ representing the cusp, and so $\{\hat{\rho}\geq \epsilon\}$ is contained in some 
region $\{U\geq c\epsilon\}$ for some $c>0$. Now the volume form of the metric $g$ 
in coordinates $(u,v,w)$ is $\frac{u^2+v^2}{u^3}dudvdw$ and a simple computation shows that 
\begin{equation}\label{intsurV} 
\int_{0}^1 \int_{-Cu}^{Cu}\Big(\frac{u}{\sqrt{u^2+v^2}}\Big)^z \frac{u^2+v^2}{u^3}dv du<\infty
\end{equation}
for all $z\in\cc$ and thus by taking $z=0$ we see that the region has finite volume for any finite constant $C>0$. It remains to show that if $\hat{\rho}$ is extended smoothly to $\bbar{X}_c$ as a boundary defining function of $\bbar{M}$ (positive in $\bbar{X}_c\setminus \bbar{M}$) then
$H(z)=\int_{X}\hat{\rho}^z {\rm dvol}_g$
is meromorphic in $\{{\rm Re}(z)>-\epsilon\}$ for some $\epsilon>0$. We can split the integral as an integral near 
$\pl\bbar{X}\cap \mc{K}$ and the meromorphy of this part follows directly from the fact that 
$\hat{\rho}$ is a smooth boundary defining function there, and there remains the integral in each $\mc{U}_j^c$. The part of the integral in $\mc{V}$ clearly gives holomorphy in $z$ by \eqref{intsurV}. For the integral in $\mc{U}_j^c\setminus \mc{V}$, we notice that the volume form in the coordinates
$(U,R,w)$ near $\mc{F}_u=\{U=0\}$ in the model $\mc{F}_\ell$ (isometric to $\mc{U}_j^c$ with $\nu=0$) of Section \ref{hamilton} is given by $dUdRdw/(U^3\sqrt{1-U^2})$ and thus from the fact that
$\hat{\rho}/U$ is a smooth positive function in these coordinates near $U=0$, the meromorphy of the remaining part of the integral $H(z)$ follows by Taylor expanding $\hat{\rho}/U$ at $U=0$. 
\qed

\subsection{Taylor expansion of the boundary defining function to second order}
For $\ell>0$ fixed, it is also straightforward to solve the  equations \eqref{HJ} near the degenerating curve and find $\omega_\ell$ and $\rho$.  The function $\omega_\ell$ will be smooth in $s$, so smooth in $U$.  In particular, at $U=0$, it has an expansion of the form
\begin{equation}    
  \omega_\ell \sim \sum_{j=0}^{\infty} a_j U^k.
\label{bdf.19}\end{equation}
To compute the limit as $\ell\to 0$ of the renormalized volume, we will need to know the terms of order $0$ and $2$.  By assumption, we have that $a_0=\varphi_{\ell}$.  
We now compute $a_1$ and $a_2$.  

\begin{prop}\label{bdf.19a}
Near $v=0$, the coefficients $a_1$ and $a_2$ in the expansion \eqref{bdf.19} are given by
$a_1=0$ and 
$$  
a_2= -\frac14 \left(|d \varphi_{\ell}|^{2}_{h_\ell} +(1+\nu^2)\left(1-\frac{\ell^2}{v^2+\ell^2}\right)-2 +2(\nu^2-1)v\pl_v\varphi_\ell -\frac{2\nu v}{v^2+\ell^2}\pl_w\varphi_\ell\right).
$$
\end{prop}
\begin{proof}
We see directly from \eqref{equchar} that $a_1=0$.
Then notice that by \eqref{gellinv}, the metric dual to $\bbar{g}_\ell=U^2g_\ell$ 
is smooth near $\mc{F}_u\setminus (\mc{F}_u\cap \mc{F}_R)$ and as $U\to 0$
\[\begin{split} 
|d\omega_\ell|^2_{g_\ell}=& U^2\Big( (v^2+\ell^2)((1+\nu^2)(\pl_v\varphi_\ell)^2)+\frac{(\pl_w\varphi_\ell)^2}{v^2+\ell^2}
-2\nu \pl_v\varphi_\ell \pl_w\varphi_\ell\Big) +\mc{O}(U^3)\\
=&  U^2|d\varphi_\ell|^2_{h_\ell}+\mc{O}(U^3)
\end{split}\]
where $h_\ell$ is given by \eqref{hell}.
Combining this with  \eqref{HJ}, \eqref{compute1} and \eqref{compute2}, we have 
\[ -2+\frac{v^2}{v^2+\ell^2}(1+\nu^2)+4a_2+2(\nu^2-1) v\pl_v\varphi_\ell-2\nu \frac{v}{v^2+\ell^2}\pl_w\varphi_\ell+|d\varphi_\ell|^2_{h_\ell}=0\]
which achieves the proof.
\end{proof}

\section{Variation formulas for the renormalized volume and K\"ahler potential for Weil-Petersson metric}

In this section we describe the properties of the renormalized volume as a function on the conformal class of the conformal boundary: we first compute the variation of the renormalized volume for families of hyperbolic 
metrics with rank $1$-cusps, and we show that ${\rm Vol}_R$ is a K\"ahler potential for Weil-Petersson metric
in a Bers slice of the Quasi-Fuchsian deformation space. 

\subsection{Variation formula}
Arguing as in \cite[Prop. 3.11]{GMS}, we have the following variation formula for the renormalized volume under a change of conformal representative in the conformal boundary.  
\begin{prop}\label{vrv.1}
Let $X$ be a geometrically finite hyperbolic $3$-manifold.
Let $h^{\rm hyp}$ be the unique hyperbolic representative in the conformal boundary of $g$ and let 
$\hat{h}:=e^{2\psi}h^{\rm hyp}$ with $\psi\in\CI_r(\bbar M)$. If 
$\rho$ and $\hat{\rho}$ are geodesic boundary defining functions associated to $h^{\rm hyp}$ and $\hat{h}$ 
given by Proposition  \ref{deffunccusp}, we have
$$  
 {\rm Vol}_R(X, \hat{h})= {\rm Vol}_R(X, h^{\rm hyp}) -\frac14 \int_{M} (|\nabla\psi|^2_{h^{\rm hyp}}-2\psi) \rm dvol_{h^{\rm hyp}}.    
$$
For any $\chi\in \mc{C}_c^\infty(\bbar{X})$ satisfying $\chi=\sum_{k=0}^2\chi_k\rho^2+\mc{O}(\rho^3)$ at 
$\pl\bbar{X}$, with $\chi_k\in \mc{C}_c^\infty(M)$ 
\begin{equation}\label{FPz=0}
\begin{split}
 {\rm FP}_{z=0}\int_{X}\hat{\rho}^z\chi {\rm dvol}_g=& {\rm FP}_{z=0}\int_{X}\rho^z \chi {\rm dvol}_g \\ 
 & -\frac14 \int_{M} (\chi_0(|\nabla\psi|^2_{h^{\rm hyp}}-2\psi)-4\chi_2\psi) \rm dvol_{h^{\rm hyp}}.    
\end{split}
\end{equation}
\end{prop}
\begin{proof}
First, by Proposition~\ref{deffunccusp}, associated to both $h^{\rm hyp}$ (resp. to $\hat{h}$), there are product coordinates $[0,\epsilon)_x\x \bbar{M}$ near $\bbar{M}$ in the compactification $\bbar{X}_c$ of $X$ in which $g$ is of the form
$$
  g= \frac{dx^2 + h_0+ x^2 h_2 + x^4 h_4}{x^2}
$$
with $h_0=h^{\rm hyp}$ (resp. $h_0=\hat{h}$) , $h_2,h_4$  some smooth cusp symmetric tensors such that $\Tr_{h_0}(h_2)= -\frac12 \Scal_{h_0}$ and $\delta_{h_0}(h_2)= \frac12 d \Scal_{h_0}$. The complement of the regions covered by these coordinates have finite volume, thus the part of the integrals above over the region $x>\epsilon/2$ are trivial to deal with.
On the other hand, by the proof of Proposition~\ref{deffunccusp}, we can also solve the Hamilton-Jacobi equation $|\frac{dx}{x}+d\omega|^2_g=1$ near $\bbar M$ with initial condition $\left. \omega\right|_{\bbar M}=\psi$.  From the symmetries of this equation, we see that $\omega$ has to have an even expansion in $x$ at $\bbar M$, 
$\omega\sim \sum_{j=0}^{\infty} \omega_{2j} x^{2j}.$
As in \cite[Lemma~3.6]{GMS}, putting this expansion back in the Hamilton-Jacobi equation, we compute that (the computation is local)
$$
      \omega_2= -\frac14 |\nabla \omega_0|^2_{h_0}, \quad \textrm{ with }\omega_0=\psi.
$$
On the other hand, the volume form of $g$ is given by ${\rm dvol}_g= v(x) {\rm dvol}_{h_0}\frac{dx}{x^3}$ with $v(x)= v_0+x^2v_2+\mc{O}(x^3)$ for $v_0=1$ and $v_2=-\frac14 \Scal_{h_0}.$  Hence, we compute just as in the proof of \cite[Lemma~{3.5}]{GMS} that 
\begin{equation}\label{Volomega}
\begin{aligned}
{\rm Vol}_R(X,\hat{h})&= {\rm Vol}_R(X) + \int_{M} \left(v_0\omega_2  +v_2\omega_0\right)\ {\rm dvol}_{h^{\rm hyp}} \\
        &=   {\rm Vol}_R(X) -\frac14 \int_{M} (|\nabla\omega_0|^2_{h^{\rm hyp}}-2\omega_0) \rm dvol_{h^{\rm hyp}}.   \end{aligned}   
\end{equation}
For \eqref{FPz=0}, the calculation is similar but one has to replace $v(x)$ by $v(x)\chi(x)$ in the reasoning, thus 
$v_0\omega_2$ and $v_2\omega_0$ become $v_0\chi_0\omega_2$ and 
$(v_2\chi_0+v_0\chi_2)\omega_0$.
\end{proof}

First we say that $(X,g^t)$ for $t\in(-t_0,t_0)$ is a smooth family of geometrically finite hyperbolic manifolds if $g:=g^0$ is a geometrically finite metric on $X$ with $j_1$ cusps of rank-$1$, represented by some disjoint curves 
$H=\cup_{j=1}^{j_1}H_j$  in the boundary $\pl\bbar{\bf{X}}$ of the compactification $\bbar{\bf{X}}$ 
as in Section \ref{geofinite}, $g^t$ is hyperbolic for all $t$
and there is a neighborhood $\mc{U}_j$ of $H_j$ in $\bbar{\bf{X}}$ such that $\rho^2g^t$ extends to a smooth 
metric on $\bbar{\bf{X}}\setminus \cup_{j}\mc{U}_j$ if $\rho$ is a boundary defining function of $\pl\bbar{\bf{X}}$, and there exists a smooth family of diffeomorphisms $\psi_j^t:\mc{U}_j\to \psi^t(\mc{U}_j)\subset \bbar{\hh_\zeta^2}\x (\rr/\demi\zz)_w$ such that for $\zeta=v+iu\in\hh^2$
\[ (\psi_j^t)_*g^t= \frac{du^2+dv^2+(u^2+v^2)dw^2}{u^2}.\]
For such a family of metrics, it is easy by extending $(\psi_j^t)^{-1}\circ \psi_j^0$ 
to $\bbar{\bf{X}}$ to construct a diffeomorphism $\theta^t$ of 
${\bf{X}}$ such that $\rho^2(\theta^t)^*g^t$ extend smoothly as a metric on 
$\bbar{X}=\bbar{\bf{X}}\setminus H$ and near $H_j$
\[(\psi^0_j)_*(\theta^t)^*g^t=\frac{du^2+dv^2+(u^2+v^2)dw^2}{u^2}.\] 
We can thus reduce the analysis to the family of metrics $(\theta^t)^*g^t$ 
with a cusp singularity at $H$, which we do now and to avoid heavy notation we write $g^t$ instead of $(\theta^t)^*g^t$. Denote by $h^t$ the hyperbolic metric in the conformal boundary of $(X,g^t)$, it is a smooth family in 
$t$ of hyperbolic metric with finite volume and cusps.
Proceeding as in the proof of Proposition~\ref{proprho0} and using $h^t$ as the representative of the conformal infinity of $g^t$, we can then solve the Hamilton-Jacobi equation 
$$
      \left| \frac{d\rho^t}{\rho^t}\right|_{g^t}=1, \quad \left.  (\rho^t)^2g^t\right|_{TM} =h^t
$$  
smoothly in $t$ to get a smooth family of boundary defining functions $\rho^t$ of $\bbar M$ in $\bbar{X}$.   As we have seen in the proof of Proposition~\ref{deffunccusp}, the gradient vector field 
$\nabla^{\bbar g^t}\rho^t$, where $\bbar g^t= (\rho^t)^2g^t$, will be defined and smooth in a neighborhood of $\bbar M$ in  $\bbar{X}_c$ and will be tangent to the cusp face $\cf$.  Integrating this vector field for each $t$ then gives a smooth family of collar neighborhood $\phi^t: \bbar M\times [0,\epsilon)_x \to \bbar{X}$ such that 
\begin{equation}\label{phi^tg^t}
        (\phi^t)^*g^t= \frac{dx^2+ h_0^t+x^2h_2^t+x^4h_4^t}{x^2}
\end{equation}
with $h^t_{2j}$ some smooth families (in $t$) of cusps symmetric tensors such that $h_0^t= h^t$.

\begin{theo}\label{variationVolR}
Let $(X,g^t)$ be a smooth family of geometrically finite hyperbolic metrics. 
Let $h^t$ be the unique hyperbolic representative of the conformal infinity of $g^t$ and $h_2^t$ the second fundamental form at $\pl \bbar{X}$ given by \eqref{phi^tg^t}. If ${\rm Vol}^t_R(X)$ denotes the renomalized volume of $(X,g^t)$, then 
$$
  \left. \pa_t {\rm Vol}_R^t(X) \right|_{t=0} = -\frac14 \int_{M} \langle \dot{h}_0, h_2 -h_0\rangle_{h_0} {\rm dvol}_{h_0},
$$
where $h_2= \left.  h_2^t\right|_{t=0}$, $h_0= h^t|_{t=0}$ and  the dot denotes a derivative in the $t$ variable evaluated at $t=0$.  
\label{vrv.2}\end{theo}
\begin{proof} The proof is very similar to the proof of \cite[Theorem 5.3]{GMS} and is based on Schl\"afli formula, but here one has to be careful about the degeneracy near the cusps to perform the argument. Like in \cite[Theorem 5.3]{GMS}, we can pull-back $g^t$ (using an extension of $(\phi^t)^{-1}\circ \phi^0$) by a family of diffeomorphisms of $\bbar{X}_c$ which is the Identity outside a neighborhood of $\bbar{M}$ so that the new metric, is isometric to the right hand side of \eqref{phi^tg^t} near $\bbar{M}$ via the diffeomorphism $\phi:=\phi^0$ that is independent of $t$.
For $\delta\in (0,\delta_0)$, consider the region $V_{\delta}:= \phi(\bbar M\times [0,\delta))\subset \bbar{X}$. 
Then, as in the proof of Proposition~\ref{deffunccusp}, $\bbar{X}\setminus V_{\delta}$ is of finite volume with respect to $g^t$, and we claim that  
\begin{equation}
       \left.  \pa_t {\rm Vol}(\bbar{X}\setminus V_{\delta}, g^t)\right|_{t=0} = \frac12 \int_{\rho=\delta } \left(  \dot H + \frac12 \langle \dot{g}, \II  \rangle_{g} \right) {\rm dvol}_{g},
\label{sch.2}\end{equation}
where $H^t$ is the mean curvature of $\phi(\bbar M\times \{\delta\})$ with respect to the metric $g^t$, $\II^t$ is its second fundamental form and $g:=g^t|_{t=0}$.  The proof of \eqref{sch.2} is then the same as the one of \cite[Lemma~{5.1}]{GMS}: using the variation formula for the scalar curvature, we find 
\[ -4\pl_t{\rm Vol}(\bbar{X}\setminus V_{\delta},g^t)=\int_{X}(\Delta_g{\rm Tr}_g(\dot{g})+
d^*\delta^g(\dot{g})){\rm dvol}_g\]
and the integration by parts of $\Delta_g \Tr_g(\dot g)$ and $d^* \delta^g \dot g$ can be done but 
there could possibly be a new contribution coming from the cusp face $\cf$ in the compactification 
$\bbar{X}_c$ of $\bbar{X}$ ($\cf\cap\{\rho=\delta\}$ corresponds to the cusp point at infinity of the Riemann surface $\{\rho=\delta\}$). In order to analyze this, we can apply Green's formula on 
$\{R\geq \la,\rho\geq \delta\}$ where $R$ is a boundary defining function of $\cf$.
If $\phi': \cf \times [0,1)\to \bbar{X}$ is a collar neighborhood of the cusp face in $\bbar{X}$, we know from the local form \eqref{Uj'} of $g^t$ that 
$$
         {\rm Area}(\phi'(\cf\times \{\la\}))\cap (\bbar{X}\setminus V_{\delta}))= \mathcal{O}(\lambda^2).
$$
It is direct to check (using \eqref{gellinv} with $\nu=0$) that $\pl_R\Tr_g(\dot g)$ and $\delta^g (\dot g)(\pl_R)$ are uniformly bounded in $\lambda$ on $\phi'(\cf\times \{\lambda\}))\cap (\bbar{X}\setminus V_{\delta})$, where $\pl_R$ is the unit normal vector to $\phi'(\cf\times \{\lambda\}))$ with respect to $g$, this means that there is in fact no contribution coming from the cusp face when we take the limit $\lambda\searrow 0$.  Thus, when we integrate by parts, we obtain the same formula as in \cite{GMS} and \eqref{sch.2} follows.   
\end{proof}

\subsection{A K\"ahler potential for Weil-Petersson metric}
 
Consider the quasi-Fuchsian space associated to Riemann surfaces with $n$-cusps.  
For each pair $(M,h_-)$ and $(M,h_+)$ of hyperbolic surfaces of finite volume with $n$ cusps,  denoting $h:=(h_-,h_+)$, there exists a unique (up to diffeomorphism) complete 
hyperbolic metric $g=g_h$ on the cylinder $X:=\rr_t\x M$, which is realized as a quotient $\Gamma\backslash\hh^3$  for $\Gamma\subset {\rm PSL}_2(\cc)$ a quasi-Fuchsian group.
The quasi-Fuchsian space is identified with $\mc{T}(M)\x \mc{T}(M)$ 
where $\mc{T}(M)$ is the Teichm\"uller space of $M$.
Fixing $h_-=h_0$, the map $h_+\mapsto g_h$ provides an embedding of $\mc{T}(M)$ into the quasi-Fuchsian deformation space, called Bers embedding,  
and we view the renormalized volume as a function on $\mc{T}(M)$: $V_{h_0}: h_+\mapsto {\rm Vol}_R(X,g_{h})$.

\begin{proof}[Proof of Theorem \ref{Kahlerpot}] 
First, we notice that applying the proof of Proposition 7.1 in \cite{GMS} mutatis mutandis, we can compute the Hessian of the renormalized volume at the Fuchsian locus $h_+=h_-$ :
\begin{equation}\label{calculHessien}
{\rm Hess}_{h_0}(V_{h_0})(k)=\frac{1}{8}\int_{M}|k|^2_{h_0}{\rm dvol}_{h_0}=
\frac{1}{8}|k|^2_{\rm WP}.\end{equation}
By Theorem \ref{variationVolR} and 
the form $x^{-2}(dx^2+(1+\frac{x^2}{2})^2h_0)$  of the quasi-Fuchsian metric for $h_-=h_+=h_0$, the metric $h_+=h_0$ is a critical point of $V_{h_0}$ and a direct computation (as in \cite[Section 8]{KrSc}) shows that for $k_1,k_2\in T\mc{T}_{h_0}(M)$ 
\[
\bbar{\pl} \pl V_{h_0}(h_0).(k_1,k_2)=\frac{i}{4}\Big({\rm Hess}_{h_0}(V_{h_0})(Jk_1,k_2)-{\rm Hess}_{h_0}(V_{h_0})(k_1,Jk_2)\Big)
\] 
if $J$ is the complex structure on $T\mc{T}(M)$.  Combining with \eqref{calculHessien}, we obtain that
$\bbar{\pl} \pl V_{h_0}(h_0)=\frac{i}{16}\omega_{\rm WP}(h_0)$. 

To obtain the final result we need to show 
that $\bbar{\pl} \pl V_{h_-}(h_0)$ does not depend on $h_-$. This follows from quasi-Fuchsian reciprocity  like in \cite[Proposition 8.9]{KrSc}.  For the convenience of the reader,  we briefly repeat  the argument. 
 Thus, consider the Lie derivative $\mc{L}_{Y_-}(d\pl V_{h_-}(h_+))$ with respect to the variable $h_-$, 
where $Y_-$ is a vector field on $\mc{T}(M)$ and $h_+$ is fixed.   To prove it vanishes, 
it then suffices to show that $\mc{L}_{Y_-}(\pl V_{h_-}(h_+))$ is the exterior derivative of a function of $h_+$.
We claim that $\mc{L}_{Y_-}(\pl V_{h_-}(h_+))=dF_{(h_-,Y_-)}(h_+)$ where $F_{(h_-,Y_-)}$ is the function on 
$\mc{T}(M)$ defined by 
$$
F_{(h_-,Y_-)}(h_+)= \pl {\rm Vol}_R(h_-,h_+).(Y_-,0).
$$
To prove this, we will need to use the Bers embedding of Teichm\"uller space and especially its relation to ${\rm Vol}_R$.
The universal cover of $M$ is the upper half plane $U=\hh^2\subset \cc$ and (after composing by an isometry) the metric $h_+$ lifts to the hyperbolic metric $g_{\hh^2}=\frac{|dz|^2}{{\rm Im}(z)^2}$. The covering map is denoted $\pi: \hh^2\to M$, the fundamental group $\pi_1(M)$ is represented by a Fuchsian co-compact group $\Gamma\subset {\rm PSL}_2(\rr)$.
The metric $h_-$ on $M$ lifts by $\pi$ to a metric $\til{h}_-$ on $\hh^2$ which is $\Gamma$-invariant 
and of curvature $-1$. Using the map $z\mapsto \bar{z}$, we can equip the lower half-plane $L=\{{\rm Im}(z)<0\}\subset \cc$ with the metric $\til{h}_-$ (the orientation of $h_-$ is then reversed), 
and this metric can be written as $\til{h}_-=a(z)|dz+\mu d\bar{z}|^2$ for some smooth $a>0$ and some complex valued Beltrami coefficient $\mu$ with $|\mu|<1$. Extend $\mu$ by $0$ on $U$, then by Ahlfors-Bers result, 
there  is a unique quasiconformal map $f:\cc\to \cc$ which satisfies 
\[ \pl_{\bar{z}}f=\mu\pl_z f\]
and $f$ fixes the points $0,1,\infty$.
Notice that $f$ is a conformal map from $(\cc,|dz+\mu d\bar{z}|^2)$ to 
$(\cc,|dz|^2)$, and thus $f$ is holomorphic in $U$. The Bers embedding is the map 
\[ \Theta_{h_+}: h_- \mapsto S(f|_{U}) ,\quad  S(f)=\Big( \pl_z\Big(\frac{\pl_z^2f}{\pl_zf}\Big)-\demi \Big(\frac{\pl_z^2f}{\pl_zf}\Big)^2\Big)dz^2.\] 
where $S$ is the Schwarzian derivative. The element $S(f)$ is a holomorphic quadratic differential with respect to the complex structure of $h_+$ on $U$, and which is $\Gamma$-equivariant, thus descends to an element 
in $(T^*_{h_+}\mc{T})^{1,0}$ if $\mc{T}$ denotes Teichm\"uller space of $M$. The Bers map $\Theta_{h_+}$ is holomorphic as a map $\mc{T}\to (T_{h_+}^*\mc{T})^{1,0}$.
The group $\Gamma':=f\Gamma f^{-1}$ is a quasi-Fuchsian subgroup of ${\rm PSL}_2(\cc)$.
Let $J=f^{-1}$ where now we consider $f: U\to \Omega_+$. Here $\Omega_+$ is the upper component of the domain of discontinuity of $\Gamma'$ on $\cc$ and $\Gamma'\backslash \Omega_+$ equipped with the complex structure induced by $\cc$ is conformal to $(M,h_+)$ by $f$. One has moreover 
$J^*g_{\hh^2}=e^{\phi}|dz|^2$ for some smooth Liouville field $\phi$ on $\Omega_+$, $\Gamma'$-equivariant, and $e^{\phi}|dz|^2$ is a hyperbolic metric on $\Omega_+$. Thus 
\[ \pl_z\pl_{\bar{z}}\phi=\demi e^{\phi}.\] 
We have also $J^*S(J)=-S(f)$ and we would like to express $S(f)$ in terms of the Liouville field $\phi$.
Remark that 
\[ |\pl_z J|^2=e^{\phi}({\rm Im}(J(z)))^2 \]
thus 
\[ \pl_z \phi= \frac{\pl_z^2J}{\pl_zJ}+i\frac{\pl_zJ}{{\rm Im}J}.\]
Now we compute 
\begin{equation}\label{SJ} 
\pl_z^2\phi -\demi (\pl_z\phi)^2= S(J)+i\frac{\pl_z^2J}{{\rm Im}J}-\demi \frac{(\pl_zJ)^2}{({\rm Im}J)^2}-
\demi\Big(i\frac{\pl_zJ}{{\rm Im}J}\Big)^2-i\frac{\pl_z^2J}{\pl_zJ}.\frac{\pl_zJ}{{\rm Im}J}=S(J)
\end{equation}
and thus $\Theta_{h_+}(h_-)=-J^*((\pl_z^2\phi -\demi (\pl_z\phi)^2)dz^2)$. Next we can use Epstein description of the equidistant foliation in \cite{Ep}, combined with Theorem \ref{variationVolR}, which show that 
$\pl V_{h_-}(h_+)= \Theta_{h_+}$: indeed we lift the quasi-Fuchsian hyperbolic metric to $\hh^3$ 
(in the half-space model $\hh^3=(0,\infty)_x\x\cc_z$), the geodesic boundary defining function $\rho$ associated to 
$h_+$ and the equidistant foliation given by the level sets $\{\rho={\rm const}\}$ also lifts to $\hh^3$, the lift of the boundary metric $h_+$ is given by $e^{\phi}dz^2$ on the domain of discontinuity $\Omega_+\subset \cc=\pl \hh^3$, 
and \cite[formula (5.5)]{Ep} gives near $\Omega_+$ as $\rho\to 0$
\[ g_{\hh^3}= \frac{d\rho^2 + e^{\phi}|dz|^2 +({\rm Re}((\pl_z^2\phi -\demi (\pl_z\phi)^2)dz^2)+\pl_z\pl_{\bar{z}}\phi |dz|^2)\rho^2+\mc{O}(\rho^4)}{\rho^2}.\]
This implies that $d V_{h_{-}}(h_{+})=\frac{1}{4} \Re(\Theta_{h_+}(h_-))$ and thus we obtain 
\[\pl V_{h_-}(h_+)=\frac{1}{4}\Theta_{h_+}(h_-)\] 
by using \eqref{SJ}. The same holds by reversing the r\^ole of $h_-$ and $h_+$.
On the other hand, if $\Theta(h_-,h_+):=\Theta_{h_+}(h_-)$, one has for any section $Y_\pm$  of $T\mc{T}(M)$ (here $Y_\pm$ depends only on the $h_\pm$ variable),
\begin{equation}
\begin{split} 
\Re\cjg\mc{L}_{Y_-}\Theta_{h_+},Y_+\cjd =&\Re (\mc{L}_{(Y_-,0)}\Theta)_{(h_-,h_+)}(0,Y_+)=  4\mc{L}_{(Y_-,0)}d{\rm Vol}_R(h_-,h_+).(0,Y_+)   \\
=& 4\nabla^2 {\rm Vol}_R(h_-,h_+).((Y_-,0),(0,Y_+))  
\\
  =& 4\mc{L}_{Y_+}(d{\rm Vol}_R(h_-,h_+).(Y_-,0)) \\
  =& \Re\cjg\mc{L}_{Y_+}\Theta_{h_-},Y_-\cjd. 
  \end{split}
  \label{qfr.1}\end{equation} 
Since $\Theta_{h_{\pm}}$ is a family of holomorphic differentials on $\mc{T}(M)$ that depends holomorphically on $h_{\pm}$, we see that \eqref{qfr.1} in fact implies the quasi-Fuchsian reciprocity
\begin{equation}
  \cjg\mc{L}_{Y_-}\Theta_{h_+},Y_+\cjd  = \cjg\mc{L}_{Y_+}\Theta_{h_-},Y_-\cjd.
  \label{qfr.2}\end{equation}
Coming back to the renormalized volume, this finally yields
\begin{equation}
\begin{split} 
4\cjg \mc{L}_{Y_-}(\pl V_{h_-}(h_+)),Y_+\cjd =& \cjg\mc{L}_{Y_-}\Theta_+,Y_+\cjd  
= \cjg\mc{L}_{Y_+}\Theta_-,Y_-\cjd  \\
  =& 4\cjg \mc{L}_{Y_+}(\pl V_{h_+}(h_-)),Y_-\cjd \\
  =& 4\cjg dF_{(h_-,Y_-)}(h_+), Y_+ \cjd
  \end{split}
  \label{qfr.2}\end{equation} 
  as claimed.  
\end{proof}   
  
\section{Limit of the renormalized volume under the formation of a rank-1 cusp}

We consider an admissible 
degeneration of convex co-compact hyperbolic metrics $g_\eps$ on a 
manifold $X$ in the sense of Definition \ref{admisdeg}; $X$ is thus the interior of a smooth compact manifold $\bbar{{\bf X}}$ with boundary $N:=\pl\bbar{{\bf X}}$, with degenerating curve $H=\cup_{j=1}^{j_1} H_j\subset N$ 
and $\bbar{X}=\bbar{{\bf X}}\setminus H$. Recall that $\bbar{X}_c$ is the smooth manifold with corners obtained by blowing-up $H$ in $\bbar{{\bf X}}$, with boundary faces $\bbar{M}$ and $\cf$, see Section \ref{Sec:uniform}.
The goal of this Section is to show 
\begin{theo}\label{mainth}
Let $g_\eps$ be an admissible 
degeneration of convex co-compact hyperbolic metrics on $X$  in the sense of Definition \ref{admisdeg}, with limiting geometrically finite hyperbolic metric $g_0$. Then 
\[ \lim_{\eps\to 0}{\rm Vol}_R(X,g_\eps)={\rm Vol}_R(X,g_0).\] 
\end{theo}

\subsection{Limit far from the cusp}
First we describe the limit of the renormalized volume of the part far from the cusp, that in a fixed compact 
region $\mc{K}\subset \bbar{X}$. 
\begin{prop}\label{limoutsidecusp}
Let $\rho_\eps\in\mc{C}^\infty(\bbar{{\bf X}})$ be a geodesic boundary defining function such that $h_\eps:=
(\rho_\eps^2g_\eps)|_{N}$ is the unique hyperbolic metric in the conformal boundary 
($\rho_\eps$ is uniquely defined near $N$). 
Let $\rho_0\in \mc{C}^\infty(\bbar{X}_c)$ be a geodesic boundary defining function of $\bbar{M}$ of Proposition \ref{deffunccusp} 
with $h_0:=(\rho_0^2g_0)|_{TM}$ being the unique finite volume hyperbolic metric in the conformal boundary ($\rho_0$ is uniquely defined near $\bbar{M}$). 
Let $\theta_\eps$ be a family of smooth functions on  $\bbar{{\bf X}}$ vanishing in a uniform neighborhood of the degenerating curve $H$ and converging in all $\mc{C}^k$-norms to $\theta$. The following limit holds 
\[ \lim_{\eps\to 0}\Big({\rm FP}_{z=0} \int_{X}\theta_\eps \rho_\eps^{z}\, {\rm dvol}_{g_\eps}\Big)= 
{\rm FP}_{z=0} \int_{X}\theta \rho_0^{z}\, {\rm dvol}_{g_0}.\]
\end{prop}
\begin{proof} Let $\mc{K}$ be a compact neighborhood of $\supp \theta$. First, we can write 
${\rm dvol}_{g_\eps}=e^{G_\eps}{\rm dvol}_{g_0}$ for some smooth function $G_\eps$ converging to $0$ in $\mc{C}^\infty(\mc{K})$.
We use the notations of Section \ref{outsidecusp}:  the geodesic boundary defining function 
$\hat{\rho}_\eps$ in $\mc{K}$ is defined by \eqref{hamjab0}. Then we get 
\begin{equation}\label{intonK0}
\begin{gathered}
\int_{X}\theta_\eps\hat{\rho}_\eps^{z}\, {\rm dvol}_{g_\eps}-\int_{X}\theta\rho_0^{z}\, {\rm dvol}_{g_0}=
\int_{X}\rho_0^{z}(\theta_\eps e^{G_\eps+z\hat{\omega}_\eps}-\theta)\, {\rm dvol}_{g_0}
\end{gathered}\end{equation}
where $\hat{\rho}_\eps=e^{\hat{\omega}_\eps}\rho_0$, and $\hat{\omega}_\eps$ and $\theta_\eps-\theta$ converge to $0$ in 
$\mc{C}^\infty(\mc{K})$ by Lemma~\ref{lemhatomega}. Now the volume form of 
$g_0$ near $\rho_0=0$ is of the form $\rho_0^{-3}e^{H}d\rho_0 d\mu$ where $d\mu$ is a smooth measure on $\mc{K}\cap M$ 
and $H$ a smooth function of $\mc{K}$, thus writing
\[e^{z\hat{\omega}_\eps}= 1+ z\hat{\omega}_\eps+z^2 F_\eps
\]
for some smooth function $F_\eps$ on $\cc_z \x \mc{K}$ and using that for small $\delta>0$, 
$\int_{0}^\delta\rho^{z-1}d\rho_0$ has a pole of order $1$ at $z=0$ with residue $1$,
we directly obtain that 
\[ \begin{split}
{\rm FP}_{z=0}\int_{X}\rho_0^{z}(\theta_\eps e^{G_\eps+z\hat{\omega}_\eps}-\theta)\, {\rm dvol}_{g_0}=&
{\rm FP}_{z=0}\int_{X}\rho_0^{z}(\theta_\eps e^{G_\eps}-\theta)\, {\rm dvol}_{g_0}\\
& +\demi \int_{\mc{K}\cap M} \pl_{\rho_0}^2(\theta_\eps e^{G_\eps+H}\hat{\omega}_\eps)|_{\rho_0=0}\, d\mu
\end{split}\]
where $\pl_{\rho_0}$ is the vector field given by the gradient of $\rho_0$ with respect to $\rho^2g_0$.
Using that $G_\eps\to 0$, $\hat{\omega}_\eps\to 0$ and $\theta_\eps\to \theta$ in $\mc{C}^\infty(\mc{K})$, 
as $\eps\to 0$, we obtain 
that the finite part of \eqref{intonK0} at $z=0$
converges to $0$ as $\eps\to 0$. We write $h_\eps=e^{2\varphi_\eps}\hat{h}_\eps$.
To conclude, we may use Proposition \ref{vrv.1}, which of course also works in the convex co-compact case: 
that is for each $\eps>0$, we get with $\theta_\eps=\sum_{k=0}^2\theta_{\eps,k}\rho_\eps^k+\mc{O}(\rho_\eps^3)$ for some $\theta_{\eps,k}\in \mc{C}_0^\infty(M)$
\[\begin{split}
{\rm FP}_{z=0}\int_{X}\hat{\rho}_\eps^{z}\theta_{\eps} 
\, {\rm dvol}_{g_\eps}=& {\rm FP}_{z=0}\int_{X}\rho_\eps^{z}\theta_\eps\, {\rm dvol}_{g_\eps}\\
& -\frac{1}{4}\int_{\mc{K}\cap M}(\theta_{\eps,k} (|d\varphi_\eps|^2_{h_\eps}+{\rm Scal}_{h_\eps} \varphi_\eps)-
4\theta_{\eps,k} \varphi_\eps){\rm dvol}_{h_\eps}.
\end{split}\]
By assumption we have $\theta_{\eps,k}\to \theta_{k}$ with $\theta=\sum_{k=0}^2\theta_{k}\rho_0^k+\mc{O}(\rho_0^3)$.
Using Proposition \ref{cf.1} and Corollary \ref{cor:limitenergy} 
we directly obtain that  (recall that $\varphi_0=0$) 
\[\lim_{\eps\to 0}\int_{\mc{K}\cap M}(\theta_{\eps,0}(|d\varphi_\eps|^2_{h_\eps}+{\rm Scal}_{h_\eps} \varphi_\eps)-4\theta_{\eps,2} \varphi_\eps){\rm dvol}_{h_\eps}=0\]
which achieves the proof since \eqref{intonK0} has finite part at $z=0$ tending to $0$.
\end{proof} 

\subsection{Limit near the cusp}
We next study the behaviour of the renormalized volume in the regions $\mc{U}_j^\eps$ containing the degeneration. We notice that Theorem \ref{mainth} follows from Propositions \ref{limoutsidecusp} and 
the following 
\begin{prop}\label{nearthecusp}
With the notations and assumptions of Proposition \ref{limoutsidecusp} and Theorem \ref{mainth}, we have
\[
\lim_{\eps\to 0}\,  {\rm FP}_{z=0}\, \int_{X}(1-\theta^\eps)\rho^z_\eps{\rm dvol}_{g_\eps}=
{\rm FP}_{z=0}\, \int_{X}(1-\theta)\rho^z_0{\rm dvol}_{g_0}.
\]
\end{prop}
\begin{proof} We can assume that $(1-\theta^\eps)$ is supported in $\cup_j\mc{U}_j^\eps$, we are reduced to a local analysis and we can use the model $\bbar{\mc{U}}_\ell$ with metric $g_\ell$ of Section \ref{hamilton}, where we have forgot the $\eps$ parameter and use rather $\ell$ with $\ell\to 0$, and $\nu=\nu(\ell)$ is converging to some limit $\nu_0$ 
as $\ell\to 0$.
First, an easy computation gives that the volume form of $g_\ell$ is given by 
\[{\rm dvol}_{g_\ell}=\frac{R^2dudvdw}{u^3}\]
where $R^2=u^2+v^2+\ell^2$. We need to prove that 
\begin{equation}\label{renormint} 
\lim_{\ell\to 0}{\rm FP}_{z=0}\int_{(u,v,w,\ell)\in \bbar{\mc{U}}_\ell} \rho_\ell^z \chi\frac{R^2dudvdw}{u^3}=
{\rm FP}_{z=0}\int\rho_0^z \chi\frac{R^2dudvdw}{u^3}
\end{equation}
where $\rho_\ell=\rho_\eps$ is the function solving \eqref{hamjab} with $e^{2\varphi_\ell}h_\ell$ being hyperbolic
if $h_\ell$ is given by \eqref{hell}, and $\chi\in \mc{C}^\infty_c(\bbar{\mc{U}}_\ell)$ is independent of $\ell$ and equal to $1$ near $u=v=0$. To study the renormalized integral \eqref{renormint} we decompose 
$\bbar{\mc{U}}_\ell$ in several regions, see Figure~\ref{ul.2}.  
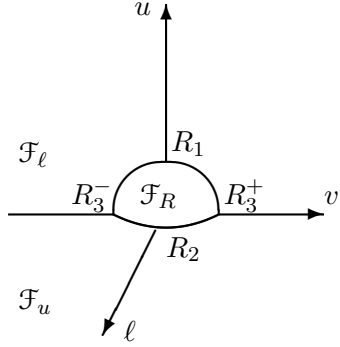
\begin{figure}
\setlength{\unitlength}{0.7cm}
\begin{picture}(6,6)
\thicklines
\put(3,2){\oval(2,2)[t]}
\put(2.5,2.2){$\mc{F}_{R}$}
\qbezier(2,2)(3,1.5)(4,2)

\put(2.4,5.8){$u$}

\put(4,2){\vector(1,0){2}}
\put(0,2){\line(1,0){2}}
\put(6,2.2){$v$}
\put(2.8,1.7){\vector(-1,-2){1}}
\put(0.2,0.2){$\mc{F}_u$}
\put(0.2,3){$\mc{F}_{\ell}$}
\put(3.1,3.2){$R_1$}
\put(4.1,2,2){$R_3^+$}
\put(1.2,2.2){$R_3^-$}
\put(3,1.2){$R_2$}
\put(3,3){\vector(0,1){3}}
\put(2.2,-0.4){$\ell$}

\end{picture}
\caption{The manifold with corners $\bbar{\mc{U}}_\ell$}\label{ul.2}
\end{figure}

We start with a region of finite volume (with the notations of Section \ref{hamilton})
$$
R_1(\ell)= \{ (u,v,w) \; | \;  u\leq \delta, \; -1\le V\le 1, \; 0< L\le 1 \},
$$
where we use the following coordinates,
\begin{equation}
  u, \quad V= \frac{v}u, \quad L=\frac{\ell}u,\quad  w.
\label{bdf.21}\end{equation}
In fact, for $\ell>0$ fixed, we have that
$$
   0\le L\le 1, \; 0\le u\le \delta \; \Longrightarrow \;   \ell\le u\le \delta.
$$
Take $\delta$ so that $\chi$ is supported in $\sqrt{u^2+v^2}\leq \delta$. 
In these coordinates, the volume form of $g_{\ell}$ is for $\ell$ fixed given by
$$
\Vol_{g_{\ell}}=\frac{(\ell^2+u^2+v^2)dudvdw}{u^3}= (1+V^2+L^2)dudVdw.
$$
Restricted to this region, the volume is thus clearly finite and there is no need to renormalize.  Thus,
\[
\begin{aligned}
{\rm FP}_{z=0}\int_{R_1(\ell)} \rho_\ell^z \chi\frac{R^2dudvdw}{u^3}&= \int_{-\frac{1}{4}}^{-\frac{1}{4}}
\int_{0}^1\int_{\ell}^\delta \chi(u,Vu,w)
\left( 1+V^2+ \frac{\ell^2}{u^2} \right)dudVdw
\end{aligned}
\]
We can use dominated convergence (using $L^2\indic_{[\ell,\delta]}(u)\leq \indic_{[0,\delta]}$) to deduce 
that 
\begin{equation}\label{bdf.22}
\begin{split}
\lim_{\ell\to 0} {\rm FP}_{z=0}\int_{(u,V,L,w)\in R_1} \rho_\ell^z \chi\frac{R^2dudvdw}{u^3}=&
\int_{-\frac{1}{4}}^{-\frac{1}{4}}
\int_{0}^1\int_{0}^\delta \chi(u,Vu,w)
\left( 1+V^2 \right)dudVdw\\
=&{\rm FP}_{z=0}\int_{R_1(0)} \rho_0^z \chi\frac{(u^2+v^2)dudvdw}{u^3}.
\end{split}\end{equation}
Next we analyze the region  $R_2(\ell)$ near the intersection $\mc{F}_u\cap \mc{F}_R$ but away from the corners $\mc{F}_R\cap\mc{F}_u\cap\mc{F}_\ell$. In this region, we can use the coordinates
$$
   \ell, \quad \tU= \frac{u}{\ell}, \quad \til{V}= \frac{v}{\ell}, \quad w.
$$
In these coordinates, we can define more precisely the region $R_2(\ell)$ by
$$
 R_2(\ell)=\{ (u,v,w) \; | \; \; 0\le \tU \le 1, \; -1 \le \til{V} \le 1 \}
$$
In these coordinates, the volume form of $g_{\ell}$ is given (for $\ell$ fixed) by
$$
   \Vol_{g_{\ell}}= \frac{\ell (1+\tU^2+\til{V}^2)d\tU d\til{V}dw}{\tU^3}.
$$
Since $U:=\frac{u}{R}=\frac{\tU}{\sqrt{1+\tU^2+\til{V}^2}}$ and $\rho_\ell=e^{\omega_\ell}U$ with the notation of \eqref{hamjab}, we have
\[
\begin{aligned}
{\rm FP}_{z=0}\int_{R_2(\ell)} \chi \rho_\ell^z & \frac{R^2dudvdw}{u^3}= {\rm FP}_{z=0} \int_{-\frac{1}{4}}^{\frac{1}{4}}\int_{-1}^{1}\int_0^{1} \rho_\ell^{z}\chi\frac{\ell(1+\tU^2+\til{V}^2)d\tU d\til{V}dw}{\tU^3} \\
&= {\rm FP}_{z=0} \int_{-\frac{1}{4}}^{\frac{1}{4}}\int_{-1}^{1}\int_0^{1}\chi \frac{\ell \tU^z e^{z\omega_\ell}(1+\tU^2+\til{V}^2)d\tU d\til{V}dw }{(1+\tU^2+\til{V}^2)^{\frac{z}2}\tU^3} \\
&= A_1(\ell) + A_2(\ell) +A_3(\ell).
\end{aligned}
\]
with 
\[\begin{gathered}
A_1(\ell):={\rm FP}_{z=0} \int_{-\frac{1}{4}}^{\frac{1}{4}}\int_{-1}^{1}\int_0^{1} \chi\frac{\ell \tU^z (1+\tU^2+\til{V}^2)d\tU d\til{V}dw}{\tU^3} ,\\
A_2(\ell):= \res_{z=0} \int_{-\frac{1}{4}}^{\frac{1}{4}}\int_{-1}^{1}\int_0^{1}\chi\frac{\ell \tU^z \omega_\ell(1+\tU^2+\til{V}^2)d\tU d\til{V}dw}{\tU^3} , \\
A_3(\ell):= -\frac12 \res_{z=0} \int_{-\frac{1}{4}}^{\frac{1}{4}}\int_{-1}^{1}\int_0^{1}\chi\frac{\ell \tU^z \log(1+\tU^2+\til{V}^2)(1+\tU^2+\til{V}^2)d\tU d\til{V}dw}{\tU^3} .
\end{gathered}\]
For $j=0,1$, the function $\ell\chi(\ell\til{U},\ell\til{V},w)(1+\tU^2+\til{V}^2)(\log(1+\tU^2+\til{V}^2))^j$ converges to $0$ in $\mc{C}^k$-norms for all $k$, and thus it is direct to see 
$\lim_{\ell\to 0} A_1(\ell)= \lim_{\ell\to 0}A_3(\ell)=0$. 
For the second term, we use the Taylor expansion of $\omega_\ell$ 
in terms of $\tU$ using \eqref{bdf.19} 
$$
     \omega_\ell= a_0 + a_2 U^2 + \mathcal{O}(U^3)= a_0 + \frac{a_2 \tU^2}{1+\til{V}^2}+ \mathcal{O}(\tU^3).
$$ 
Thus, we compute that
\begin{equation}
\begin{aligned}
A_2(\ell) &=\ell\int_{-\frac{1}{4}}^{\frac{1}{4}}\int_{-1}^{1} \left( (a_0+a_2)(\chi(0,\ell\til{V},w)+
a_0\ell^2(1+\til{V}^2)\pl_u^2\chi(0,\ell\til{V},w) \right)dVdw \\
&=\ell\int_{-\frac{1}{4}}^{\frac{1}{4}}\int_{-1}^{1}\chi(0,\ell\til{V},w) 
(a_0+a_2)dVdw+\mc{O}(\ell^3). \\
&= \int_{-\frac{1}{4}}^{\frac{1}{4}}\int_{-\ell}^{\ell} \chi(0,v,w) 
\left( \varphi_{\ell}-\frac14 |d\varphi_{\ell}|^2_{h_{\ell}}+ \frac{C_1\ell^2+C_2v\pl_w\varphi_\ell}{(\ell^2+v^2)} 
+C_3v\pl_v\varphi_\ell +\frac12 \right)dvdw.
\end{aligned}
\end{equation}
where $C_j$ are constant depending smoothly on $\nu$, and we used that $a_0=\varphi_\ell$ is uniformly bounded in $\ell$ in the second line.
From Proposition~\ref{cf.1} and Corollary~\ref{cor:limitenergy}, we see that
\[\int_{-\frac{1}{4}}^{\frac{1}{4}}\int_{-\ell}^{\ell} \chi(0,v,w) 
\Big( \varphi_{\ell}-\frac14 |d\varphi_{\ell}|^2_{h_{\ell}}+ \frac{C_1\ell^2}{(\ell^2+v^2)} \Big)dvdw\to 0.\]
Using Cauchy-Schwartz and 
$|d\varphi_\ell|_{h_\ell}\geq C(|(v^2+\ell^2)^{-\demi}\pl_w\varphi_\ell|+|v\pl_v\varphi_\ell|)$ we also get that
\[\begin{gathered}
\int_{-\frac{1}{4}}^{\frac{1}{4}}\int_{-\ell}^{\ell} |\chi(0,v,w)| 
\left( \frac{C_2|v\pl_w\varphi_\ell|}{(\ell^2+v^2)} +C_3|v\pl_v\varphi_\ell| \right)dvdw\leq \\
C'\Big(\sqrt{\ell}||d\varphi_\ell||_{L^2}+
\Big(\int_{-\frac{1}{4}}^{\frac{1}{4}}\int_{-\ell}^{\ell}|d\varphi_\ell|_{h_\ell}^2dvdw\Big)^{\demi}\Big)
\end{gathered}\]
for some $C'$ independent of $\ell$, thus this converges to $0$ by Corollary~\ref{cor:limitenergy}, and we conclude that $\lim_{\ell\to 0}A_2(\ell)=0$ and 
\[\lim_{\ell \to 0}{\rm FP}_{z=0}\int_{R_2(\ell)} \chi \rho_\ell^z \frac{R^2dudvdw}{u^3}=0.\]

Next, we consider the coordinates, smooth near the corners $\mc{F}_R\cap \mc{F}_u\cap\mc{F}_\ell$
$$
  v, \hU= \frac{u}{|v|}, \hat{L}= \frac{\ell}{|v|}, w
$$
and taking region $R_3(\ell)\cup R_4(\ell)$ given by 
$$
  \begin{gathered}
   R_3(\ell)= \{ (u,v,w)\; | \;  |v|\leq \delta, \, \hat{L} \leq 1 , \hU \le 1\}.
  \end{gathered} 
$$
we see that $\chi$ can written as $\sum_{j=1}^4\chi \indic_{R_j(\ell)}.$
In these coordinates, the volume form of $g_{\ell}$ is given for fixed $\ell$ by
$$
\Vol_{g_{\ell}}= \frac{(1+\hat{L}^2+\hU^2)d\hU dvdw}{\hU^3}.
$$
 Thus, since $U= \frac{\hU}{\sqrt{1+\hat{L}^2+\hU^2}}$, we have 
\begin{equation}\label{FPR3}
\begin{aligned}
 {\rm FP}_{z=0}\int_{R_3(\ell)} \chi \rho_\ell^z \frac{R^2dudvdw}{u^3} &= 
 {\rm FP}_{z=0}\int_{-\frac{1}{4}}^{\frac{1}{4}}\int_{\ell\leq |v|\leq \delta}\int_{0}^{1} \chi \frac{\hat{U}^ze^{z\omega_\ell}(1+ \frac{\ell^2}{v^2}+\hU^2)d\hU dvdw}{(1+\frac{\ell^2}{v^2}+\hat{U}^2)^{z/2}\hU^3} \\
 &=  I_1(\ell)+ I_2(\ell)+ I_3(\ell).\end{aligned}\end{equation} 
 with 
 \[\begin{gathered}
I_1(\ell):= {\rm FP}_{z=0}\int_{-\frac{1}{4}}^{\frac{1}{4}}\int_{\ell\leq |v|\leq \delta}\int_{0}^{1} \chi\frac{\hU^z(1+ \frac{\ell^2}{v^2}+\hU^2)d\hU dvdw}{\hU^3}, \\
I_2(\ell):=  \res_{z=0}\int_{-\frac{1}{4}}^{\frac{1}{4}}\int_{\ell\leq |v|\leq \delta} \int_{0}^{1}  \chi \frac{\hU^z \omega_\ell(1+ \frac{\ell^2}{v^2}+\hU^2)d\hU dvdw}{\hU^3}, \\
 I_3(\ell):= -\frac12  \res_{z=0}\int_{-\frac{1}{4}}^{\frac{1}{4}}\int_{\ell\leq |v|\leq \delta} \int_{0}^{1} \chi \log\left( 1+\hU^2+ \frac{\ell^2}{v^2}  \right) \frac{\hU^z (1+ \frac{\ell^2}{v^2}+\hU^2)d\hU dvdw}{\hU^3}.   \end{gathered}
\]
We notice that, in view of the smoothness of $\omega_\ell$ as a function of $U,v,w$, these three terms 
also make sense for $\ell=0$, and \eqref{FPR3} for $\ell=0$ is given by $\sum_{j=1}^3I_j(0)$.
To conclude the proof, we want to prove that $I_j(\ell)\to I_j(0)$ as $\ell\to 0$ for $j=1,2,3$.
For the first term, we compute that
\begin{equation}
I_1(\ell)= \int_{-\frac{1}{4}}^{\frac{1}{4}}\int_{\ell\leq |v|\leq \delta}((1+\frac{\ell^2}{v^2})q_1(v,w)+q_2(v,w))
 dvdw,
\end{equation}
where $q_1$ and $q_2$ are smooth and independent of $\ell$, and it is then clear that 
$$
  \lim_{\ell\to 0} I_1(\ell)=I_1(0)
$$
To deal with $I_3(\ell)$, we can proceed similarly: we remark that for $\ell\geq 0$, the integrand in $I_3(\ell)$ is of the form 
$\hat{U}^{z-3}Q(\hat{U},\frac{\ell^2}{v^2},v,w)$ where $Q$ is some smooth function of its parameters, thus it is straightforward to see that 
\[ I_3(\ell)=\int_{-\frac{1}{4}}^{\frac{1}{4}}\int_{\ell\leq |v|\leq \delta} q_3(v,\tfrac{\ell^2}{v^2},w)dvdw \]
for some smooth function $q_3$ of its parameters. We conclude as in the case of  $I_1$ that 
$$
  \lim_{\ell\to 0} I_3(\ell)=I_3(0).
$$
Finally we study $I_2(\ell)$. From the expansion \eqref{bdf.19}, we have for $\ell\geq 0$ that
$$
  \omega_\ell= a_0 +a_2U^2 + \mathcal{O}(U^3)= a_0 + \frac{a_2 \hU^2}{1+\frac{\ell^2}{v^2}}+ \mathcal{O}(\hU^3).
$$
Hence, we compute that for $\ell\geq 0$
\[
\begin{aligned}
   I_2(\ell) =& \int_{-\frac{1}{4}}^{\frac{1}{4}}\int_{\ell\leq |v|\leq \delta}\left( (a_0+a_2)(\chi(0,v,w)+
a_0(v^2+\ell^2)\pl_u^2\chi(0,v,w) \right)dv dw \\
   =&\int_{-\frac{1}{4}}^{\frac{1}{4}}\int_{\ell\leq |v|\leq \delta}  \chi(0,v,w) 
\left( \varphi_{\ell}-\frac14 |d\varphi_{\ell}|^2_{h_{\ell}}+ \frac{C_1\ell^2+C_2v\pl_w\varphi_\ell}{(\ell^2+v^2)} 
+C_3v\pl_v\varphi_\ell +\frac12 \right) dv dw \\
 &+ \int_{-\frac{1}{4}}^{\frac{1}{4}}\int_{\ell\leq |v|\leq \delta}  \varphi_\ell (v^2+\ell^2)\pl_u^2\chi(0,v,w)dvdw.
   \end{aligned}
\]
for some constant $C_j$ depending smoothly on $\nu$. By Proposition~\ref{cf.1}, the last line is continuous at $\ell= 0$, and using Corollary~\ref{cf.3} with the stronger estimate \eqref{cf.3a}, it is direct to check (like we did for the term $A_2(\ell)$) that $I_2(\ell)$ is continuous at $\ell=0$, ie. $\lim_{\ell \to 0}I_2(\ell)=I_2(0)$. We have finished the proof.
\end{proof}

\section{Appendix}

\textsl{Proof of Proposition \ref{model2}}.
We will construct $\Phi_L$ in two steps, as a composition 
$\Phi_L=\Xi_L\circ \Upsilon_L$. Let us first construct the diffeomorphism 
$\Upsilon_L$, which is done by changing coordinates on $X_{m(q)}$. 

Let $r=\sqrt{x^2+|z|^2}$ be the Euclidean radial coordinate in $\hh^3=\rr^+_x\x \hh^2_z$,
 then the hyperbolic metric takes the form in the Euclidean radial coordinates
$(r,\omega)$ with $\omega\in \mathbb{S}^2$
\[ g_{\hh^3}= \frac{dr^2+r^2 g_{\mathbb{S}^2}}{r^2\omega^2_x}\]
where $r\omega_x=x$ and $\omega_x=x(\omega)$ is the vertical coordinate on the sphere. We denote by 
$\omega_1={\rm Re}(z(\omega))$ and $\omega_2={\rm Im}(z(\omega))$ the coordinates of $\omega$ in the horizontal direction 
$z$. Consider the stereographic projection $\mathbb{S}^2\to \rr^2$ from the point 
$(x,z)=(0,-1)\in \mathbb{S}^2\subset \rr^3$, providing coordinates $\hat{u},\hat{v}\in \rr^2$ 
so that 
\[ \hat{u}=\frac{\omega_x}{\omega_1+1}, \quad \hat{v}=\frac{\omega_2}{\omega_1+1},
\textrm{ and the metric } g_{\mathbb{S}^2}=\frac{4(d\hat{u}^2+d\hat{v}^2)}{(1+\hat{u}^2+\hat{v}^2)^2}.\]
In the coordinates $(r,\hat{u},\hat{v})\in \rr^+\x \rr^+\x \rr$, the hyperbolic metric takes the form
\[ g_{\hh^3}= \frac{(1+\hat{u}^2+\hat{v}^2)^2dr^2}{4\hat{u}^2 r^2} + \frac{d\hat{u}^2 + d\hat{v}^2}{\hat{u}^2}.\]
Notice that $\hat{v}+i\hat{u}$ define coordinates on the hyperbolic plane $\hh^2$ 
(viewed as the upper half-space in $\cc$), and the stereographic projection is an isometry 
from the half-sphere $H(0,1)$ equipped with the metric induced from $\hh^3$ to this hyperbolic 
plane. The action $z\mapsto qz = e^{\ell(1+i\nu)}z$ in $\cc$ 
corresponds in $\hh^3$ to a dilation by $e^{\ell}$ centered at $(x,z)=(0,0)$ 
followed by a hyperbolic rotation $R_{\hh^3}(\nu \ell,x)$ 
of angle $\nu\ell$ around the $x$ axis in $\hh^3=\rr^+_x\x \cc_z$.  The latter is an elliptic isometry for $g_{\hh^3}$ and so, its restriction to $H(0,1)$ 
becomes an elliptic isometry of the  hyperbolic half-plane $\hh^2$ with coordinate $z=\hat{v}+i\hat{u}$, 
fixing the point $z=i$, and considering the derivative at this point shows that $R_{\hh^3}(\nu\ell,x)|_{H(0,1)}$, 
viewed in the variable $z=\hat{v}+i\hat{u}\in \hh^2$ via the stereographic projection, acts 
as the hyperbolic rotation of angle $\nu\ell$ and center $z=i\in\hh^2$. 
We denote by 
\[R_{\nu\ell}=\left(\begin{array}{cc} \cos \frac{\nu\ell}2 & 
\sin \frac{\nu\ell}2 \\
-\sin\frac{\nu\ell}2 & \cos\frac{\nu\ell}2\end{array}\right)\in {\rm PSL}_2(\rr)\]
this hyperbolic rotation.

In the quotient \eqref{quotientLq}, the fundamental domain is 
$e^{-\demi\ell}\leq r\leq e^{\demi\ell}$ so to have coordinates with uniform 
behavior with respect to the deformation parameters $\ell$, we introduce 
the rescaled coordinates
\[u'= \ell\hat{u}, \quad v'=\ell\hat{v}, \quad w= \frac{\log r}{2\ell}.\]  
We denote by $\Upsilon_L: (x,z)\mapsto (w,v'+iu')$ the diffeomorphism corresponding to the change of coordinates.
In these coordinates, the hyperbolic metric on $e^{-\demi\ell}\leq r\leq e^{\demi\ell}$ takes the form:
\[(\Upsilon_L)_*g_{\hh^3}= \frac{  du'^2 + dv'^2 + (\ell^2+ u'^2+ v'^2)^2 dw^2}{u'^2},\]
where $w\in [-\frac{1}{4},\frac{1}{4}]$. Moreover the transformation $\gamma_L$ becomes 
in these coordinates 
\[ (w,v'+iu')\mapsto (w+\tfrac{1}{2}, \ell R_{-\nu\ell}(\ell^{-1} (v'+iu'))).\]

The intersection of the half-sphere $\pl B(e(L),\rho(L))$ of \eqref{quotientLq} 
with the half-sphere $H(0,e^{2\ell w})$ (with $|w|<1/4$) is the half-circle 
obtained by intersecting the plane 
\[{\rm Re} (z)= \kappa(w,\ell,\delta)= 
\frac{e(L)^2+ e^{4\ell w}-\rho(L)^2}{2e(L)}\]
with $H(0,e^{2\ell w})$. Under the stereographic projection 
$H(0,e^{2\ell w})\to \{(x,z); {\rm Re}(z)=0\}=\rr^2$ from the point 
$(x,z)=(0,-e^{2\ell w})$, a small computation shows that it is thus sent to the 
half circle centered at $0$ of radius
\begin{align*}e^{2\ell w}\sqrt{\frac{e^{2\ell w}+\kappa(w,\ell(L),\delta)}
{e^{2\ell w}-\kappa(w,\ell,\delta)} }= 
\frac{r_{\la}(w)}{\ell}+\mc{O}_{\delta}(1),\quad 
r_{\la}(w):=\left(\frac{\lambda^2}{2\delta^2}-4w^2\right)^{-\frac{1}{2}}\end{align*}
where we have used \eqref{ckrhok} in the last equality. Consequently, the intersection 
of the half-ball $B(e(L),\rho(L))$ of \eqref{quotientLq} with the half-sphere 
$H(0,e^{2\ell w})$ (with $|w|<1/4$) 
becomes, in the coordinates $\zeta'=v'+iu'\in\hh^2$, a half-disc of the form
\begin{equation}\label{halfcircle}
  {\rm Im}(\zeta')>0,\quad |\zeta'|\le \ell 
\sqrt{\frac{(e^{2\ell w}-\kappa(w,\ell,\delta))e^{2\ell w}}
{e^{2\ell w}+ \kappa(w,\ell,\delta)}  }=r_{\la}(w) + \mc{O}_{\delta}(\ell).
  \end{equation}
and thus, taking $\delta$ small enough (independent of $\ell$) so that $\la/\delta-4>1$ 
this set is asymptotic to the half-disk
\begin{equation}\label{rq0w} 
\{\zeta'\in\cc;{\rm Im}(\zeta')>0, |\zeta'|\le r_{\la}(w)\}.
\end{equation}
We have thus showed the following 
\begin{lem}\label{model1}
There is an isometry $\Upsilon_L$ between 
$\cjg \gamma_L\cjd \backslash \hh^3$ and 
\begin{equation}\label{Xq}  
X_{\gamma_{q}}:= \cjg \gamma_{q}\cjd\backslash \left(\rr_w\x \hh^2_{\zeta'=v'+iu'}, \,\, \frac{  du'^2 + dv'^2 + (\ell^2+ u'^2+ v'^2)^2 dw^2}{u'^2}\right),
\end{equation}
where $\gamma_{q}$ is the map  
\[\gamma_{q}: (w, \zeta')\mapsto \left(w+\tfrac{1}{2},  
\frac{ \cos(\nu\ell/2)\zeta' +\ell \sin(\nu\ell/2)}
{-\ell^{-1}\sin(\nu\ell/2)\zeta' + \cos(\nu\ell/2)}
\right).\]
Moreover, if $\delta>0$ is small enough, the model neighborhood 
\eqref{quotientLq} is mapped via $\Upsilon_L$  to  
\begin{equation}\label{subsetofX}
\pi_{\gamma_{q}}\Big(\{ (w,\zeta')\in [-\tfrac{1}{4},\tfrac{1}{4})\x \hh^2; |\zeta'|<r_{q}(w)\}\Big)
\end{equation} 
where $\pi_{\gamma_q}: \rr\x\hh^2\to X_{\gamma_q}$ is the covering map, and $r_{q}(w)$ 
is the radius of the half-circle given by equation \eqref{halfcircle} and converging to $r_{\la}(w)>0$ 
with $r_\la(w)=\mc{O}(\delta)$ uniformly in $|w|<1/4$.
\end{lem}

Notice that $\ell\to 0$, then $\gamma_{q}$ converges to some transformation $\gamma_{\nu}:(w, \zeta')\to (w+\tfrac{1}{2}, P_\nu(\zeta'))$ 
with $P_\nu\in {\rm PSL}_2(\rr)$ the parabolic transformation $\zeta'\mapsto \frac{2\zeta'}{\nu \zeta'+2}$, and $X_{\gamma_q}$ converges to 
\[ X_{\gamma_{\nu}}:=\cjg \gamma_{\nu}\cjd\backslash  \left(\rr_w\x \hh^2_{\zeta'=v'+iu'}, \,\, g_0=\frac{  du'^2 + dv'^2 + (u'^2+ v'^2)^2 dw^2}{u'^2}\right).\] 
Conjugating by an inversion $\zeta'\mapsto -1/\zeta'$ on $\hh^2$, $P_\nu$ becomes the transformation $\zeta'\mapsto \zeta'-\nu/2$
 and the transformation $\gamma_{\nu}$ viewed in the coordinates $(w,y+ix)$
defined by $y+ix=-1/(v+iu)$ is the parabolic isometry of  $\hh^3=\rr_w\x \hh^2_{y+ix}$ fixing $\infty$ and given by 
$T_\nu: (w,y+ix)\mapsto (w+\tfrac{1}{2},y-\tfrac{\nu}{2}+ix)$.
Then $X_{\gamma_\nu}$ is isometric to $\cjg T_\nu\cjd\backslash \hh^3$, which is the model of a hyperbolic cusp of rank $1$.  Clearly, the model of Lemma \ref{model1} extends smoothly to 
the parabolic boundary $\{\ell=0\}$ of $\bbar{\mc{Q}}$.

We also need to control  the change of coordinates from the neighborhood $\mc{U}^\delta_{L}$ of \eqref{ULqk} to this new model when $\ell\to 0$, that is we want to know 
$\Upsilon_L\circ \Theta_L$.
A direct computation gives  
\begin{equation}\label{r2} 
\begin{gathered}
r^2(\Theta_L(x,z))=\frac{x^2\la^2\ell^2+|x^2+|z|^2-z\la\ell|^2}{(x^2+|z-\la\ell|^2)^2},\\
\omega_x(\Theta_L(x,z))= \frac{x\la\ell}{\eta_L(x,z)},\quad 
\omega_1(\Theta_L(x,z))=\frac{-x^2-|z|^2+{\rm Re}(z)\la\ell}{\eta_L(x,z)}\\
\omega_2(\Theta_L(x,z))=\frac{{\rm Im}(z)\la\ell}{\eta_L(x,z)}
\end{gathered}\end{equation} 
with $\eta_L(x,z):=\sqrt{(x^2+{\rm Im}(z)^2)\la^2\ell^2+(x^2+|z|^2-{\rm Re}(z)\la\ell)^2}$, thus 
\begin{equation}\label{ucirctheta}
\begin{gathered}
u'(\Theta_L(x,z))=\frac{x}{\la(x^2+{\rm Im}(z)^2)}(\eta_L(x,z)+x^2+|z|^2-{\rm Re}(z)\la\ell),\\
v'(\Theta_L(x,z))=\frac{{\rm Im}(z)}{\la(x^2+{\rm Im}(z)^2)}(\eta_L(x,z)+x^2+|z|^2-{\rm Re}(z)\la\ell).
\end{gathered}\end{equation}
Notice that $\Upsilon_L\circ \Theta_L$ extends smoothly in a neighborhood of the cusp region of $\mc{X}$ of the form
\[\mc{V}^\delta:=\{(L,x,z)\in \bbar{\mc{Q}}\x \pi_{\gamma_L}(\bbar{B(0,\delta)}); (x,z)\in \bbar{\til{F}_L}\setminus \{0\}\}.\]
Indeed one has $w(\Theta_L(x,z))=\frac{\log(r\circ \Theta_L(x,z))}{2\ell}$, and by \eqref{r2} 
we can write it under the form $w(\Theta_L(x,z))=\frac{\log(1+\ell F(L,x,z))}{2\ell}$ 
for some $F(L,x,z)$ smooth in $\mc{V}^\delta$ and thus $w$ extends smoothly in $\mc{V}^\delta$. 
It is also easily checked that $(u',v')$ extend smoothly to $\mc{V}^\delta$ by \eqref{ucirctheta}. 
The inverse also admits a smooth extension to $\{\ell=0,(u',v')\not= (0,0)\} $ by a similar computation.

To finish the proof of the Proposition, we shall construct a diffeomorphism $\Xi_L$ corresponding to a new change of coordinates.
In the $\hh^3=\rr_w\x \hh^2_{\zeta'=v'+iu'}$ hyperbolic space, we define the function  
\[ \mu(w,\zeta') := d_{\hh^2}(\zeta'; i\ell)\]
which is invariant by the transformation $\gamma_{q}$. One has in particular 
\[\cosh(\mu)=\frac{u'^2+v'^2+\ell^2}{2u'\ell}.\]
Let us make the following change of coordinates on $[-1/4,1/4] \x \hh^2$, which defines  
$\Xi_L$,
\[ \Xi_L: (w,\zeta')\mapsto (w,\zeta:=\ell R_{-2\nu\ell w}(\ell^{-1}\zeta')) \]
where $R_{\theta}\in {\rm PSL}_2(\rr)$ is the hyperbolic rotation of angle $\theta$ and center $i$.
The transformation $\gamma_{q}$ becomes in the $(w,\zeta)$ coordinates (ie. after conjugation with $\Xi_L$) the transformation
\[ \Xi_L\circ \gamma_{q}\circ (\Xi_L)^{-1}: (w,\zeta)\mapsto (w+\tfrac{1}{2}, \zeta).\]
We see that $\Xi_L$ extends smoothly to $\{\ell=0; |\zeta'|<\delta\}$ if $\delta$ is small enough, with value 
\[ \Xi_{(0,\nu,\la)}(w,\zeta')= \frac{\zeta'}{\nu w\zeta'+1}\] 
and the same holds for its inverse. Thus we deduce that $\Phi_L:=\Xi_L\circ \Upsilon_L$ is such that 
$(L,x,z)\mapsto \Phi_L\circ \Theta_L$ extends smoothly to $\mc{V}^\delta$ if $\delta>0$ is chosen small enough.
We write $\zeta=v+iu\in \hh^2$, then the function $\cosh(\mu)$ is clearly invariant by rotation, so 
\begin{equation}\label{RR'}
\frac{u'^2+v'^2+\ell^2}{u'}=\frac{u^2+v^2+\ell^2}{u}
\end{equation}
and we compute
\begin{equation}\label{uu'}
\begin{gathered}
u'=\frac{u}{|-\ell^{-1}\sin(\nu\ell w)\zeta+\cos(\nu\ell w)|^2} , \quad 
d\zeta'=\frac{d\zeta +dw(\nu{\zeta}^2+\nu\ell^2)}{(-\ell^{-1}\sin(\nu\ell w)\zeta+\cos(\nu\ell w))^2}
\end{gathered}\end{equation}
Therefore the metric $g_L$ becomes in the new coordinates.
\[\begin{split} 
g_L:=(\Xi_L\circ \Upsilon_L)_*g_{\hh^3}= & \frac{du^2+dv^2+((1+\nu^2){R}^4-4\nu^2\ell^2{u}^2)dw^2}{u^2}\\
&+\frac{2\nu(R^2-2u^2)dwdv+4\nu uv dudw}{u^2} 
\end{split}\]
where  $R:=\sqrt{u^2+v^2+\ell^2}$.
Here, we notice that the change of coordinates $v'+iu'\mapsto v+iu$ for a fixed $w$ is a hyperbolic rotation of angle $-2\nu\ell w$ and center $i\ell$ in $\hh^2$. In particular it maps the half-circle \eqref{halfcircle} (which is a geodesic of $\hh^2$) 
to the half-circle in $\hh^2$ which intersects the real axis at the two points 
\[ v_{\pm}(q)= \frac{\pm r_{q}(w) \cos(\nu\ell w)+\ell\sin(\nu\ell w)}{\mp r_{q}(w) \ell^{-1}\sin(\nu\ell w)+\cos(\nu\ell w)}=\frac{\pm r_{q}(w)}{1\mp \nu wr_{q}(w)}+\mc{O}(\ell).\]
This shows that the region \eqref{subsetofX} in the coordinates $(w,\zeta)$ becomes the set
\[\{ (w,\zeta)\in [-\tfrac{1}{4},\tfrac{1}{4})\x \hh^2; |\zeta-v_{q}(w)|\leq \tau_{q}(w)\}/ \{w\sim w+\tfrac{1}{2}\}\]
for $v_{q}(w)=\demi(v_{+}(q)+v_-(q))$ and $\tau_{q}(w)=\demi(v_{+}(q)-v_-(q))$ which clearly converge as $\ell\to 0$, and satisfy the desired properties (recall that $r_{q}=r_\la(w)+o(1)$ as $\ell\to 0$ with the notation of \eqref{rq0w}).


\begin{thebibliography}{0}

\bibitem[Al]{Al} P.~Albin, \emph{Renormalizing curvature integrals on Poincar\'e-Einstein manifolds},
Adv.\ Math.\ \textbf{221} (2009), no.\ 1, 140--169.

\bibitem[AAR]{AAR} P.~Albin, C.~Aldana, F.~Rochon 
\emph{Ricci flow and the determinant of the Laplacian on non-compact surfaces},
Comm.\ Partial Diff.\ Eq.\ \textbf{38} (2013), 711--749.

\bibitem[Bo]{Bo}  B.~H.~Bowditch, \emph{Geometrical finiteness for hyperbolic groups}, 
J.\ Funct.\ Anal.\ \textbf{113} (1993), 245-317.

\bibitem[BrCa]{BrCa} M.~Bridgeman, R.~D.~Canary, \emph{Renomalized volume and the volume of the convex core}, 
arXiv 1502.05018. 

\bibitem[Br]{Br} J.~Brock, \emph{The Weil-Petersson metric and volumes of $3$-dimensional hype
rbolic convex cores}, J.\ Amer.\ Math.\ Soc.\ \textbf{16} (2003), no. 3, 495--535.

\bibitem[BrBr]{BrBr}
J.~Brock, K.~Bromberg,
\emph{Inflexibility, Weil-Petersson distance, and volumes of fibered $3$-manifolds}, 
arXiv 1412.0733.

\bibitem[Ch]{Ch} V.~Chuckrow, \emph{On Schottky groups with applications to kleinian groups}, 
Ann.\ Math.\ (2) \textbf{88} (1968), 47--61.

\bibitem[CiMo]{CiMo}  C.~Ciobotaru, S.~Moroianu, 
\emph{Positivity of the renormalized volume of almost-Fuchsian hyperbolic 3-manifolds}, 
preprint arXiv:1402.2555, to appear in Proc.\ AMS. 

\bibitem[Ep]{Ep} C.L. Esptein, \emph{Envelopes of horospheres andWeingarten surfaces in hyperbolic 3-space}. Preprint 1984, available at https://www.math.upenn.edu/$\sim$cle/papers/index.html.

\bibitem[FeGr]{FeGr} C.~Fefferman, C.~R.~Graham, \emph{The ambient metric}, 
Annals of Mathematics Studies {\bf 178}. Princeton University Press, Princeton, NJ, 2012.

\bibitem[Gr]{Gr} C.~R.~Graham, \emph{Volume and area renormalizations for conformally compact Einstein metrics},
Rend.\ Circ.\ Mat.\ Palermo Ser.\ II \textbf{63} (2000), Suppl., 31--42.

\bibitem[GuMa]{GuMa} C.~Guillarmou, R.~Mazzeo, 
\emph{Resolvent of the Laplacian on geometrically finite hyperbolic manifolds.} 
Invent.\ Math.\ \textbf{187} (2012), no.\ 1, 99--144.

\bibitem[GuMo]{GuMo} C.~Guillarmou, S.~Moroianu, \emph{Chern-Simons line bundle on Teichm\"uller space}, 
Geometry \& Topology, \textbf{18} (2014) 327--377.

\bibitem[GMS]{GMS} C.~Guillarmou, S.~Moroianu, J.-M.~Schlenker, 
\emph{The renormalized volume and uniformisation of conformal structures}
arXiv: 1211.6705.

\bibitem[HeSk]{HeSk} M.~Henningson, K.~Skenderis, \emph{The holographic Weyl anomaly}, JHEP 9807:023 (1998).

\bibitem[JMM]{JMM}  T.~Jorgensen, A.~Marden, B.~Maskit, \emph{The boundary of classical Schottky space,} 
Duke Math.\ J.\ \textbf{46} (1979), no.\ 2, 441--446. 

\bibitem[KoMc]{KoMc}
S.~Kojima, G.~McShane,
\emph{Normalized Entropy versus Volume}, arXiv:1411.6350.

\bibitem[KrSc]{KrSc}  K.~Krasnov, J.-M.~Schlenker, 
\emph{On the renormalized volume of hyperbolic 3-manifolds},
Comm.\ Math.\ Phys.\ \textbf{279} (2008), no.\ 3, 637--668.


\bibitem[Ma]{Ma} A.~Marden, \emph{Schottky groups and circles}, Contributions to analysis 
(a collection of papers dedicated to Lipman Bers), pp.\ 273--278. Academic Press, New York, 1974.

\bibitem[Ma2]{Ma2} 
A.~Marden, \emph{Deformation of Kleinian groups}, 
Chap.\ 9, Handbook of Teichm\"uller theory, Vol. I, ed.\ A.~Papadopoulos, 
IRMA Lect.\ Math.\ Theor.\ Phys.\ 11, Eur.\ Math.\ Soc.\ (2007).

\bibitem[MaMe]{MM} R.~R.~Mazzeo, R.~B.~Melrose,
{\em Meromorphic extension of the resolvent on complete spaces
with asymptotically constant negative curvature,\/}
J.\ Funct.\ Anal.\ \textbf{75}(1987), no.~2, 260--310.

\bibitem[MaPh]{MaPh} R.~Mazzeo, R.~Phillips, 
\emph{Hodge theory for hyperbolic manifolds}, Duke Math.\ J.\ 
\textbf{60} (1990), no.\ 2, 509--559. 

\bibitem[Me]{Me} R.~B.~Melrose, \emph{Differential analysis on manifolds with corners}, 
book in preparation, available at http://www-math.mit.edu/$\sim$rbm/book.html

\bibitem[MeZh]{MelZhu} R.~B.~Melrose, X. Zhu, \emph{Resolution of the canonical fiber metrics for a Lefschetz fibration}, 
arXiv 1501.04124.

\bibitem[Mo]{Mo} S.~Moroianu, \emph{Convexity of the renormalized volume of hyperbolic $3$-manifolds}, 
arXiv 1503.07981.

\bibitem[RoZh]{Rochon-Zhang} F.~Rochon, Z.~Zhang, 
\emph{Asymptotics of complete K\"ahler metrics of finite volume on quasiprojective manifolds}, Adv.\ Math.\ 
\textbf{235} (2012), 2892--2952. 

\bibitem[Sc]{Sc} J.-M.~Schlenker, \emph{The renormalized volume and the volume 
of the convex core of quasifuchsian manifolds,}
Math.\ Res.\ Let.\ \textbf{20} (2013), 773--786.

\bibitem[TaTe]{TaTe}  L.~A.~Takhtajan, L.-P.~Teo,
\emph{Liouville action and Weil-Petersson metric on deformation spaces, 
global Kleinian reciprocity and holography}, Comm.\ Math.\ Phys.\ \textbf{239} (2003) no.\ 1--2, 183--240.

\bibitem[PTT]{PTT} J. Park, L.~A.~Takhtajan, L.-P.~Teo, \emph{Potentials and Chern forms for Weil-Petersson and Takhtajan-Zograf metrics on moduli spaces}, arXiv:1508.02102. 
 
\bibitem[Va]{Va} F. Vargas Pallete, \emph{Local convexity of renormalized volume for rank-1 cusped manifolds},   arXiv:1505.00479. 
\end{thebibliography}
\end{document}